\documentclass[reqno, twoside]{amsart}
\usepackage[pdfstartview=FitH,
            pdfauthor={Antieau, Mathew, Morrow, and Nikolaus},
            pdftitle={On the Beilinson fiber square},
            colorlinks,
            linkcolor=reference,
            citecolor=citation,
            urlcolor=e-mail,
            ]{hyperref}
\newcommand{\crys}{\mathrm{crys}}
\usepackage{enumerate}
\usepackage{mathrsfs}
\newcommand{\LL}{\mathbb{L}}
\newcommand{\fil}{\mathrm{Fil}}
\newcommand{\qlrsp}{\mathrm{qrsPerfd}}

\newcommand{\qsyn}{\mathrm{QSyn}}
\usepackage{relsize} 
\usepackage[bbgreekl]{mathbbol}
\DeclareSymbolFontAlphabet{\mathbb}{AMSb} % to ensure \mathbb does not change
\DeclareSymbolFontAlphabet{\mathbbl}{bbold}

\newcommand{\Prism}{{ \widehat{\mathbbl{\Delta}}}}
\newcommand{\Prismnc}{{ {\mathbbl{\Delta}}}}

\usepackage{verbatim}
\usepackage[left=1.2in, right=1.2in, top=1.5in, bottom=1.5in]{geometry}
\usepackage{amsfonts}
\usepackage{amssymb}
\usepackage{amsmath}
\usepackage{cleveref}
\usepackage{amsthm}

\usepackage{color}
\definecolor{todo}{rgb}{1,0,0}
\definecolor{conditional}{rgb}{0,1,0}
\definecolor{e-mail}{rgb}{0,.40,.80}
\definecolor{reference}{rgb}{.20,.60,.22}
\definecolor{mrnumber}{rgb}{.80,.40,0}
\definecolor{citation}{rgb}{0,.40,.80}

\newcommand{\TP}{\mathrm{TP}}

\newcommand{\can}{\mathrm{can}}
\newcommand{\grcyc}{\mathrm{GrCycSp}}
\newcommand{\triv}{\mathrm{triv}}
\newcommand{\pfd}{\mathrm{Perfd}}

\newcommand{\qrsp}{\mathrm{QRSPerfd}}
\newcommand{\ainf}{\mathrm{A}_{\mathrm{inf}}}
\newcommand{\acrys}{\mathrm{A}_{\mathrm{crys}}}
\newcommand{\bcrys}{\mathrm{B}^+_{\mathrm{crys}}}
\newcommand{\bdr}{\mathrm{B}_{\mathrm{dR}}}

\newcommand{\op}{\mathrm{op}}
\newcommand{\Map}{\mathrm{Map}}
\newcommand{\Mod}{\mathrm{Mod}}

\DeclareMathOperator{\spec}{Spec}
\DeclareMathOperator{\Spec}{Spec}

\DeclareMathOperator{\spf}{Spf}

\newcommand{\cycsp}{\mathrm{CycSp}}

\newcommand{\THH}{\mathrm{THH}}

\newcommand{\tr}{\mathrm{tr}}

\newcommand{\alg}{\mathrm{Alg}}
\renewcommand{\hom}{\mathrm{Hom}}

\newcommand{\fun}{\mathrm{Fun}}

\usepackage{xy}
\input xy
\renewcommand{\sp}{\mathrm{Sp}}
\newcommand{\et}{\mathrm{\acute{e}t}}
\newcommand{\iso}{\cong}

\newcommand{\sub}[1]{{\mbox{\rm \scriptsize #1}}}

\xyoption{all}

\theoremstyle{definition}
\newtheorem{definition}{Definition}[section]

\newtheorem{cons}[definition]{Construction}
\newtheorem{construction}[definition]{Construction}
\newtheorem{example}[definition]{Example}
\newtheorem{remark}[definition]{Remark}
\theoremstyle{theorem}
\newtheorem{proposition}[definition]{Proposition}
\newtheorem{lemma}[definition]{Lemma}
\newtheorem{corollary}[definition]{Corollary}

\newtheorem{theorem}[definition]{Theorem}

\renewcommand{\L}{{L}}
\newcommand{\WW}{\mathbb{W}}

\newcommand{\FF}{\mathbb{F}}
\newcommand{\ZZ}{\mathbb{Z}}
\newcommand{\QQ}{\mathbb{Q}}
\newcommand{\Sp}{\mathrm{Sp}}
\newcommand{\we}{\simeq}
\renewcommand{\SS}{\mathbb{S}}
\newcommand{\Oscr}{\mathcal{O}}
\newcommand{\Nscr}{\mathcal{N}}
\newcommand{\K}{\mathrm{K}}

\newcommand{\R}{\mathrm{R}}
\newcommand{\dR}{\mathrm{dR}}
\newcommand{\HP}{\mathrm{HP}}
\newcommand{\HC}{\mathrm{HC}}
\newcommand{\Cscr}{\mathcal{C}}
\renewcommand{\phi}{\varphi}
\newcommand{\Dscr}{\mathcal{D}}

\newcommand{\Fun}{{\mathrm{Fun}}}

\newcommand{\TC}{\mathrm{TC}}
\newcommand{\HH}{\mathrm{HH}}

\newcommand{\act}{\mathrm{act}}
\newcommand{\xto}{\xrightarrow}
\newtheoremstyle{named}{}{}{\itshape}{}{\bfseries}{.}{.5em}{#1 \thmnote{#3}}
\theoremstyle{named}
\newtheorem*{namedtheorem}{Theorem}
\newtheorem*{namedcorollary}{Corollary}

\theoremstyle{definition}
\newtheorem{conjecture}[definition]{Conjecture}
\newtheorem{question}[definition]{Question}

\keywords{$p$-adic $K$-theory, cyclic homology, deformation  of algebraic cycles, motivic cohomology}
\subjclass[2010]{14F30, 14F40, 19D55, 19E15}

\begin{document}

\title{On the Beilinson fiber square}

\author[B. Antieau]{Benjamin Antieau}
\address{Department of Mathematics, Northwestern University}
\email{antieau@northwestern.edu}

\author[A. Mathew]{Akhil Mathew}
\address{Department of Mathematics, University of Chicago}
\email{amathew@math.uchicago.edu}

\author[M. Morrow]{Matthew Morrow}
\address{CNRS, Institut de Math\'ematiques de Jussieu--Paris Rive Gauche, Sorbonne Universit\'e}
\email{matthew.morrow@imj-prg.fr}

\author[T. Nikolaus]{Thomas Nikolaus}
\address{FB Mathematik und Informatik, Universit\"at M\"unster }
\email{nikolaus@uni-muenster.de}

\date{\today}

\maketitle

\begin{abstract}
Using topological cyclic homology, we give a  refinement of Beilinson's
$p$-adic Goodwillie isomorphism between relative continuous $\K$-theory and cyclic homology. 
As a result, we generalize results of Bloch--Esnault--Kerz and Beilinson
on the $p$-adic deformations of $\K$-theory classes. 
Furthermore, we prove structural results for the Bhatt--Morrow--Scholze filtration on
$\mathrm{TC}$ and identify the graded pieces with the syntomic cohomology of Fontaine--Messing.
\end{abstract}

\tableofcontents

\section{Introduction}

\subsection{Fiber squares}
For any ring $R$, one has its connective algebraic $\K$-theory $\K(R)$ and its negative cyclic homology $\HC^-(R)$; they are related via
the Goodwillie--Jones trace map  $\mathrm{tr}_\sub{GJ}\colon \K(R) \to \HC^-(R)$, often
interpreted and referred to as a Chern character, \cite[Ch.~8]{Loday}.
Moreover, when $R$ is a $\mathbb{Q}$-algebra, the 
map $\mathrm{tr}_{\sub{GJ}}$ induces an isomorphism on relative theories for
nilimmersions, via the following theorem of Goodwillie.

\begin{theorem}[Goodwillie~\cite{G86}] 
If $I \subseteq R$ is a nilpotent ideal in an associative $\mathbb{Q}$-algebra $R$, then the
commutative square
\[ \xymatrix{
    \K(R)\ar[r]\ar[d]^{\tr_{\sub{GJ}}}&\K(R/I)\ar[d]^{\tr_{\sub{GJ}}}\\
    \HC^-(R)\ar[r]&\HC^-(R/I)
}\]
is cartesian, i.e., the Goodwillie--Jones trace map induces an equivalence
$\tr_{\sub{GJ}}\colon\K(R,I)\we\HC^-(R,I)$ on relative theories.
\end{theorem}

Here for a pair $(R, I)$ with $I \subseteq R$ an ideal, we write $\K(R, I)$
for the
fiber of $\K(R) \to\K(R/I)$, and similarly for other functors such as $\HC^-$ and so on.

In order to extend Goodwillie's theorem to more general rings, 
one uses topological cyclic homology $\mathrm{TC}(R)$,
introduced in \cite{BHM} in the $p$-complete case and in \cite{DGM} integrally,
and the cyclotomic trace $\mathrm{tr}\colon\K(R) \to \TC(R)$, which
refines the Goodwillie--Jones trace map.

\begin{theorem}[Dundas--Goodwillie--McCarthy~\cite{DGM}]\label{thm:dgm}
    If $I \subseteq R$ is a nilpotent ideal in an associative ring $R$,
    then the commutative square
    \[ \xymatrix{
        \K(R)\ar[r]\ar[d]^{\tr}&\K(R/I)\ar[d]^{\tr}\\
        \TC(R)\ar[r]&\TC(R/I)
    }\] is cartesian, i.e., the cyclotomic trace induces an equivalence
    $\mathrm{tr} \colon \K(R,I) \we \TC(R,I)$ on relative theories.
\end{theorem} 

Topological cyclic homology is thus the primary tool in  
calculations of relative $\K$-theory (see for example \cite{Madsen94, HM, HM03}), but it is a significantly 
more complicated invariant than cyclic homology. 
However, recently Beilinson \cite{Beilinson}
gave a version of Goodwillie's original result 
in a $p$-adic setting, when the ideal in question is $(p)$ and $R$ is assumed
to be complete along $(p)$. 
The first goal of this paper is to construct a variant of the Chern character
and prove a  strengthening of Beilinson's theorem. 

Throughout this paper, we fix a prime number $p$. 
We use the convention that 
the modifier ``$\mathbb{Z}_p$'' refers to $p$-adic completion of an object, and
``$\mathbb{Q}_p$'' to
the rationalization of the $p$-completion; for example $\K(R; \mathbb{Z}_p)$
denotes the $p$-complete $\K$-theory of $R$, and $\K(R; \mathbb{Q}_p)$ denotes the
rationalization of $\K(R;\ZZ_p)$. 
Similarly, the modifier ``$\mathbb{Q}$'' refers to rationalization. 
We denote by  $\HP$
(resp.~$\HC$)  periodic cyclic (resp.~cyclic) homology. 

\begin{namedtheorem}[A]\label{thm:a}
For an associative ring $R$, there is a natural  
$p$-adic Chern character map 
\begin{equation} \label{cchern} \mathrm{tr}_{\sub{crys}}\colon \K(R/p; \mathbb{Q}_p) \to \HP(R; \mathbb{Q}_p)
\end{equation} 
which fits into a natural commutative square
\begin{equation}\label{Bpullback}\begin{gathered} \xymatrix{
    \K(R;\QQ_p)\ar[r]\ar[d]^{\tr_{\sub{GJ}}}&\K(R/p;\QQ_p)\ar[d]^{\tr_\sub{crys}}\\
    \HC^-(R;\QQ_p)\ar[r]&\HP(R;\QQ_p).}
\end{gathered}\end{equation}
If $R$ is commutative and henselian along $(p)$ then this square is cartesian, thereby giving an equivalence $\K(R, (p); \mathbb{Q}_p) \simeq \Sigma \HC(R; \mathbb{Q}_p)$. 
\end{namedtheorem} 

In \cite{Beilinson}, 
Beilinson constructs a natural equivalence\footnote{Since the methods are different, we do not know if
our identification on fiber terms
 is the same 
as Beilinson's.}
$\K^{\mathrm{cts}}(R, (p); \mathbb{Q}_p) \simeq \Sigma \HC(R; \mathbb{Q}_p)$
under the assumption that $R$ is $p$-complete with bounded
$p$-power torsion,   $R/p$
has finite stable range,\footnote{The stable range of a ring $R$ was defined
    in~\cite{bass-stable} (see also~\cite[V.3]{bass}) and is sometimes, as
    in~\cite{Beilinson}, called the stable rank.}
    and the relative $\K$-theory term $\K(R, (p))$ 
is replaced by the ``continuous'' relative $\K$-theory
$\K^{\mathrm{cts}}(R, (p)) = \varprojlim\K(R/p^n,
(p))$; this replacement does not affect the conclusion if $R$ is 
commutative 
thanks to \cite[Theorem 5.23]{CMM}. 
Beilinson's arguments  rely on some $p$-adic Lie theory.

In this
paper, we will construct the map 
\eqref{cchern} 
using the description of topological
cyclic homology from Nikolaus--Scholze \cite{NS18}, as a consequence of
B\"okstedt's calculation of $\THH(\mathbb{F}_p)$. 
Together with  the rigidity results of Clausen--Mathew--Morrow \cite{CMM}, 
we explain a short, homotopy-theoretic proof of Theorem~\hyperref[thm:a]{A}. In
fact, Theorem~\hyperref[thm:a]{A} and all the corollaries listed below hold for
any (possibly non-commutative) ring $R$ if we replace $\K$-theory by $\TC$ (see
\Cref{cons:basicfibersquare} and \Cref{TCformThmA}); the henselian condition is only needed to
translate between $\K$-theory and $\TC$.

Next, we observe some consequences of and complements to Theorem~\hyperref[thm:a]{A}.
In \cite{Beilinson},  slightly more than an equivalence of
rational spectra $\K(R,(p);\QQ_p)\we\Sigma\HC(R;\QQ_p)$ is proved: there is a natural zig-zag of 
``quasi-isogenies'' of spectra before inverting $p$.    By definition, a
quasi-isogeny is a map which is an equivalence up to uniformly bounded denominators in any finite range of degrees.
We also obtain the same conclusion in 
our setting and can keep track of the denominators at least in some
range.

\begin{namedcorollary}[B] 
\label{quasiisogenythmB}
Let $R$ be a commutative ring which is henselian along $(p)$. Then there is a natural
zig-zag of quasi-isogenies between $\K(R, (p); \mathbb{Z}_p)$ and $  \Sigma \HC(R, (p);
\mathbb{Z}_p)$. If $R$ is moreover $p$-torsion free, then there are
isomorphisms $\pi_i\K(R,(p);\ZZ_p)\simeq\pi_i\Sigma\HC(R,(p);\ZZ_p)$ for $i\leq 2p-5$.
\end{namedcorollary} 

A similar result for an arbitrary nilpotent ideal, albeit in a smaller range of
degrees (depending on the exponent of nilpotence of the ideal), is proved in \cite{Brunfiltered}. The argument we use here seems to be special to the ideal
$(p)$. 

As explained above, one could formulate Theorem~\hyperref[thm:a]{A} entirely in
the language of topological cyclic homology, completely avoiding the mention
of $\K$-theory. At some point in the proof, however, we translate back into
$\K$-theory and use a homology argument. Therefore,  we also offer an alternative purely cyclotomic proof of this step. 
This relies on a study of 
quasi-isogenies in the homotopy theory of cyclotomic spectra, based on the
$t$-structure introduced by Antieau--Nikolaus \cite{AN18}.  
The key step is an extension of a theorem of Geisser--Hesselholt \cite{GHfinite}
and Land--Tamme \cite{LandTamme}. 

\begin{namedtheorem}[C]\label{thm:c}
    Let $f\colon  A\rightarrow A'$ be a map of connective associative ring
    spectra. Suppose that
    \begin{enumerate}
        \item[{\rm (1)}] $f$ is a quasi-isogeny of spectra and
        \item[{\rm (2)}] the map $\pi_0(f)\colon\pi_0(A)\rightarrow\pi_0(A')$
            is surjective with nilpotent kernel.
    \end{enumerate}
    Then $\THH(A; \mathbb{Z}_p)\rightarrow\THH(A'; \mathbb{Z}_p)$ is a quasi-isogeny in cyclotomic spectra
    and in particular the induced map $\TC(A;\ZZ_p)\rightarrow\TC(A';\ZZ_p)$ is a
    quasi-isogeny.
\end{namedtheorem}

\subsection{$p$-adic deformations of $\K$-theory classes}
In our first main application of  Theorem~\hyperref[thm:a]{A}, we generalize work of
Bloch--Esnault--Kerz~\cite{bek1} and Beilinson~\cite{Beilinson} on the formal
$p$-adic deformation of
rational $\K$-theory classes. 
Let us first recall the motivation for their work. 

Fix a complete discretely valued field $K$ of mixed characteristic $(0, p)$ with ring
of integers $\mathcal{O}_K$ and perfect residue field $k$ as well as a proper smooth
scheme $X \to \Spec(\mathcal{O}_K)$ with special fiber $X_k$ and generic fiber
$X_K$. 
Given $X$, we can consider the algebraic
de Rham cohomology $H^{*}_{\dR}(X_K/K)$ of the generic fiber, together with its
Hodge filtration $\mathrm{Fil}^{\geq\star} H^{*}_{\dR}(X_K/K)$; these are
finite-dimensional $K$-vector spaces, and arise as the 
cohomology groups of objects $\mathrm{Fil}^{\geq \star} R \Gamma_{\dR}(X_K/K)$
in the derived category of $K$.

As usual, we have a Chern character 
\begin{equation} \mathrm{ch}\colon \K_0(X; \mathbb{Q}) 
\twoheadrightarrow \K_0(X_K; \mathbb{Q}) \to 
H^{\mathrm{even}}_{\mathrm{dR}}(X_K/K). \end{equation}
A foundational motivating question is to determine the image of this map: in
other words, to determine which cohomology classes come from algebraic cycles
on $X_K$ (or equivalently on $X$). 
A conjecture of Fontaine--Messing, a $p$-adic analog of the variational Hodge
conjecture of Grothendieck \cite[Footnote 13]{Gro66},  predicts 
that this question can essentially be reduced from mixed to equal characteristic. 

To formulate the conjecture, we consider also the 
(absolute)
crystalline cohomology 
$H^*_{\crys}(X_k)$ of the special fiber, a family of finitely generated
$W(k)$-modules. 
By the de Rham-to-crystalline
comparison \cite{BO83}, we have an isomorphism  
$H^*_{\crys}(X_k) \otimes_{W(k)} K \iso H^*_{\dR}(X_K/K) .$
Finally, we have the crystalline Chern character map \cite{Gros85},
\begin{equation} \label{cryschernchari}
\mathrm{ch}_{\mathrm{crys}}\colon \K_0(X_k) \to \bigoplus_{i \geq 0}
H^{2i}_{\crys}(X_k; \mathbb{Q}_p), 
\end{equation}
leading to a commutative diagram
\begin{equation} 
    \begin{gathered}
\xymatrix{
\K_0(X; \mathbb{Q})  \ar[d]^{\mathrm{ch}}  \ar[r] &  \K_0(X_k; \mathbb{Q})
\ar[d]^{\mathrm{ch}_{\mathrm{crys}}}  \\
H^{\mathrm{even}}_{\mathrm{dR}}(X_K/K) \ar[r]^<<<<{\simeq} \ar[r] &  
H^{\mathrm{even}}_{\mathrm{crys}}(X_k) \otimes_{W(k)} K.
}
    \end{gathered}
\end{equation}

\begin{conjecture}[$p$-adic variational Hodge conjecture] \label{FMconj} Let 
$\alpha \in H^{\mathrm{even}}_{\mathrm{dR}}(X_K/K)$. Then $\alpha$ belongs to
the image of the Chern character from $\K_0(X; \mathbb{Q})$ if and only if
\begin{enumerate}
\item the image of $\alpha$ under the 
de Rham-to-crystalline isomorphism in 
$H^{\mathrm{even}}_{\mathrm{crys}}(X_k) \otimes_{W(k)} K$ belongs to the image
of the crystalline Chern character from 
$\K_0(X_k; \mathbb{Q})$ and
\item the class $\alpha$ belongs to $\bigoplus_i \mathrm{Fil}^{\geq i}
H^{2i}_{\mathrm{dR}}(X_K/K) \subseteq H^{\mathrm{even}}_{\mathrm{dR}} (X_K/K)$. 

\end{enumerate}
\end{conjecture} 
For further details and arithmetic applications of the $p$-adic variational Hodge
 conjecture, we refer
to~\cite{emerton-variational}.

Motivated by Conjecture~\ref{FMconj},
 Bloch--Esnault--Kerz \cite{bek1} considered  
the following {$p$-adic deformation} question, which starts with a $\K_0$-class
on the special fiber (rather than a cohomology class) and asks when it lifts
infinitesimally. 

\begin{question}[The $p$-adic deformation problem] 
Given the data as above, define 
the ``continuous'' $\K$-theory $$\K^{\mathrm{cts}}(X)
\stackrel{\mathrm{def}}{=} \varprojlim \K(X/\pi^n)
,$$ 
where $\pi$ is a uniformizer of $\Oscr_K$.
Given a class $x \in 
\K_0(X_k; \mathbb{Q})$, when does it belong to the image of 
the reduction map from 
the continuous $\K$-theory
$\K_0^{\mathrm{cts}}(X; \mathbb{Q}) = \pi_0
(\K^{\mathrm{cts}}(X))_{\mathbb{Q}}$? 
\end{question}

Since the map $\K(X) \to \K^{\mathrm{cts}}(X)$ is generally not an equivalence,
the $p$-adic deformation problem does not imply \Cref{FMconj}. 
However, the $p$-adic deformation problem 
is a (pro)-infinitesimal one, so it can be
studied using methods of topological cyclic homology. 
Using the Beilinson fiber square, we answer the $p$-adic deformation problem as
follows; 
in \cite{bek1}, this result is proved for  a smooth projective scheme of
dimension $d < p + 6$ when $K$ is unramified. 

\begin{namedtheorem}[D] \label{thm:d} 
Let $X$ be a proper smooth scheme over $\mathcal{O}_K$. 
A class $x \in \K_0(X_k; \mathbb{Q})$ lifts to $\K_0^{\mathrm{cts}}(X;
\mathbb{Q})$ if and only if 
$\mathrm{ch}_{\mathrm{crys}}(x) \in \bigoplus_{i \geq
0}H^{2i}_{\mathrm{crys}}(X_k; {\mathbb{Q}_p})$ is carried by the de
Rham-to-crystalline comparison isomorphism to a class in $\bigoplus_{i \geq 0}
\mathrm{Fil}^{\geq i}H^{2i}_{\dR}(X_K/K) \subseteq \bigoplus_{i \geq 0}
H^{2i}_{\dR}(X_K/K)$.
\end{namedtheorem} 

Our main observation is that 
Theorem~\hyperref[thm:a]{A} together with
Hochschild--Kostant--Rosenberg comparisons between cyclic and
de Rham cohomology yield a fiber square
\begin{equation}\label{OKfibsquare}\begin{gathered}  \xymatrix{
\K^{\mathrm{cts}}(X; \mathbb{Q}) \ar[d] \ar[r] &  \K(X_k; \mathbb{Q}) \ar[d]
\\
\prod_{i \in \mathbb{Z}} \mathrm{Fil}^{\geq i} R \Gamma_{\mathrm{dR}}(X_K/K)[2i] \ar[r] &  
\prod_{i \in \mathbb{Z}}  R \Gamma_{\mathrm{dR}}(X_K/K)[2i].
}\end{gathered} \end{equation}
Moreover, on $\K_0$, one checks that the vertical map on the right-hand side induces 
 the crystalline Chern character \eqref{cryschernchari}, at least up to
 scalars, implying the result. 
For this argument, it is crucial that one has the fiber square 
\eqref{OKfibsquare}, rather than a fiber sequence alone. 

One can also generalize the above questions to higher $\K$-theory. 
In  \cite{Beilinson}, 
the Beilinson fiber sequence is used to prove that if $x \in \K_j(X_k;
\mathbb{Q})$, then 
there exists a natural obstruction class 
in 
$\bigoplus_{i \geq 0} H^{2i-j}_{\mathrm{dR}}(X_K)/\mathrm{Fil}^{\geq i}
H^{2i-j}_{\mathrm{dR}}(X_K)$  which vanishes if and only if $x $ lifts to the
continuous $\K$-theory $\K^{\mathrm{cts}}_i( X; \mathbb{Q})$;
however, \cite{Beilinson}
does not 
identify the class
with the crystalline Chern character for $i  = 0$. 
Here we also extend this result to an
arbitrary quasi-compact and quasi-separated (qcqs) scheme with bounded $p$-power
torsion, using $p$-adic derived de Rham
cohomology \cite{Bhattpadic} and results of~\cite{antieauperiodic}.

\begin{namedtheorem}[E]\label{thm:e}
    Let $X$ be a qcqs scheme with
	 bounded $p$-power torsion. For each
     $n$ we write $X_n$ for $X\times_{\Spec\ZZ}\ZZ/p^n$, and put $\K^{\mathrm{cts}}(X)
\stackrel{\mathrm{def}}{=} \varprojlim \K(X_n)$.
	 Given a
    class $x\in\K_j(X_1;\QQ)$, there is a natural class
    $$c(x)\in\bigoplus_{i\geq
    0}H^{2i-j}(\L\Omega_X/\L \Omega^{\geq i}_X)_{\mathbb{Q}_p},$$ where 
	$\L\Omega_X$ is the $p$-adic derived de Rham cohomology of $X$ with the derived Hodge
	filtration $\L\Omega^{\geq\star}_X$. The class
    $x$ lifts to $\K_i^{\mathrm{cts}}(X; \mathbb{Q})$ if and only if
    $c(x)=0$.
\end{namedtheorem}

\subsection{The motivic filtration on $\mathrm{TC}$}
In 
\cite{BMS2}, 
 Bhatt--Morrow--Scholze 
discovered a fundamental additional structure on the $p$-adic topological cyclic
homology $\mathrm{TC}(-; \mathbb{Z}_p)$ of $p$-adic \emph{commutative} rings: a
``motivic filtration''
$\mathrm{Fil}^{\geq\star} \mathrm{TC}( -; \mathbb{Z}_p)$
on $\mathrm{TC}(-; \mathbb{Z}_p)$
with associated graded terms 
denoted
 $\mathbb{Z}_p(i)[2i]$. 
 The objects $\mathbb{Z}_p(i)$ thus obtained are related to integral $p$-adic Hodge
 theory and can be defined (independently of topological cyclic homology) as a type of filtered Frobenius invariants on the
 prismatic cohomology \cite{Prisms}. 
 They are known explicitly in some cases: in characteristic $p> 0$ they can be identified
 with logarithmic de Rham--Witt sheaves (up to a shift), and for formally smooth
 algebras over $\mathcal{O}_{C}$ (for $C$ a complete algebraically closed
 nonarchimedean field), they can be identified with truncated $p$-adic nearby
 cycles.

Recall also that for $p$-adic (commutative) rings, $\mathrm{TC}(-;
\mathbb{Z}_p)$ is $p$-adic \'etale $\K$-theory in
nonnegative degrees \cite{GH99, CMM, CM}. Therefore, it is expected (but not known
in mixed characteristic) that 
the filtration 
$\mathrm{Fil}^{\geq\star} \mathrm{TC}( -; \mathbb{Z}_p)$
is the \'etale sheafified motivic filtration on algebraic $\K$-theory, and that
the $\mathbb{Z}_p(i)$ are $p$-adic \'etale motivic cohomology, at least where all of
these objects are defined. 
One also has constructions of Schneider and Sato \cite{Sa07} of ``$p$-adic
\'etale Tate twists,'' which satisfy a type of arithmetic duality. 
In general, one expects that the $\mathbb{Z}_p(i)$ should be related to
important foundational questions in arithmetic geometry and $\K$-theory. 
An advantage of the construction of 
$\mathrm{Fil}^{\geq\star} \mathrm{TC}( -; \mathbb{Z}_p)$
and the $\mathbb{Z}_p(i)$ as in \cite{BMS2}
is that it
works in a much more general setting (for the quasisyntomic rings; see
Section~\ref{sec:review}
below for a review) than existing approaches to motivic
cohomology. Moreover, its definition is extremely direct: it is simply a sheafified
Postnikov tower (albeit for a ``large'' topology). 
 
Using Theorem~\hyperref[thm:a]{A}, 
we will give a description of the $\mathbb{Z}_p(i)$ for $i \leq p-2$ and of
the $\mathbb{Q}_p(i)$ for all $i$ in terms of syntomic
cohomology as considered by Fontaine--Messing \cite{FM87} and Kato \cite{Ka87}.
In particular, this construction gives a description of the $\mathbb{Z}_p(i)$
(with the above restrictions) that relies only on derived de Rham theory, rather
than prismatic theory. 
Our result is an analog of 
a result of Geisser \cite{Ge04} for \'etale motivic cohomology for
smooth schemes over Dedekind rings. 

To formulate the result, we write $L \Omega_R$ for the $p$-adic derived de Rham
cohomology for a commutative ring $R$ equipped with its derived Hodge filtration
$L \Omega_R^{\geq\star}$ \cite{Bhattpadic}. 
The object $L \Omega_R$ carries a crystalline Frobenius $\varphi\colon
L \Omega_R \to L \Omega_R$. 
For $i < p$, 
one has a ``divided'' Frobenius $\varphi/p^i\colon L \Omega_R^{\geq i} \to L
\Omega_R$. Using the techniques of \cite{BMS2} (in particular, quasisyntomic
sheafification) applied to the Beilinson fiber square, we deduce our next
theorem.

\begin{namedtheorem}[F]\label{thm:f}
Let $R$ be a quasisyntomic ring.
\begin{enumerate}
    \item[{\rm (1)}]  
For each $i \geq 0$, there is an identification
$$ \mathbb{Q}_p(i)(R) \simeq 
\mathrm{fib}( \varphi - p^i\colon L \Omega_R^{\geq i} \to L
\Omega_R)_{\mathbb{Q}_p}.$$
\item[{\rm (2)}]
For $i \leq p-2$, there is an identification 
$$ \mathbb{Z}_p(i)(R) \simeq \mathrm{fib}( \phi/p^i - \mathrm{id}\colon L
\Omega_R^{\geq i} \to L \Omega_R ).$$
\end{enumerate}
\end{namedtheorem} 

We explicitly analyze Theorem~\hyperref[thm:f]{F} in
three cases in which one has alternate descriptions of the $\mathbb{Z}_p(i)$: rings of integers in $p$-adic fields, perfectoid rings, and
formally smooth $\mathcal{O}_C$-algebras where $C$ is an algebraically closed,
complete nonarchimedean field of mixed characteristic.
The first case recovers classical calculations of the rational $p$-adic $\K$-theory of $p$-adic fields;
the second case recovers the fundamental exact sequence in $p$-adic
Hodge theory; the last case 
recovers results of Colmez--Nizio\l~
\cite{CN17} on $p$-adic vanishing cycles, albeit only in the formally smooth case.

Finally, Theorem~\hyperref[thm:f]{F}  provides a complete computation of low-degree or rationalized
$\mathrm{TC}$ in terms of syntomic cohomology. This computation relies on the  following 
 connectivity estimate about the
$\mathbb{Z}_p(i)$ and about the filtration on $\mathrm{TC}(-;
\mathbb{Z}_p)$.  
The estimate for algebras over a perfectoid ring is stated in
\cite[Constr.~7.4]{BMS2}; the argument for all quasisyntomic rings relies on the use of
relative topological Hochschild homology and the spectral sequence of
Krause--Nikolaus~\cite{KN}. 

\begin{namedtheorem}[G]
If $R$ is a quasisyntomic ring, then 
$\mathbb{Z}_p(i)(R) \in D^{\leq i+1}( \mathbb{Z}_p)$. 
If $R$ is
$w$-strictly local (e.g., strictly henselian local), then 
$\mathbb{Z}_p(i)(R) \in D^{\leq i}(
\mathbb{Z}_p)$. 
Consequently, $\mathrm{Fil}^{\geq i} \mathrm{TC}(R; \mathbb{Z}_p)$ is
concentrated in homological degrees $\geq i-1$ (and pro-\'etale locally $\geq
i$). 
\end{namedtheorem} 

\begin{namedcorollary}[H]
If $R$ is any commutative ring, then there is a natural equivalence
\[ \TC(R; \mathbb{Q}_p) \simeq \bigoplus_{i \geq 0} \mathrm{fib}( \varphi -
p^i\colon L \Omega_R^{\geq i} \to L \Omega_R)_{\mathbb{Q}_p}.  \]
\end{namedcorollary}

\subsection*{Notation} 
Throughout, we write $\sp$ for the $\infty$-category of spectra. 
Given a ring $R$, we write $D(R)$
for the derived $\infty$-category of $R$.

We will use homological indexing conventions indicated with a subscript
when referring to spectra and
cohomological indexing conventions indicated with a superscript when referring to objects of the derived
category. 
For instance, given $n$, we
write $\sp_{\geq n}$ (resp.~$\sp_{\leq n}$) for spectra with homotopy groups
concentrated in degrees $\geq n$ (resp.~$\leq n$); 
we write $D(R)^{\leq n}$ (resp.~$D(R)^{\geq n})$ for objects of $D(R)$ with cohomology groups concentrated in
degrees $\leq n$ (resp.~$\geq n$).

We write $\HH(R)$ for the Hochschild homology of $R$, always relative to
$\mathbb{Z}$ and always computed in a derived sense (also known as Shukla
homology), and we let $\THH(R)$ denote the topological Hochschild homology of
$R$. We write
$\HC^-(R)  = \HH(R)^{hS^1}$ for negative cyclic homology and $\HP(R) =
\HH(R)^{tS^1}$ for periodic cyclic homology. 
For a scheme $X$, we let $L\Omega_X$ denote its $p$-completed derived
de Rham cohomology (relative to $\mathbb{Z}$) and 
$\L\Omega_X^{\geq i}$ for the $i$th stage of the
(derived) Hodge filtration. 
We denote the Hodge-completed variants by $\widehat{\L \Omega}_X$
and $\widehat{\L
\Omega}^{\geq i}_X$, respectively.

\subsection*{Acknowledgments}
We are very grateful to Peter Scholze for suggesting this question to us and for
sharing many insights.  We would also like to thank Johannes Ansch\"utz,
Alexander Beilinson, Bhargav Bhatt, H\'el\`ene Esnault,   Lars Hesselholt,
Luc Illusie, Moritz Kerz,
Arthur-C\'esar Le Bras,
Wies\l{}awa Nizio\l{}, and
Sam Raskin for helpful conversations. 
The third author would like to thank Uwe Jannsen, Moritz Kerz, and Guido Kings
for organising and inviting him to the 2014 Kleinwalsertal workshop on
Beilinson's paper \cite{Beilinson}. 

Two referees provided detailed, invaluable comments on the paper which has been
improved as a result.

The first two authors would like to thank the University of Bonn and the University
of M\"unster for their hospitality. This material is based upon work supported by the National Science Foundation
under Grant No. DMS-1440140 while the first three authors were in residence at the
Mathematical Sciences Research Institute in Berkeley, California, during the
Spring 2019 semester.
The first author was supported by NSF Grant DMS-1552766.
This work was done while the second author was a Clay Research Fellow. The
fourth author was funded by the Deutsche Forschungsgemeinschaft (DFG, German Research Foundation) under Germany's Excellence Strategy 
EXC 2044 Ð390685587, Mathematics M\"unster: Dynamics--Geometry--Structure.

\section{The Beilinson fiber square}\label{sec:square}

\subsection{Background}

We review some background on the theory of  cyclotomic spectra and topological cyclic homology as in
\cite{NS18}, of which we will  use the $p$-typical variant. This theory uses the $\infty$-category $\mathrm{Fun}(BS^1,\Sp)$ of spectra
equipped with $S^1$-actions. Given a spectrum $X$ equipped with an
$S^1$-action, we can form the homotopy $S^1$-orbits $X_{hS^1}$, the homotopy
$S^1$-fixed points $X^{hS^1}$, and the $S^1$-Tate construction $X^{tS^1}$.
These are related by a natural fiber sequence $\Sigma X_{hS^1}\rightarrow
X^{hS^1}\rightarrow X^{tS^1}$, which we will use constantly and without further
comment. See for example~\cite[Cor.~I.4.3]{NS18}.

\begin{definition}[Nikolaus--Scholze \cite{NS18}] 
We let $\cycsp$ denote the symmetric monoidal, stable $\infty$-category of
\emph{cyclotomic spectra}.\footnote{Our conventions are slightly different from
those of \cite{NS18}, which requires a Frobenius map for each prime number
$q$ and not only the fixed prime $p$. What we call a cyclotomic spectrum is called a
\emph{$p$-typical cyclotomic spectrum} by \cite{AN18}, and what we write as
$\cycsp$ is written as $\cycsp_p$ in \cite{AN18}. For $p$-complete objects
(the primary case of interest), there is no difference between a
cyclotomic spectrum in the sense of \cite{NS18} and the definition used here.}
An object of $\cycsp$ consists of a spectrum $X$
equipped with an $S^1$-action and an $S^1$-equivariant
map $\varphi_p\colon X \to
X^{tC_p}$ called the cyclotomic Frobenius.

Given $X \in \cycsp$, we write
    \[ \TC^-(X) = X^{hS^1} \quad\text{and}\quad \TP(X) = X^{tS^1}.  \]
We will define only $p$-complete $\TC$ for $X \in \cycsp$  and will assume that $X$ is bounded below.
We have two maps
\( \can\colon \TC^-(X;\ZZ_p) \to \TP(X;\ZZ_p)\text{ and } \varphi\colon \TC^-(X;\ZZ_p) \to
\TP(X;\ZZ_p).  \)
By definition, $\can$ is the canonical map from $S^1$-invariants to the Tate
construction, and $\varphi$ is induced from the Frobenius $\varphi_p$. 
The $p$-complete {topological cyclic homology} $\TC(X;\ZZ_p)$, for $X \in
\cycsp$ bounded below, can be
computed as the
fiber  of the difference of the two maps, i.e., 
\begin{equation} \label{TCformula}  \TC(X;\ZZ_p) = \mathrm{fib}\left(\can  - \varphi\colon
\TC^-(X;\ZZ_p) \to \TP(X;\ZZ_p)\right).   \end{equation}
\end{definition}

\begin{remark}
    We will use throughout the basic fact that if $X \in \cycsp$ has underlying
    $n$-connective spectrum, then $\TC(X;\ZZ_p)$ is $(n-1)$-connective (see for
    example~\cite[Lem.~2.5 and Rem.~2.14]{CMM}).
\end{remark}

\begin{example} \label{ex_cyc}
\begin{enumerate}
\item  
Given a ring $R$, we can form the topological Hochschild
homology $\THH(R)$ as a cyclotomic spectrum.   
\item
Given a spectrum $Y$, 
we let $Y^{\triv}$ be the cyclotomic spectrum where we view $Y$ as a spectrum
with trivial $S^1$-action and with cyclotomic Frobenius
given by the natural map
$Y \to Y^{hC_p} \to Y^{tC_p}$. 
\item For a spectrum $X$ with $S^1$-action we get a cyclotomic spectrum by
    letting $\varphi_p\colon X \to X^{tC_p}$ be zero (as an $S^1$-equivariant map). 
\end{enumerate}
\end{example} 

\begin{remark} 
Let $Y$ be a bounded below spectrum of finite type, meaning that each $\pi_iY$
is a finitely generated abelian group. If $X \in \cycsp$ is
$p$-complete and bounded below, then there is a natural equivalence
    \begin{equation} \label{TCotimestriv} \TC(X \otimes_{\SS} Y^{\triv}; \mathbb{Z}_p) \simeq \TC(X;
\mathbb{Z}_p) \otimes_\SS Y.\end{equation} This is a
consequence of the fact that the functor $\TC(-; \mathbb{Z}_p)$ commutes with 
geometric realizations of
connective cyclotomic spectra, since it is exact and carries $n$-connective
objects into $(n-1)$-connective objects. 
This even holds if $Y$ is only assumed to be bounded below, as long as the
right side of \eqref{TCotimestriv} is $p$-adically completed, by~\cite[Th.~2.7]{CMM}.
\end{remark} 

As a simple exercise with the above definitions, we prove the following result
for use below. This has been used in other references as well, e.g.,
\cite[Sec.~1.4]{HN}. 

\begin{proposition}
\label{TCS1orb}
If $X \in \cycsp$ is bounded below and $\TP(X;\ZZ_p) =0 $, then we have natural equivalences
$ \TC(X;\ZZ_p) \simeq 
(\Sigma X_{hS^1})^\wedge_p \simeq
\TC^-(X;\ZZ_p)
$.\end{proposition} 

\begin{proof} 
    In fact, the formula \eqref{TCformula} shows that $\TC(X;\ZZ_p) = \TC^-(X;\ZZ_p)$. 
    Now $\TP(X;\ZZ_p)$ is the cofiber of the norm map $(\Sigma X_{hS^1})^\wedge_p \to \TC^-(X;\ZZ_p)$.
    Since we have assumed $\TP(X;\ZZ_p) =0$, we have $(\Sigma X_{hS^1})^\wedge_p\simeq
    \TC^-(X;\ZZ_p)$.
    Combining the two identifications, we conclude. 
\end{proof} 

Next, we apply this to a specific crucial example.

\begin{cons}[The cyclotomic spectrum $\mathbb{Z}_{hC_p}$]\label{cons:thhfp}
Recall first that the cyclotomic
trace (or a direct construction) gives a map  \begin{equation}
\mathbb{Z}^{\triv} \to \THH( \mathbb{F}_p) \label{ZtoTHHFp}
\end{equation} 
in $\cycsp$, and that 
as objects of $\fun(BS^1, \sp)$ we have 
$\mathrm{THH}( \mathbb{F}_p) \simeq \tau_{\geq 0}(
\mathbb{Z}^{tC_p})$ by \cite[Sec.~IV-4]{NS18}; here we obtain the $S^1$-action
on $\mathbb{Z}^{tC_p}$ via the sequence $C_p \to S^1 \xrightarrow{z \mapsto
z^p} S^1$. This is a refinement of (and deduced from) B\"okstedt's calculation of $\THH(
\mathbb{F}_p)$. 
Consequently, 
there is a cofiber sequence in $\cycsp$,
\begin{equation} \ZZ_{hC_p}\rightarrow\ZZ^\triv\rightarrow\THH(\FF_p),
\end{equation} where $\ZZ_{hC_p}$ is
a cyclotomic spectrum with 
underlying spectrum with $S^1$-action $\mathbb{Z}_{hC_p} \in \fun(BS^1, \sp)$. 

Note that 
$(\ZZ_{hC_p})^{tC_p}\we 0$ by the Tate orbit lemma \cite[Lem.~I.2.1]{NS18}.
In particular, the map 
$\mathbb{Z}^{\mathrm{triv}} \to \mathrm{THH}( \mathbb{F}_p)$ induces an
equivalence on $\TP(-; \mathbb{Z}_p)$, i.e., 
$\mathbb{Z}_p^{tS^1} \simeq \mathrm{THH}(\mathbb{F}_p)^{tS^1}=\TP(\FF_p)$.
We obtain 
by \Cref{TCS1orb} that $\mathrm{TC}( \mathbb{Z}_{hC_p}; \mathbb{Z}_p) \simeq
\Sigma (\mathbb{Z}_p)_{hS^1}$. 
Our next observation is that this remains true after tensoring with any bounded
below cyclotomic spectrum.
\end{cons}

\begin{lemma} 
\label{Tatevanish}
If $X \in \fun(BS^1, \sp)$, then
\(
    (X \otimes_\SS \mathbb{Z}_{hC_p})^{tS^1}
\)
is $p$-adically zero (here we use the diagonal $S^1$-action).
\end{lemma} 

\begin{proof}
    The spectrum $(X \otimes_\SS \mathbb{Z}_{hC_p})^{tS^1}$ is a module over
    $(\mathbb{Z}^{hC_p})^{tS^1}$. Since $(\mathbb{Z}^{hC_p})^{tS^1}$ 
    vanishes $p$-adically by the Tate fixed point lemma \cite[Lem.~I.2.2 and
	 Lem.~II.4.2]{NS18}, the lemma follows.
	 Alternatively, one can easily verify that $(\mathbb{F}_p)_{hC_p} \in
	 \fun(BS^1, \sp)$, is induced from the
	 trivial subgroup, which forces the $p$-adic Tate vanishing. 
\end{proof} 

Combining Proposition~\ref{TCS1orb} and Lemma~\ref{Tatevanish}, we conclude that if $X$ is any bounded below
cyclotomic spectrum, then there are equivalences
\begin{equation} \label{TCXtimesZ}
    \TC(X \otimes_\SS \mathbb{Z}_{hC_p}; \mathbb{Z}_p) \simeq \Sigma((X
    \otimes_\SS \mathbb{Z}_{hC_p})_{hS^1})^\wedge_p \simeq \TC^-(X \otimes_\SS \mathbb{Z}_{hC_p};\ZZ_p).
\end{equation}

\subsection{Pullback squares} Next, we establish some 
pullback squares involving cyclotomic spectra and give a proof of
Theorem~\hyperref[thm:a]{A}. 

\begin{proposition}
\label{mainfibseq}
Let $X \in \cycsp$ be a bounded below cyclotomic spectrum. Then 
the commutative square
\begin{equation} \label{TChtpycart}\begin{gathered}
\xymatrix{
\TC( X \otimes_\SS \mathbb{Z}^{\triv}; \mathbb{Z}_p) \ar[d]  \ar[r] &  \TC(X
    \otimes_\SS \THH(
\mathbb{F}_p); \mathbb{Z}_p)
\ar[d]  \\
\TC^-(X  \otimes_\SS \mathbb{Z}^{\triv};\ZZ_p) \ar[r] & \TC^-(X \otimes_\SS
\THH(\mathbb{F}_p);\ZZ_p)
} \end{gathered}
\end{equation}
is cartesian,
where the horizontal maps arise from the map $\mathbb{Z}^{\triv} \to
\THH(\mathbb{F}_p)$ in $\cycsp$ of Construction~\ref{cons:thhfp} and the vertical maps are the
canonical maps $\TC(-;\ZZ_p) \to \TC^-(-;\ZZ_p)$ arising from the definition of
$\TC(-;\ZZ_p)$.

Moreover, there is a natural fiber sequence 
\begin{equation}
\label{fibseqTC1}
(\Sigma (X \otimes_\SS \mathbb{Z}_{hC_p})_{hS^1})^\wedge_p \to \TC(X
    \otimes_\SS
\mathbb{Z}^{\triv}; \mathbb{Z}_p)  \to \TC(X \otimes_\SS \THH(\mathbb{F}_p);
\mathbb{Z}_p). 
\end{equation}
\end{proposition} 
\begin{proof}
Since $\mathrm{TC}(Z; \mathbb{Z}_p)$  for a bounded below cyclotomic spectrum $Z$ is an equalizer of two maps 
$\TC^-(Z;\ZZ_p) \rightrightarrows \TP(Z;\ZZ_p)$, the statement that 
\eqref{TChtpycart} is cartesian follows from 
the fact that $X \otimes_{\SS} \mathbb{Z}^{\triv} \to X \otimes_{\SS} \THH(\mathbb{F}_p)$
induces an equivalence on $\TP(-;\ZZ_p)$, via 
\Cref{Tatevanish}. 
Moreover, the fiber sequence \eqref{fibseqTC1} then follows from 
\eqref{TChtpycart} via taking fibers, and using \Cref{Tatevanish} again
to replace homotopy fixed points by homotopy orbits. 
Alternatively, to prove that \eqref{TChtpycart} is cartesian, one observes that the
fibers of the horizontal arrows
are $\TC(X \otimes_\SS \mathbb{Z}_{hC_p}; \mathbb{Z}_p)$ and $\TC^-(X
    \otimes_\SS
\mathbb{Z}_{hC_p};\ZZ_p)$
and these are naturally equivalent as in \eqref{TCXtimesZ}. 
\end{proof}

\begin{corollary}\label{cor_ring}
For every connective ring spectrum $R$ we have a natural fiber sequence
of $p$-complete spectra
\[
\Sigma  \left(\THH(R;\ZZ_p) \otimes_\SS \mathbb{Z}_{hC_p}\right)_{hS^1} \to \TC(R;\ZZ_p)\otimes_\SS \ZZ  \to \TC(R\otimes_\SS\FF_p;\ZZ_p)
\]
\end{corollary}
\begin{proof}
We apply  \Cref{mainfibseq} to $X = \THH(R; \mathbb{Z}_p)$. We have that $\THH(R)
\otimes_\SS \THH(\FF_p) \simeq \THH(R \otimes_\SS \FF_p)$ which gives the
identification of the third term. For the identification of the term in the
middle we observe that   $ \TC(X \otimes_\SS
\mathbb{Z}^{\triv}; \mathbb{Z}_p) \simeq  \TC(X; \mathbb{Z}_p) \otimes_\SS
\mathbb{Z}$ by 
\eqref{TCotimestriv}. 
Note finally that for bounded below spectra, tensoring with $\mathbb{Z}$
preserves $p$-completeness as $\mathbb{Z}$ is finite type. 
\end{proof}

Next, we study what happens in \eqref{TChtpycart} after rationalization. 

\begin{corollary} 
\label{getcartsquare1}
Let $X \in \cycsp$ be a bounded below, $p$-complete cyclotomic spectrum. 
Then there exists a natural map 
$\TC(X \otimes_{\SS} \THH(\mathbb{F}_p); \mathbb{Z}_p) \to (X \otimes_{\SS}
\mathbb{Z}^{\triv})^{tS^1}$ which fits into a natural commutative square
    \begin{equation} \label{secondsquare}\begin{gathered}  \xymatrix{
\TC(X \otimes_\SS \mathbb{Z}^{\triv}; \mathbb{Z}_p) \ar[d] \ar[r] &  \TC(X
        \otimes_\SS\THH(\mathbb{F}_p);
\mathbb{Z}_p) \ar[d] \\
(X \otimes_\SS \mathbb{Z}^{\triv})^{hS^1} \ar[r] & (X \otimes_\SS
\mathbb{Z}^{\triv})^{tS^1}.
    }\end{gathered} \end{equation}
Moreover, this square 
becomes cartesian after inverting $p$. 
\end{corollary} 
\begin{proof} 
We can vertically extend the cartesian square \eqref{TChtpycart} 
via the canonical maps $(-)^{hS^1} \to (-)^{tS^1}$. 
In this case, as we saw earlier, the map 
$(X \otimes_\SS \mathbb{Z}^{\triv})^{tS^1} \to (X \otimes_\SS
\THH(\mathbb{F}_p))^{tS^1}$ is an equivalence. 
Using this identification, we obtain the commutative square \eqref{secondsquare}. 
The fact that \eqref{secondsquare} is cartesian after inverting $p$
follows from the facts that \eqref{TChtpycart} is cartesian and that
$(X \otimes_\SS \THH(\mathbb{F}_p))^{hS^1} \to 
(X \otimes_\SS \THH(\mathbb{F}_p))^{tS^1} \simeq 
(X \otimes_\SS \mathbb{Z}^{\triv})^{tS^1}$ becomes an equivalence after inverting
$p$. 
\end{proof}

\begin{remark}[Effective bounds for the denominators in \Cref{getcartsquare1}] 
\label{remark:effbound1}
For future reference, it will be helpful to give a more effective version of 
\Cref{getcartsquare1}. 
Consider the total cofiber (cofiber of horizontal cofibers) 
of the square \eqref{secondsquare}. This is given by $\Sigma^2 ( X \otimes_\SS \THH(
\mathbb{F}_p))_{hS^1}$ because \eqref{TChtpycart} is homotopy cartesian.
If $X$ is connective, then it follows that the $\tau_{\leq 2i}$ of the total cofiber is 
annihilated by $p^i$, since $\tau_{\leq 2i-2} \THH(\mathbb{F}_p)$ is
$S^1$-equivariantly annihilated by $p^i$: indeed, this follows by the
($S^1$-equivariant) Postnikov
filtration, since each nonzero homotopy group of $\THH(\mathbb{F}_p)$ is in even
degree and is an
$\mathbb{F}_p$-vector space. 
\end{remark} 

Consequently, we can deduce the following fiber square, which is the basic
$\TC$-theoretic result from which the Beilinson fiber square is a consequence. 
\begin{theorem}[] 
\label{cons:basicfibersquare}
Let $R$ be a ring (or more generally, a connective associative
$H\mathbb{Z}$-algebra spectrum). 
Then there is a natural commutative square of spectra
    \begin{equation}\label{eqn:basicsquare}\begin{gathered}
\xymatrix{
\TC(R; \mathbb{Z}_p) \ar[d]  \ar[r] &  \TC(R \otimes_{\mathbb{S}} \mathbb{F}_p)
\ar[d]  \\
    \HC^-(R; \mathbb{Z}_p) \ar[r] &  \HP(R; \mathbb{Z}_p),}\end{gathered}
\end{equation}
which becomes cartesian after inverting $p$. 
Aside from the right vertical arrow, all the maps are the canonical
ones. 
\end{theorem} 
\begin{proof} 
    Via \eqref{secondsquare} for $X = \THH(R;\ZZ_p)$, 
    we obtain a natural commutative diagram
    \[ 
    \xymatrix{
        \TC(R;\ZZ_p)\ar[d]&\\
    \TC(\THH(R) \otimes_\SS \mathbb{Z}^{\triv}; \mathbb{Z}_p) \ar[d] \ar[r] &  \TC( R \otimes_{\mathbb{S}} \mathbb{F}_p)
    \ar[d]  \\
    (\THH(R; \mathbb{Z}_p) \otimes_\SS \mathbb{Z}^{\triv})^{hS^1} \ar[r] \ar[d]&  (\THH(R;
    \mathbb{Z}_p)
    \otimes_\SS \mathbb{Z})^{tS^1}\ar[d]\\
    \HC^-(R;\ZZ_p)\ar[r]&\HP(R;\ZZ_p),
    }
    \]
    where we use the natural cyclotomic map
    $\THH(R;\ZZ_p)\rightarrow\THH(R;\ZZ_p)\otimes_\SS\ZZ^\triv$
    and the natural $S^1$-equivariant map
    $\THH(R;\ZZ_p)\otimes_\SS\ZZ\rightarrow\HH(R;\ZZ_p)$. The upper square is
    cartesian after inverting $p$ by Corollary~\ref{getcartsquare1}.
    The map
    $\TC(R;\ZZ_p)\rightarrow\TC(\THH(R;\ZZ_p)\otimes_\SS\ZZ^\triv)\we\TC(R;\ZZ_p)\otimes_\SS\ZZ$
    is an equivalence after inverting $p$. The induced map on the bottom horizontal
    fibers is
    $\Sigma(\THH(R;\ZZ_p)\otimes_\SS\ZZ)_{hS^1}\rightarrow\Sigma\HH(R;\ZZ_p)_{hS^1}$,
    which is an equivalence after inverting $p$ since
    $\THH(R;\ZZ_p)\otimes_\SS\ZZ\rightarrow\HH(R;\ZZ_p)$ is an equivalence
	 after inverting $p$ and this property is preserved by taking $S^1$-homotopy orbits. Thus, the bottom square is cartesian
    after inverting $p$. Using these identifications, the theorem follows.
\end{proof} 

\begin{remark}[Effective bounds II]
\label{effectiveII}
Again, one can make effective the statement that \eqref{eqn:basicsquare} is
cartesian after inverting $p$, at least in the range $\leq 2p-4$. 
In this case, 
we find (via \Cref{remark:effbound1}) that 
for $i \leq p-1$, $\tau_{\leq 2i}$ of the total cofiber of
\eqref{eqn:basicsquare} is annihilated by $p^i$. Indeed, the map on cofibers of
the bottom rows of 
\eqref{secondsquare} and 
\eqref{eqn:basicsquare} is given by 
$\Sigma^2 (\THH(R) \otimes_\SS \mathbb{Z}^{\mathrm{triv}} )_{hS^1} \to 
\Sigma^2  \HH(R; \mathbb{Z}_p) _{hS^1}$. This map is an equivalence in degrees
$\leq 2p-2$. 
\end{remark}

\begin{definition}[The $p$-adic Chern character] 
Let $R$ be a ring. 
Consider the map  $\TC(R \otimes_{\mathbb{S}} \mathbb{F}_p) \to \HP(R;
\mathbb{Z}_p)$ from above. After inverting $p$, in view of \Cref{isogenyTHH}
below, 
we have an equivalence $\TC( R \otimes_{\mathbb{S}} \mathbb{F}_p; \mathbb{Q}_p)
\simeq \TC(R/p; \mathbb{Q}_p)$. 
We therefore obtain a map 
$\beta\colon\TC(R/p; \mathbb{Q}_p) \to \HP(R; \mathbb{Q}_p)$, and precomposing with the
trace we obtain 
\[\textrm{tr}_{\mathrm{crys}}=\mathrm{tr}\circ\beta\colon\K(R/p; \mathbb{Q}_p) \to \HP(R; \mathbb{Q}_p).\] 
We call $\textrm{tr}_{\mathrm{crys}}$ the \emph{$p$-adic Chern character} and record that it fits into a natural commutative diagram
    \begin{equation}\label{eqn:basicsquare2}\begin{gathered}
\xymatrix{
\K(R;\mathbb Q_p)\ar[r]\ar[d]_{\mathrm{tr}} &\K(R/p;\mathbb Q_p)\ar[d]^{\mathrm{tr}}\ar@/^1.5cm/[dd]^{\textrm{tr}_{\mathrm{crys}}} \\ 
\TC(R;\mathbb Q_p)\ar[r]\ar[d] & \TC(R/p;\mathbb Q_p)\ar[d]^\beta \\
\HC^-(R;\mathbb Q_p)\ar[r] & \HP(R;\mathbb Q_p)
    }\end{gathered}\end{equation}
in which the bottom square is a pullback.
\end{definition} 

\begin{remark} 
In recent work, Petrov--Vologodsky \cite{PV19} have shown that if 
$p>2$ and $R$ is $p$-torsion free, then there is a natural equivalence
$\HP(R; \mathbb{Z}_p) \simeq \TP(R/p; \mathbb{Z}_p)$. 
Thus, one could attempt to compare the 
$p$-adic Chern character $\mathrm{tr}_{\mathrm{crys}}$ with the usual trace
$\K(R/p; \mathbb{Z}_p) \to \TP(R/p; \mathbb{Z}_p)$. We have not considered this
question. 
\end{remark} 

We can now give a quick proof of Theorem A by combining the above results with the following theorem.

\begin{theorem}[Clausen--Mathew--Morrow \cite{CMM}]\label{CMMthm}
If $R$ is commutative and henselian along $(p)$, then the trace induces an equivalence
$\K(R,(p); \ZZ_p) \simeq  \TC(R,(p); \ZZ_p)$.
\end{theorem}

\begin{remark} 
If $R$ is only associative, but $p$-complete and has bounded $p$-power
torsion,\footnote{If $R$ is non-commutative and 
$p$-complete,  it is natural to ask whether there is still an
equivalence $\K(R,(p); \ZZ_p) \simeq  \TC(R,(p); \ZZ_p)$. We do not know the answer to this question.} then there is an equivalence
$\varprojlim\K(R/p^n,
(p)) \simeq \TC(R,(p); \ZZ_p)$.
This follows by the Dundas--Goodwillie--McCarthy theorem \cite{DGM} and the
$p$-adic continuity of $\mathrm{TC}$, cf.~\cite[Theorem 5.19]{CMM}. 
\end{remark} 

\begin{proof}[Proof of Theorem~{\hyperref[thm:a]{A}}]
As we have already noted, the square (\ref{eqn:basicsquare}) from
    \Cref{cons:basicfibersquare} is a pullback after inverting $p$; that is,
    the bottom square in (\ref{eqn:basicsquare2}) is a pullback. But the top
    square in (\ref{eqn:basicsquare2}) is a pullback by Theorem \ref{CMMthm};
    assembling these cartesian squares completes the proof of the theorem.
\end{proof}

\subsection{Fiber sequences up to quasi-isogeny}

Next, we review some definitions and terminology as in \cite{Beilinson},
identify more carefully the fiber terms in the above squares, and prove
Corollary~{\hyperref[quasiisogenythmB]{B}} from the introduction.

\begin{definition}[Isogenies and quasi-isogenies] 
\label{isogenies}
Given an additive category (or $\infty$-category) $\mathcal{C}$, we say that a
map $f\colon X \to Y$ is an \emph{isogeny} if there exists $g\colon Y \to X$ and an
integer $N > 0$
such that $g \circ f = N \mathrm{id}_X$ and $f \circ g = N \mathrm{id}_Y$. 
Let $\mathcal{C}$ be a stable $\infty$-category equipped with a $t$-structure 
which is left-complete.\footnote{Recall that $\mathcal{C}$ said to be
    left-complete (with respect to the given $t$-structure) if the natural map
$\Cscr\rightarrow\varprojlim_n\Cscr_{\leq n}$ is an equivalence. This is a technical
condition satisfied by many stable $\infty$-categories such as $\Sp$ and
$\Dscr(\ZZ)$.} We say that a map $f\colon X \to Y$ of bounded below objects is 
a \emph{quasi-isogeny} if the following equivalent conditions are satisfied: 
\begin{enumerate}
\item for each $n$, 
the map $\tau_{\leq n} f\colon \tau_{\leq n} X \to \tau_{\leq n }Y $ 
is an isogeny in $\mathcal{C}; $
\item 
for each $n$, the map 
$\pi_n X \to \pi_n Y$ in the heart $\mathcal{C}^{\heartsuit}$ is an isogeny. 
\end{enumerate}

\end{definition} 
We will need some elementary observations about quasi-isogenies. 
A map $f\colon X\rightarrow Y$ of bounded below objects in $\Cscr$ is a
quasi-isogeny if and only if the fiber $\mathrm{fib}(f)$ is quasi-isogenous to zero.
If one restricts to $\mathcal{C}_{\geq 0}$ (i.e., connective
objects), then quasi-isogenies are preserved under finite colimits and
geometric realizations (but generally not under filtered colimits). 
Next, let $\mathcal{C}, \mathcal{D}$ be stable $\infty$-categories with left-complete
$t$-structures. 
Given a right $t$-exact functor $F\colon\mathcal{C} \to \mathcal{D}$ (or just a right
bounded exact functor), 
it is easy to see that $F$ preserves quasi-isogenies.\footnote{A functor
$\Cscr\rightarrow\Dscr$ is right $t$-exact with respect to fixed
$t$-structures on $\Cscr$ and $\Dscr$ if it restricts to a functor
$\Cscr_{\geq 0}\rightarrow\Dscr_{\geq 0}$. It is right bounded if it
restricts to a functor $\Cscr_{\geq 0}\rightarrow\Dscr_{\geq n}$ for some $n\in\ZZ$.}

Given an $\infty$-category $\mathcal{I}$, we will say that a natural
transformation $f \to g$ of functors 
$f, g \colon \mathcal{I} \to \mathcal{C}$ 
is a \emph{quasi-isogeny} if it is a quasi-isogeny in $\fun(\mathcal{I}, \mathcal{C})$
with the pointwise $t$-structure. 
We will say that two functors are \emph{naturally quasi-isogenous} if they are
are related by a zig-zag of quasi-isogenies of functors.

\begin{example}
For $\Cscr = \Sp$, the map $\mathbb{S} \to \mathbb{Z}$ is a quasi-isogeny but
of course not an isogeny (as there is no nontrivial map back). In fact, in $\Sp$
one has the following formality result of Beilinson \cite{Beilinson}: every
bounded below spectrum $X$ is quasi-isogenous to the spectrum $\bigoplus_n
H\pi_n(X)[n]$. In particular, two bounded-below spectra $X$ and $Y$ are
quasi-isogenous precisely if for each $n$ separately the abelian groups $\pi_nX$
and $\pi_n Y$ are isogenous. To see that every spectrum is formal in the above
sense, it suffices to observe that every $k$-invariant of a
connective spectrum $X$ is
bounded torsion (where the torsion degree only depends on the degree of the
$k$-invariant and not on the specific homotopy groups). For explicit bounds,
cf.~\cite{Mat16}. 

This formality result of course does depend on choices and thus does not give
similar results in functor categories $\mathcal{C} = \fun(\mathcal{I} ,\Sp)$. 
\end{example}

The fiber sequence of Corollary~\ref{cor_ring} is the key to obtain our version of Beilinson's theorem 
\cite{Beilinson}, as follows. 

\begin{theorem}\label{thm_a_text}
For any associative ring $R$ the following spectra are naturally quasi-isogenous
to each other (i.e., related via a natural zig-zag of quasi-isogenies)
\[
\TC(R,(p);\ZZ_p) \qquad \qquad  \Sigma \HC(R, (p); \ZZ_p) \qquad \qquad \Sigma \HC(R;\ZZ_p)  \ . 
\]
Moreover,
    \begin{enumerate}
        \item[{\rm (a)}] 
if $R$ is $p$-torsion free, then the first two are equivalent after $(2p-5)$-truncation, and
\item[{\rm (b)}]
if $R$ is $p$-torsion free and\footnote{This is true pro-\'etale locally if $R$
is commutative, thanks to \cite[Th.~F]{HM}.} $\pi_{-1}( \TC(R; \mathbb{Z}_p)) = 0$, then the
first two are equivalent after $(2p-4)$-truncation. 
\end{enumerate}
\end{theorem}
\begin{proof}

For every associative ring $R$ we have the following commutative diagram of fiber
sequences 
\begin{equation}\label{diagram_big}
\begin{gathered}
\xymatrix{
\Sigma  \left(\THH(R;\ZZ_p) \otimes_\SS \mathbb{Z}_{hC_p}\right)_{hS^1}\ar[r]&\TC(R;\ZZ_p)\otimes_\SS\ZZ\ar[r]&\TC(R\otimes_\SS\FF_p;\ZZ_p)\\
F \ar[u] \ar[d] \ar[r] &  \mathrm{TC}(R; \mathbb{Z}_p) \ar[d]^{\mathrm{id}}  \ar[r]
\ar[u] &  \mathrm{TC}( R
\otimes_{\mathbb{S}} \mathbb{F}_p; \mathbb{Z}_p) \ar[d] \ar[u]^{\mathrm{id}} \\
\TC(R,(p); \ZZ_p)\ar[r]&\TC(R; \ZZ_p)\ar[r]&\TC(R/p; \ZZ_p).
}
\end{gathered}
\end{equation}
To form the above diagram, 
we use the map  $\mathrm{TC}(R; \mathbb{Z}_p) \to \mathrm{TC}(R; \mathbb{Z}_p)
\otimes_{\mathbb{S}} \mathbb{Z}$
induced from the map $\SS \to \ZZ$ as well as the map on $\mathrm{TC}(-;
\mathbb{Z}_p)$ induced by the  Postnikov section $R \otimes_\SS \FF_p \to R/p$. All horizontal sequences in
\eqref{diagram_big} are fiber sequences, either by \Cref{cor_ring} or by
definition; that is, $F$ is defined as the fiber of $\mathrm{TC}(R;
\mathbb{Z}_p) \to \mathrm{TC}( R \otimes_{\mathbb{S}} \mathbb{F}_p;
\mathbb{Z}_p)$.

We claim now that all the vertical maps in diagram \eqref{diagram_big}  are quasi-isogenies. 

\begin{lemma} 
The map 
$\mathrm{TC}(R; \mathbb{Z}_p) \to \mathrm{TC}(R; \mathbb{Z}_p)
\otimes_{\mathbb{S}} \mathbb{Z}$ in the 
diagram \eqref{diagram_big} is a natural quasi-isogeny of spectra.
Moreover, its fiber is  $(2p-4)$-connective.
If $\pi_{-1}(\TC(R; \mathbb{Z}_p)) = 0$, then the fiber is $(2p-3)$-connective.
\end{lemma} 
\begin{proof} 
The first part follows from the observation that tensoring a quasi-isogeny (in this case $\SS \to \ZZ$) with a bounded below spectrum (here $\TC(R;\ZZ_p)$) is again a quasi-isogeny. 
Moreover, the fiber of $\TC(R; \ZZ_p) \to \TC(R; \ZZ_p)\otimes_\SS \ZZ$ is $(2p-4)$
connective since $\TC(R;\ZZ_p)$ is $(-1)$-connective and the fiber of $\SS_{(p)} \to \ZZ_{(p)}$ is $(2p-3)$ connective. 
The last assertion follows similarly. 
\end{proof} 

The right horizontal map $\TC(R\otimes_\SS\FF_p;\ZZ_p)\rightarrow\TC(R/p;\ZZ_p)$ in diagram \eqref{diagram_big} is also a quasi-isogeny. This
follows from \Cref{isogenyTHH} that we will discuss and prove in
Section \ref{sec_thmb} and which is purely internal to cyclotomic spectra. But
we also want to give a direct proof here using $\K$-theory and the
Dundas--Goodwillie--McCarthy theorem. 

\begin{proposition}\label{isogeny_GDM}
The natural map $\TC(R \otimes_\SS \FF_p;\ZZ_p) \to \TC(R/p; \ZZ_p)$ is a quasi-isogeny. If $R$ is $p$-torsion free then the fiber is $(2p-1)$-connective.
\end{proposition}
\begin{proof}
The map of connective ring spectra $R \otimes_\SS \FF_p \to R/p$ is an
isomorphism on $\pi_0$. Thus the Dundas--Goodwillie--McCarthy theorem (for ring spectra) implies that its fiber is equivalent to the fiber of the map
\[
\K(R \otimes_\SS \FF_p;\ZZ_p) \to \K(R/p; \ZZ_p)
\]
But the map $R \otimes_\SS \FF_p \to R/p$ of ring spectra is a quasi-isogeny
and, if $R$ is $p$-torsion free, has fiber which is $(2p-2)$-connective. Thus, the
map on $\K$-theory is a quasi-isogeny and has fiber which is $(2p-1)$-connective,
cf.~\cite[Prop. 2.19]{LandTamme} (the proof in \emph{loc.~cit.} shows that the map is truly a
quasi-isogeny of functors).  
\end{proof}

Now we know that the vertical maps in diagram \eqref{diagram_big} are quasi-isogenies, so we conclude that $\Sigma  \left(\THH(R,\ZZ_p) \otimes_\SS \mathbb{Z}_{hC_p}\right)_{hS^1}$
and $\TC(R, (p); \ZZ_p)$ are quasi-isogenous to one another. Moreover, if $R$ is
$p$-torsion free, then the vertical maps from $F$ have $(2p-4)$-connective
fibers by the above discussion (which upgrades to $(2p-3)$-connective fibers if
$\pi_{-1} \TC(R; \mathbb{Z}_p) = 0$). Thus, we conclude that  $\Sigma  \left(\THH(R,\ZZ_p) \otimes_\SS \mathbb{Z}_{hC_p}\right)_{hS^1}$
and $\TC(R, (p); \ZZ_p)$ are equivalent in degrees $\leq (2p-5)$, and in
degrees $\leq 2p-4$ if $\pi_{-1} \TC(R; \mathbb{Z}_p) = 0$. 
\Cref{thm_a_text} now follows from the arguments above and the following lemma. 
\end{proof} 

\begin{lemma}
\label{quasiisog2}
The following spectra are naturally quasi-isogenous to each other
\[
\left(\THH(R;\ZZ_p) \otimes_\SS \mathbb{Z}_{hC_p}\right)_{hS^1} \qquad \qquad \HC(R, (p); \ZZ_p) \qquad \qquad \HC(R;\ZZ_p)
\]
and the first two are equivalent after $(2p-4)$-truncation if $R$ is $p$-torsion free.
\end{lemma}
\begin{proof}
We have that $\THH(R;\ZZ_p) \otimes_\SS \mathbb{Z}_{hC_p}$ is equivalent to the fiber of 
\[
\THH(R; \ZZ_p)\otimes_\SS \ZZ \to \THH(R \otimes_\SS \FF_p; \ZZ_p) 
\]
and this map sits in a commutative square 
\[
\xymatrix{
\THH(R; \ZZ_p)\otimes_\SS \ZZ \ar[r] \ar[d]&  \THH(R \otimes_\SS \FF_p; \ZZ_p)  \ar[d]\\
\HH(R; \ZZ_p) \ar[r] & \HH(R/p; \ZZ_p)
}
\]
of spectra with $S^1$-action. Both vertical maps are quasi-isogenies, so that we
get the desired quasi-isogeny between the first two terms of the statement by
taking $S^1$-orbits. The term $\HH(R/p)$ is 
quasi-isogenous to $0$, so that we get the last quasi-isogeny too. 
If $R$ is $p$-torsion free, the fibers of the left and right vertical maps are in
degrees $\geq 2p-3$ and $\geq 2p-2$, respectively, so the last assertion follows
too. 
\end{proof}

Corollary~{\hyperref[quasiisogenythmB]{B}} follows by
combining Theorem \ref{CMMthm} with Theorem \ref{thm_a_text}.  In particular, we
have an isomorphism $\K_*(R, (p);\ZZ_p) \iso  \HC_{*-1}(R, (p);\ZZ_p)$ for
$\ast \leq 2p-5$, for $R$ commutative, $p$-torsion free, and henselian along
$(p)$.
Note also that 
with the same proof,  we can deduce the following variant of Theorem
\ref{thm_a_text} for arbitrary connective $\mathbb{Z}$-algebra ring spectra
(also known as $\ZZ$-linear dgas).

\begin{proposition}
    If $R$ is a connective $\ZZ$-algebra spectrum, then the fiber of $\TC(R;\ZZ_p)
\to \TC(R \otimes_\ZZ \FF_p; \ZZ_p)$ is quasi-isogenous to $\Sigma \HC(R;
\mathbb{Z}_p)$ and after $(2p-5)$-truncation equivalent to the fiber of
\[
\Sigma \HC(R;\ZZ_p) \to \Sigma \HC(R \otimes_\ZZ \FF_p; \ZZ_p) \ .
\]
\end{proposition}

\begin{remark} 
In all of the above, the denominators involved in the above quasi-isogenies are
uniform: they do not depend on the choice of $R$. More formally, one could state
all of the above quasi-isogenies via the $\infty$-category of functors from
rings $R$ to spectra. The denominators in the next result are not independent in the same fashion. 
\end{remark} 

\begin{theorem} 
Let $(R, I)$ be a pair consisting of an associative ring $R$ and a nilpotent ideal $I$. 
Then there is a natural zig-zag of quasi-isogenies between $\K(R, I;
\mathbb{Z}_p)$ and $\Sigma \HC(R, I; \mathbb{Z}_p)$. 
\end{theorem} 
\begin{proof} 
By the Dundas--Goodwillie--McCarthy theorem, we can replace $\K$-theory with $\TC$. 
We have a natural map 
    \begin{equation} \label{modIqiso}  \TC(R, I;\mathbb{Z}_p) \to \mathrm{fib}\left( \TC(R, (p); \mathbb{Z}_p) \to \TC(R/I, (p);
\mathbb{Z}_p)\right)  \end{equation}
Now $\TC(R/p; \mathbb{Z}_p) \to \TC(R/(I, p); \mathbb{Z}_p)$
is a quasi-isogeny in view of \Cref{isogenyTHH} below, 
so that \eqref{modIqiso} is a quasi-isogeny. Combining with the quasi-isogenies
of \Cref{thm_a_text} now completes the proof. 
\end{proof}

\section{Quasi-isogenies of cyclotomic spectra}\label{sec_thmb}

In this section, we systematically study quasi-isogenies in  cyclotomic spectra,
give another
proof of Theorem~{\hyperref[thm:a]{A}} and
Corollary~{\hyperref[quasiisogenythmB]{B}}, and prove Theorem~{\hyperref[thm:c]{C}},
sharpening some
results of Geisser--Hesselholt \cite{GHfinite}.

\subsection{Preliminaries}
We will apply the notion of quasi-isogeny (\Cref{isogenies}) to the
$\infty$-category $\cycsp$ of  cyclotomic
spectra using the $t$-structure of~\cite{AN18};\footnote{Recall that what we
denote by $\cycsp$ is denoted $\cycsp_p$ in \cite{AN18}.} 
this $t$-structure is defined so that the connective objects of $\cycsp$ are
those whose underlying spectrum is connective and it is checked
in~\cite[Theorem~2.1]{AN18} that the $t$-structure is left-complete. 
Note that a quasi-isogeny of bounded-below cyclotomic spectra $f\colon X \to Y$ is a quasi-isogeny of
underlying spectra, and $\TC(f;\ZZ_p)\colon \TC(X;\ZZ_p) \to \TC(Y;\ZZ_p)$ is also a quasi-isogeny. However,  $\THH(\mathbb{F}_p)
\in \cycsp$ has underlying spectrum quasi-isogenous to zero, but is not itself
quasi-isogenous to zero because $\TC(\mathbb{F}_p)\we\TC(\FF_p;\ZZ_p)$ is torsion free and
nonzero: $\pi_0\TC(\FF_p;\ZZ_p)\iso\pi_{-1}\TC(\FF_p;\ZZ_p)\iso\ZZ_p$.

\newcommand{\TR}{\mathrm{TR}}
In the next result, we use the notion of $\TR$ of a cyclotomic spectrum, which
plays an important role in the work \cite{AN18}. See \cite{BM15} for an account
of $\TR$ in the approach to cyclotomic spectra via genuine equivariant homotopy theory. 
Implicitly, $\TR$ is computed with respect to our fixed prime $p$, but it will
not generally be $p$-complete unless we $p$-complete it forming $\TR(X;\ZZ_p)$ for a
cyclotomic spectrum $X$.

\begin{proposition}\label{prop:tr}
A map $f\colon X \to Y$ of bounded-below cyclotomic spectra is a
quasi-isogeny in $\cycsp$
if and only if the map of spectra $\TR(f)\colon \TR(X) \to \TR(Y)$ is a quasi-isogeny
of spectra.
\end{proposition} 
\begin{proof} 
    This follows from the description of 
    the cyclotomic $t$-structure of \cite{AN18}. In particular, the cyclotomic homotopy groups
    of $X \in \cycsp$ are precisely the homotopy groups of the spectrum $\TR(X)$,
    together with the Frobenius and Verschiebung maps. These are $p$-typical
    Cartier modules, which is to say abelian groups $M$ with endomorphisms $F$ and $V$
    satisfying $FV=p$,
    and the heart $\cycsp^\heartsuit$ is equivalent to a full subcategory of the
    category of $p$-typical Cartier modules. Now, it suffices to check that a
    map $h\colon W\rightarrow Z$ between $p$-typical Cartier
    modules is an isogeny precisely if the underlying map of abelian groups is
    an isogeny. One implication is immediate and for the other suppose
    $g\colon Z\rightarrow W$ is a map of abelian groups such that $g\circ h=N\mathrm{id}_W$ and $h\circ
    g=N\mathrm{id}_Z$. The map $g$ might not respect the $F$ and $V$ maps.
    However, for $z\in Z$, $F(g(z))-g(F(z))$ and $V(g(z))-g(F(z))$ are both in
    the kernel of $h$, which consists of $N$-torsion elements of $W$.
    Therefore, $Ng$ is a map of $p$-typical Cartier modules, $Ng\circ
    h=N^2\mathrm{id}_W$, and $h\circ Ng =N^2\mathrm{id}_Z$.
\end{proof} 

\begin{proposition} \label{zero_Frob}
Let $X \in \cycsp$ be a cyclotomic spectrum such that $X$ is bounded-below, such
that the
Frobenius $\varphi\colon X \to X^{tC_p}$ is nullhomotopic in $\fun(BS^1, \sp)$, and 
such that $X$ is quasi-isogenous to zero as a spectrum. 
Then $X$ is quasi-isogenous to zero as a cyclotomic spectrum. 
\end{proposition}

\begin{proof} 
The assumption that the Frobenius is nullhomotopic implies that $\mathrm{TR}(X)$
is a product $\prod_{n \geq 0} X_{hC_{p^n}}$, using the description of $\TR$ as an
iterated pullback, cf.~\cite[Remark 2.5]{AN18} and \cite[Corollary II.4.7]{NS18}. 
The assumption that $X$ is quasi-isogenous to zero now implies that the
above product 
    is also quasi-isogenous to zero, so we conclude by
    Proposition~\ref{prop:tr}. 
\end{proof} 

We observe that the theory of cyclotomic spectra 
admits a natural graded variant. 
A \emph{graded spectrum} is an object of the functor category $\fun(\mathbb{Z}_{\geq 0}^{\mathrm{ds}},
\sp)$ where $\mathbb{Z}_{\geq 0}^{\mathrm{ds}}$ denotes the discrete category
of nonnegative integers with no non-identity morphisms; given a graded spectrum $X$ we let $X_i \in
\sp, i
\geq 0$ denote the $i$th graded piece. 
We let $\mathrm{GrSp}$ denote the $\infty$-category of graded spectra, which we
consider as a symmetric monoidal $\infty$-category under Day convolution using
the multiplication symmetric monoidal structure on $\mathbb{Z}_{\geq
0}^{\mathrm{ds}}$. 
A \emph{graded cyclotomic spectrum} $X$ consists of a 
graded 
spectrum $X = \left\{X_i\right\}$ equipped with an $S^1$-action together with a
family of $S^1$-equivariant maps 
$\varphi_i\colon X_i \to X_{pi}^{tC_p}$ for $i \geq 0$.  
We let $\grcyc$ denote the $\infty$-category of graded cyclotomic spectra. 
Any graded cyclotomic spectrum $X = \left\{X_i\right\}$ has an
underlying cyclotomic spectrum 
$\bigoplus_{i \geq 0} X_i$, and this defines a forgetful functor $\grcyc \to
\cycsp$. 

More formally, the $\infty$-category $\grcyc$ is defined  as follows.  We consider the $\infty$-category 
$\fun(BS^1, \mathrm{GrSp})$ 
of graded spectra equipped with an $S^1$-action. 
This admits a natural endofunctor $F$
which sends $\left\{X_i, i \geq 0\right\}$ to $\left\{X_{pi}^{tC_p}\right\}$, 
where we regard $X_{pi}^{tC_p}$ as a spectrum with an $S^1/C_p \simeq
S^1$-action. Then $\grcyc$ is defined as the $\infty$-category of
$F$-coalgebras, as in \cite[Section~II.5]{NS18}. 

Given a graded ring spectrum $R$, there is a graded  cyclotomic spectrum 
$\THH(R)$ 
obtained by applying the cyclic bar construction in the category of graded
spectra. This refines the usual $\THH$ and admits an $S^1$-action in graded
spectra. See \Cref{app:gradedcyc} for the details of this construction. Compare also \cite{Brunfiltered} for a treatment of filtered cyclotomic spectra and
filtered $\TC$ using more classical methods.

\begin{proposition} 
\label{gradedisogeny}
Let $X$ be a graded  cyclotomic spectrum.
If
    \begin{enumerate}
        \item[{\rm (1)}] the underlying spectrum of $X$ is quasi-isogenous to zero,
        \item[{\rm (2)}] the graded piece $X_0 $ is contractible, and
        \item[{\rm (3)}] the connectivity of the pieces $X_i$ tends to infinity in $i$,
\end{enumerate}
then $X$ is quasi-isogenous to zero as an object of $\cycsp$. 
\end{proposition} 
\begin{proof} 
Given a graded cyclotomic spectrum $X$, for each $i$, we can
construct a graded 
cyclotomic spectrum $X_{\leq i} \in \grcyc$  such that $(X_{\leq i})_j = 0$ for
$j > i$ and $(X_{\leq i})_j = X_i$ for $j \leq i$ and a tower of maps 
$X \to \dots \to X_{\leq n} \to X_{\leq n-1} \to \dots \to X_{\leq 1}$. 
This is a tower in $\grcyc$, and we can consider it as a tower of underlying
objects in $\cycsp$ too. 

We need to show that for each $j$, $\pi_j( \TR(X))$ is isogenous to zero. 
Our assumptions imply that $\pi_j(\TR(X)) \to \pi_j  \TR( X_{\leq n})$ is an
isomorphism for $n \gg 0$.
However, the object $\mathrm{fib}( X_{\leq i} \to X_{\leq i-1})$ defines a
cyclotomic spectrum with Frobenius homotopic to zero, in view of the grading. 
It follows from \Cref{zero_Frob} that 
$\TR( \mathrm{fib}(X_{\leq i} \to X_{\leq i-1}))$ is quasi-isogenous to zero, and by
induction $\TR( X_{\leq n})$ is quasi-isogenous to zero. Putting these observations
    together with Proposition~\ref{prop:tr} completes the proof. 
\end{proof}

\subsection{Quasi-isogenies on $\THH$}
Our main result here is the following, which restates
Theorem~\hyperref[thm:c]{C}.
On  $\TC$ and for discrete rings in which $p$ is nilpotent, it is due to Geisser--Hesselholt
\cite{GHfinite}, and the main arguments are based on theirs. 

\begin{theorem} 
\label{isogenyTHH}
Let $f\colon A \to A'$ be a map of connective associative ring spectra. 
If
\begin{enumerate}
    \item[{\rm (i)}] $f$ is a quasi-isogeny of spectra and
    \item[{\rm (ii)}] the map $\pi_0(f)\colon \pi_0(A) \to \pi_0(A')$ is surjective with nilpotent
kernel,
\end{enumerate}
then $\THH(A) \to \THH(A')$ is a quasi-isogeny in $\cycsp$. 
\end{theorem}

 We will first verify some special cases. 

\begin{proposition} 
\label{THHgrisog}
Let $R$ be a connective associative \emph{graded} ring spectrum. 
If
\begin{enumerate}
    \item[{\rm (a)}] each $R_i, i > 0$, is isogenous to zero as a spectrum and
    \item[{\rm (b)}] the connectivity of the $R_i$ tends to $\infty$  as $i \to \infty$,
\end{enumerate}
then the map $\THH(R) \to \THH(R_0)$ is a quasi-isogeny in
$\cycsp$.  \end{proposition} 
\begin{proof} 
Since $R$ is a graded ring spectrum, $\THH(R)$ admits the structure
of a graded cyclotomic spectrum (refining the usual cyclotomic structure on
$\THH(R)$), and in degree zero one has $\THH(R_0)$. 
For this, compare \Cref{app:gradedcyc}, or the work of Brun \cite{Brunfiltered}, who uses 
the more classical approach to cyclotomic spectra. 

Now we wish to apply \Cref{gradedisogeny}. 
Consider the subcategory $\mathcal{C} \subseteq \mathrm{GrSp}_{\geq 0}$ of \emph{connective} graded spectra 
spanned by graded spectra $Z$ such that $Z_i$ is quasi-isogenous to zero for $i
>0$ and such that the connectivity of $Z_i$ grows without bound as $i \to
\infty$. Then $\mathcal{C}$ is closed under tensor products and geometric
realizations. 
The assumptions on $R$ imply that $R \in \mathcal{C}$, and
consequently 
$\THH(R) \in \mathcal{C}$ as well. That is, 
$\THH(R)_i$ is quasi-isogenous to zero for $i
>0$ and the connectivity of $\THH(R)_i$ grows without bound as $i \to
\infty$. 
Thus, we can apply 
\Cref{gradedisogeny}. 
\end{proof}

\begin{proposition} 
\label{nontrivsqzero}
Let $A$ be a connective associative ring spectrum. 
Let $M$ be a connective $(A, A)$-bimodule which is quasi-isogenous to zero. 
Suppose $\widetilde{A}$ is a square-zero extension of $A$ by $M$, in the sense
that 
one has a map 
$f\colon A \to A \oplus M[1]$ in $\alg_{/A}$ and a pullback diagram
\[ \xymatrix{
\widetilde{A} \ar[d]  \ar[r] &  A \ar[d]^0  \\
A \ar[r]^-{f} & A \oplus M[1].
}\]
Then the map $\THH(\widetilde{A}) \to \THH(A)$ is a quasi-isogeny in $\cycsp$. 
\end{proposition} 
\begin{proof} 
We can form the \v{C}ech nerve of the map $\widetilde{A} \to A$, i.e., the
simplicial object
$ \dots \widetilde{A} \times_A \widetilde{A} \rightrightarrows \widetilde{A}$. 
This yields a simplicial object $X_\bullet$ of $\alg$ which resolves $A$. 
It follows that $|\THH(X_\bullet)| \simeq \THH(A)$ in $\cycsp$. 

Now $\widetilde{A} \times_A \widetilde{A}$ is a \emph{trivial} square-zero extension of
$\widetilde{A}$ by $M$. It follows 
that $\THH(\widetilde{A} \times_A \widetilde{A})$ is quasi-isogenous to
$\THH(\widetilde{A})$ by \Cref{THHgrisog}, since a trivial square-zero extension can be
given a grading. Continuing in this way, it follows that all the 
maps in the simplicial object $\THH(X_\bullet)$ are quasi-isogenies. 
Taking geometric realizations now, it follows that
$\THH(\widetilde{A}) \to |\THH(X_\bullet)| \simeq \THH(A)$ is a quasi-isogeny in
$\cycsp$. 
\end{proof}

\begin{proposition} 
\label{augmentedcase}
Let $B$ be a connective associative ring spectrum and let 
$B'$ be 
an object of $\alg_{B//B}$. Suppose that 
the augmentation map $B' \to B $ is a quasi-isogeny  and the map $\pi_0(B') \to
\pi_0(B)$ has nilpotent kernel. Then $\THH(B) \to \THH(B')$ is a quasi-isogeny.
\end{proposition} 
\begin{proof} 
Recall first that $\alg_{B//B}$ is equivalent to the $\infty$-category of
\emph{nonunital} associative algebra objects in $(B, B)$-bimodules. 
In particular, $I = \mathrm{fib}(B' \to B)$ has such a structure.
We can work up the Postnikov tower $\tau_{\leq\star} I$; since
$\TR$ behaves well with respect to Postnikov towers, it suffices to prove the
result for each $\tau_{\leq n} I$. 
General results as in \cite[Section~7.4.1]{HA} (which go back at least to
\cite{Basterra99}) 
now show that $\tau_{\leq n} I$ can be obtained in finitely
many steps via square-zero extensions from $B$, by bimodules which 
are quasi-isogenous to zero.
Now we conclude via \Cref{nontrivsqzero}. 
\end{proof} 

\begin{proof}[Proof of Theorem~\ref{isogenyTHH}] 
We consider the \v{C}ech nerve of $A \to A'$. 
We obtain a simplicial object $X_\bullet$ in $\alg_{\geq 0}$ 
such that $|X_\bullet| \simeq A'$ and such that each $X_i$ is an iterated fiber
product of copies of $A$ over $A'$. 
Each $X_i$ can be given the structure of an object of $\alg_{A//A}$ 
(via appropriate face and degeneracy maps), 
whence we conclude by \Cref{augmentedcase} that all the maps in the simplicial object
$\THH(X_\bullet)$ are quasi-isogenies in $\cycsp$. Finally, the result now
follows by taking geometric realizations. 
\end{proof} 

One important corollary of Theorem~\ref{isogenyTHH} is the following result
of Geisser and Hesselholt; see \cite{LandTamme} for generalizations.

\begin{corollary}[Geisser--Hesselholt \cite{GHfinite}] \label{GHfiniteresult}
    If $p$ is nilpotent in $A$ and $I\subseteq A$ is a two-sided
    nilpotent ideal, then $\K(A,I) \simeq \TC(A, I)$ is quasi-isogenous to zero.
\end{corollary}

\begin{proof}
    In this case, $A\rightarrow A/I$ is a quasi-isogeny (they are both
    quasi-isogenous to zero) and hence Theorem~\ref{isogenyTHH} applies to
    prove that $\THH(A)\rightarrow\THH(A/I)$ is a quasi-isogeny of cyclotomic
    spectra. This implies in particular that $\TC(A,I)$ is quasi-isogenous to
    zero. \end{proof}

\subsection{Quasi-isogenies and the Beilinson fiber sequence}

Recall that Proposition \ref{isogeny_GDM}, which was key in proving Theorem \hyperref[thm:a]{A}, asserts that for every ring $R$ the induced map 
 $\TC(R \otimes_\SS \FF_p;\ZZ_p) \to \TC(R/p; \ZZ_p)$ is a quasi-isogeny. Our
 proof in Section \ref{sec:square} relied on $\K$-theory. Alternatively we can now
 also deduce this fact directly from Theorem~\ref{isogenyTHH}, which implies
 that $\THH(R \otimes_\SS \FF_p) \to \THH(R/p)$ is a quasi-isogeny of cyclotomic
 spectra and therefore $\TC(R \otimes_\SS \FF_p;\ZZ_p) \to \TC(R/p; \ZZ_p)$ is a
 quasi-isogeny of spectra. 
 This immediately implies also the following variant of 
Theorem \hyperref[thm:a]{A}. 

\begin{corollary} \label{TCformThmA}
Let $R$ be a ring. Then there is a natural pullback square
\begin{equation}
    \begin{gathered}
\xymatrix{
\TC(R; \mathbb{Q}_p) \ar[d]  \ar[r] & \TC(R/p; \mathbb{Q}_p) \ar[d]  \\
\HC^-(R; \mathbb{Q}_p) \ar[r] &  \HP(R; \mathbb{Q}_p). }
    \end{gathered}
\end{equation} 
\end{corollary}

 In this section we want to take this a step further and prove a cyclotomic version of Theorem \hyperref[thm:a]{A}.

\begin{theorem}\label{thm_cyc}
For every ring $R$, the following  cyclotomic spectra are quasi-isogenous to each other
\[
\THH\left(R, (p);\ZZ_p\right) \qquad\qquad  \HH(R;\ZZ_p)  \qquad \qquad  \HH(R,
    (p);\ZZ_p)
\]
where  $\HH(R;\mathbb{Z}_p)$ and $\HH(R, (p);\mathbb{Z}_p)$ are equipped with the canonical $S^1$-actions and  the zero Frobenius (see Example \ref{ex_cyc}). 
Moreover if $R$ is $p$-torsion free then we have an equivalence of cyclotomic spectra
\[
\tau^{\mathrm{cyc}}_{\leq (2p -4)} \THH\left(R, (p);\ZZ_p\right) \quad
    \simeq \quad \tau^{\mathrm{cyc}}_{\leq (2p - 4)}  \HH(R, (p);\ZZ_p).
\]

\end{theorem}

We note that the last theorem immediately implies Theorem \hyperref[thm:a]{A} by passing to $\TC(-; \ZZ_p)$ since $\TC(-;\ZZ_p)$ of a cyclotomic spectrum with zero Frobenius is just given by the $p$-completion of the shifted $S^1$-orbits. But Theorem \ref{thm_cyc} is strictly stronger than Theorem \hyperref[thm:a]{A} since quasi-isogenies cannot be detected on $\TC$. The remainder of this section is devoted to proving Theorem~\ref{thm_cyc}. 
\begin{lemma}\label{lem:tensor}
    Suppose that $X\rightarrow Y$ is a quasi-isogeny of  cyclotomic spectra and
    that $M$ is any bounded below cyclotomic spectrum. Then, $X\otimes_\SS
    M\rightarrow Y\otimes_\SS M$ is a quasi-isogeny of cyclotomic spectra.
\end{lemma}

\begin{proof}
    We can assume $M$ is connective, in which case the functor $-\otimes_\SS M$ is
    a right $t$-exact endofunctor of $\cycsp$ and hence preserves quasi-isogenies.
\end{proof}

We now want to apply a similar proof-strategy as in Section \ref{sec:square} and consider the diagram
$$\xymatrix{
\THH(R;
\mathbb{Z}_p)\otimes_\SS\ZZ_{hC_p}\ar[r]&\THH(R;
\mathbb{Z}_p)\otimes_\SS\ZZ^\triv\ar[r]&\THH(R\otimes_\SS\FF_p;
\mathbb{Z}_p)\ar[d]\\
\THH(R, (p); \mathbb{Z}_p)\ar[r]&\THH(R; \mathbb{Z}_p)\ar[u]\ar[r]&\THH(R/p;
\mathbb{Z}_p),}$$
of cyclotomic spectra in which the horizontal rows are fiber sequences. 
Both vertical maps are quasi-isogenies in cyclotomic spectra: the first
by Lemma~\ref{lem:tensor} and because $\SS^\triv\rightarrow\ZZ^\triv$ is a quasi-isogeny (since $(-)^\triv$
is right $t$-exact) and the second by Theorem~\ref{isogenyTHH}.
The right vertical map has homotopy fiber in degrees $\geq 2p-2$, while the
middle vertical map has homotopy fiber in degrees $\geq 2p-3$. 
It follows that $\THH(R, (p); \ZZ_p)$ is quasi-isogenous in cyclotomic spectra to
$\THH(R; \mathbb{Z}_p)\otimes_\SS\ZZ_{hC_p}$, and their cyclotomic
$(2p-4)$-truncations $\tau_{\leq 2p-4}^{\mathrm{cyc}}$ are equivalent. Note that the cyclotomic
Frobenius on $\THH(R;\ZZ_p)\otimes_\SS\ZZ_{hC_p}$ is nullhomotopic by the Tate
orbit lemma.

Now the next lemma finishes the proof of Theorem \ref{thm_cyc}.

\begin{lemma}
The following cyclotomic spectra are quasi-isogenous to each other
\[
\THH(R; \mathbb{Z}_p)\otimes_\SS\ZZ_{hC_p} \qquad \qquad \HH(R;\ZZ_p)
    \qquad\qquad \HH(R, (p);\ZZ_p).
\]
Moreover, the cyclotomic truncations $\tau_{\leq 2p-4}^{\mathrm{cyc}}$ of the
first and the third are
naturally equivalent. 
\end{lemma}
\begin{proof}
We consider the square
\[
\xymatrix{
\THH(R; \ZZ_p)\otimes_\SS \ZZ \ar[r] \ar[d]&  \THH(R \otimes_\SS \FF_p; \ZZ_p)  \ar[d]\\
\HH(R; \ZZ_p) \ar[r] & \HH(R/p; \ZZ_p)
}
\]
of spectra with $S^1$-action, in which the vertical maps are quasi-isogenies
    and the right hand terms are quasi-isogenous to zero. We consider it as a
    square of cyclotomic spectra by equipping all spectra with the zero
    Frobenius map. It follows from Proposition \ref{zero_Frob} that the
    vertical maps are quasi-isogenies of cyclotomic spectra. The horizontal fibers are
    equivalent to $\THH(R;\ZZ_p)\otimes_\SS\ZZ_{hC_p}$ and $\HH(R, (p);\ZZ_p)$,
    which shows they are quasi-isogenous to each other.

Finally, the induced map of cyclotomic spectra
(with zero Frobenii)
$ \THH(R;\ZZ_p)\otimes_\SS\ZZ_{hC_p} \to \HH(R, (p);\ZZ_p)$
has the property that it is an equivalence of underlying spectra in degrees $\leq 2p-4$
(as in \Cref{quasiisog2}) 
and
consequently induces an equivalence on 
cyclotomic homotopy groups in degrees $\leq 2p-4$, e.g., using the description
of $\TR$ as in the proof of~\Cref{zero_Frob}. 
\end{proof}

\newcommand{\CycSp}{\mathrm{CycSp}}
\newcommand{\Sscr}{\mathcal{S}}

\section{Application to $p$-adic deformations}

\newcommand{\X}{\mathfrak{X}}
In this section, we prove
Theorems~\hyperref[thm:d]{D} and~\hyperref[thm:e]{E}.
Throughout this section, let $\mathfrak{X}$ be a quasi-compact and
quasi-separated (qcqs) $p$-adic formal scheme  with bounded
$p$-power torsion, and write $\mathfrak{X}_n=\mathfrak{X}\times_{\Spec\ZZ_{p}}\Spec\ZZ/p^n$. 
We are interested in the following invariants of $\X$, and in
particular the $p$-adic deformation problem (\Cref{imageofctsKq} below).  \begin{definition}[Continuous invariants of formal schemes] 

\label{continuousinv}
Let $F$ be an invariant of schemes (such as $\K, \THH, \HH, 
\HC^-, \HP, \TC$).
Given the formal scheme $\X$, we define
$F^{\mathrm{cts}}(\mathfrak{X})$  via
\begin{equation} 
F^{\mathrm{cts}}( \mathfrak{X}) = \varprojlim_n F(\mathfrak{X}_n)
.
\end{equation} 
\end{definition} 

 If the $p$-adic formal scheme $\X$ arises as the $p$-adic
completion of a scheme $X$, we have a natural comparison map 
\begin{equation}  \label{contcompmap} F( X) \to F^{\mathrm{cts}}( \X) .  \end{equation}

\begin{proposition} 
\label{contcompmapeq}
Suppose $\X$ is the $p$-adic completion of a qcqs scheme $X$ with bounded $p$-power
torsion. Then the maps 
\eqref{contcompmap} for $F = \HH, \THH, \HC^-, \HP, \TC$ are $p$-adic equivalences. 
\end{proposition} 
\begin{proof} 
Using Zariski descent on $X$, we may assume that $X = \spec(R)$ where $R$ is a
ring of bounded $p$-power torsion, and then $\X = \spf( \hat{R}_p)$. 
Using the cyclic bar construction, 
it is not difficult to show that 
$\THH^{\mathrm{cts}}( \X; \mathbb{Z}_p) = \THH(R; \mathbb{Z}_p)$, i.e., that 
\eqref{contcompmap} is a $p$-adic equivalence for $F = \THH$
(cf.~the
proof of \cite[Theorem 5.19]{CMM}). 
Tensoring over $\THH(\mathbb{Z})$ with $\mathbb{Z}$, one deduces the result for
$\HH$, and then
taking $S^1$-invariants and coinvariants, we find that 
\eqref{contcompmap} is a $p$-adic equivalence for $F = \HC^-, \HP$. 
Running the above argument with $\THH$ instead of $\HH$, one concludes 
that \eqref{contcompmap} is a $p$-adic equivalence for $F = \TC$
\cite[Theorem 5.19]{CMM}. See also \cite[Cor.~4.8]{Morrow_Dundas} for these results,
when $R$ is assumed noetherian and $F$-finite. 
\end{proof}

By contrast, it is much more difficult to control \eqref{contcompmap} when $F =
\K$. We mention the two following cases. 

\begin{example}[Formal affine schemes] 
Suppose $\mathfrak{X}$ is affine, i.e., 
$\mathfrak{X} = \spf(R)$, for $R$ a \emph{$p$-adically complete ring} with bounded
$p$-power torsion. We can then write $\mathfrak{X}$ as the $p$-adic completion
(as a formal scheme) of $X = \spec(R)$. 
In this case, 
the comparison map 
\eqref{contcompmap} is a $p$-adic equivalence for $F = \K$ as well, cf.~
\cite[Theorem 5.23]{CMM} and \cite[Theorem C]{GH06}. 
\end{example}
\begin{example}[Proper schemes]
Suppose $R$ is  $p$-complete. 
Suppose $\mathfrak{X}$ is the $p$-completion of a proper scheme $X
\to \spec(R)$. The map $\K(X; \mathbb{Z}_p) \to \K^{\mathrm{cts}}(\mathfrak{X};
\mathbb{Z}_p)$ is probably not an equivalence; compare \cite[App. B]{BEK14} for
a related counterexample in equal characteristic zero. 
In this case, $\K^{\mathrm{cts}}( \mathfrak{X}; \mathbb{Z}_p)$ is generally much
more tractable than $\K(X; \mathbb{Z}_p)$ via comparisons with topological
cyclic homology. 
\end{example}

\begin{question}[The $p$-adic deformation problem] 
Let $\mathfrak{X}$ be a $p$-adic formal scheme with special fiber $\mathfrak{X}_1$ as above.
\label{imageofctsKq}
For $i \geq 0$,
what is the image\footnote{By the Milnor exact sequence, this is equivalent to describing the image of the map 
$\left(\varprojlim_n \K_i(\mathfrak{X}_n))\right)_{\mathbb{Q}} \to 
\K_i(\mathfrak{X}_1; {\mathbb{Q}})$.} 
of the map
\begin{equation}
\label{bekmap1}
\K_i^{\mathrm{cts}}(\mathfrak{X}; \QQ)\rightarrow\K_i(\mathfrak{X}_1; \QQ) ?
\end{equation}
\end{question} 

We first observe that \Cref{imageofctsKq} is essentially a $p$-adic question in
$\TC$. 
For each $n \geq 1$, let  $\K(\X_n,\X_1)$ be the
fiber of $\K(\X_n)\rightarrow\K(\X_1)$. Since $\X_n$ is a $p$-adic nilpotent thickening of $\X_1$, 
the relative $\K$-theory $\K(\X_n, \X_1)$ has homotopy groups which are
bounded $p$-power torsion (cf.~\Cref{GHfiniteresult}, due to \cite{GHfinite}), and the spectrum is therefore $p$-complete. 
Using the Dundas-Goodwillie-McCarthy theorem \cite{DGM}, and taking limits, we obtain a
cartesian square
\begin{equation}
\label{cartsquare1}\xymatrix{\K^{\mathrm{cts}}(
\mathfrak{X}) \ar[r]\ar[d]&\K(\X_1)\ar[d]\\
\TC^{\mathrm{cts}}( \mathfrak{X}; \mathbb{Z}_p)\ar[r]&\TC(\X_1;\ZZ_p).} \end{equation}
Since the $\X_n, n \geq 1$ are $p$-power torsion schemes, their $\TC$ are already $p$-adically complete. 
Using this diagram, we see that  it suffices to determine the image of the map 
$\TC_i^{\mathrm{cts}}( \mathfrak{X}; \mathbb{Q}_p) \to \TC_i( \mathfrak{X}_1;
\mathbb{Q}_p)$. 

In this section, we will describe an explicit obstruction class for
Question~\ref{imageofctsKq} in case $i = 0$ 
(sharpening results of \cite{bek1}) in certain geometric situations and construct general obstruction 
classes in all cases (after \cite{Beilinson}). 

\subsection{The Bloch--Esnault--Kerz theorem}

In \cite{bek1}, Bloch--Esnault--Kerz consider Question~\ref{imageofctsKq} in the
case $i = 0 $ and where $\mathfrak{X} $ has the following form. 
Let $K$ be a complete discretely valued field of
characteristic zero with ring of integers $\mathcal{O}_K$,
whose residue field $k$ is perfect of characteristic $p>0$. We let $\pi \in
\mathcal{O}_K$ be a uniformizer and denote by $K_0$ the ring of fractions
$W(k)[\tfrac{1}{p}]$. 
We take $\mathfrak{X} \to \mathrm{Spf}(\mathcal{O}_K)$ to be a smooth $p$-adic formal scheme, 
with special fiber $\mathfrak{X}_k \to \spec(k)$ and rigid analytic generic
fiber $\mathfrak{X}_K$
over $K$. The goal is to understand the image of the map 
$\K_0^{\mathrm{cts}}( \mathfrak{X}; \mathbb{Q}_p) \to \K_0(\mathfrak{X}_k; \mathbb{Q}_p)$. 

We refer to \cite{bek1, emerton-variational} for more detailed motivation for the above question, as
well as \cite{BEK14, Mor14, Mor19} for discussions of the analogous question in
equal characteristic. 
Note that when $\mathfrak{X}$ arises from a smooth proper scheme $X \to \spf
(\mathcal{O}_K)$, the above question says nothing about the image of the map 
$\K_0(X; \mathbb{Q})\to \K_0(X_k; \mathbb{Q})$; this (at least up to homological
equivalence) is the subject of the far more difficult $p$-adic variational Hodge conjecture 
of Fontaine--Messing 
(Conjecture~\ref{FMconj}).

Here we will unwind the Beilinson fiber square to answer
Question~\ref{imageofctsKq} in this case in terms of the crystalline Chern
character. 
To begin with, we need to review the crystalline Chern
character and the crystalline to de Rham comparison.

\begin{construction}[de Rham cohomology] 
Given a smooth $p$-adic formal scheme $\X \to \mathrm{Spf}(\mathcal{O}_K)$, 
we will consider the ($p$-adic)
de Rham cohomology $R \Gamma_{\dR}(\X/\mathcal{O}_K) \in
D(\mathcal{O}_K)$, equipped with the
descending, multiplicative Hodge
filtration 
$\mathrm{Fil}^{\geq \star} R \Gamma_{\dR}(\X/\mathcal{O}_K)$. 
When $\X$ is also assumed proper, then 
all of these are perfect complexes in $D(\mathcal{O}_K)$. 
Furthermore, after inverting $p$, we write
$R \Gamma_{\dR}(\X_K/K) \in D(K)$ and 
$\mathrm{Fil}^{\geq \star} R \Gamma_{\dR}(\X_K/K)$ for the induced objects. 

A basic fact we will use is that when $\X$ is proper, the induced spectral sequence from the Hodge
filtration 
on $R \Gamma_{\dR}(\X_K/K)$
degenerates after rationalization; this is the degeneration of the
Hodge-to-de Rham spectral sequence for proper smooth rigid analytic varieties,
proved by Scholze \cite{Sc13b}. 
\end{construction} 

\begin{construction}[Comparison between crystalline and de Rham cohomology] 
\label{comparisondRcrys}
Given $\X \to \spf(\mathcal{O}_K)$ a smooth $p$-adic formal scheme, 
we can consider the crystalline cohomology $R \Gamma_{\crys}(\X_k)$ of the
special fiber as well. In the absolutely unramified case (when $\mathcal{O}_K = W(k)$), the usual de
Rham to crystalline comparison theorem yields an equivalence 
$R \Gamma_{\crys}(\X_k) \simeq R \Gamma_{\dR}(\X/\mathcal{O}_K)$. 
In general, by \cite[Theorem 2.4]{BO83}, we have a natural equivalence after
rationalization
\begin{equation} \label{dRtocrys} R \Gamma_{\dR}(\X_K/K) \simeq  R
\Gamma_{\crys}(\X_k; \mathbb{Q}_p)  \otimes_{K_0} K. \end{equation}
\end{construction} 
\begin{construction}[The crystalline Chern character] 
Let $Y$ be a regular scheme of 
characteristic $p$. 
Given a vector bundle $\mathcal{V}$ on $Y$, we can define \emph{Chern classes}
$c_i (\mathcal{V}) \in H^{2i}_{\mathrm{crys}}(Y)$ for $i \geq 0$ satisfying the
usual axioms, cf.~\cite{Gros85} (e.g., using the classical method of \cite{Gro58}). The usual formula then yields a \emph{crystalline Chern
character}, i.e., a
natural ring homomorphism
into the rationalized crystalline cohomology
\[ \mathrm{ch}_{\mathrm{crys} }\colon \K_0(Y) \to \bigoplus_{i \geq 0}
H^{2i}_{\mathrm{crys}}(Y; \mathbb{Q}_p),  \]
which carries the class of a  line bundle  $\mathcal{L}$ to $1 +
c_1(\mathcal{L})$. 
\end{construction}

Our main result is the following theorem, which extends results of
Bloch--Esnault--Kerz \cite{bek1}. In \cite{bek1}, this result is proved in the
case where $K=K_0$ is absolutely unramified, $X$ arises
from a smooth projective scheme, and  $p >
\mathrm{dim}(X) + 6$. In \cite{Beilinson}, it is shown that there is an obstruction 
in $\bigoplus_{i \geq 0} H^{2i}_{\mathrm{dR}}(X_K)/\mathrm{Fil}^{\geq
i}H^{2i}_{\mathrm{dR}}(X_K)$, but the obstruction is not identified with the
Chern character; see Section~\ref{sec:beilinsonobstruction} below for more discussion. 

\begin{theorem} 
\label{BEKthm}
Let $K$ be a complete discretely valued field of characteristic zero with ring of integers $\mathcal{O}_K$,
whose residue field $k$ is perfect of characteristic $p > 0$. Let $\X \to \mathrm{Spf}(\mathcal{O}_K)$ be a
 proper smooth $p$-adic formal scheme with special fiber $\X_k$.
A class $x \in \K_0(\X_k; \mathbb{Q}_p)$ lifts to
$\K_0^{\mathrm{cts}}(\X;
\mathbb{Q}_p)$ if and only if the
crystalline 
Chern character $\mathrm{ch}_{\mathrm{crys}}(x) \in \bigoplus_{i \geq 0}
H^{2i}_{\mathrm{crys}}(\X_k; \mathbb{Q}_p)$ maps (via the comparison map of
\eqref{dRtocrys}) to 
$\bigoplus_{i \geq 0}
\mathrm{Fil}^{\geq i} H^{2i}_{\dR}(\X_{K}/K) \subseteq \bigoplus_{i\geq 0}
H^{2i}_{\dR}(\X_{K}/K)$. 
\end{theorem} 

The proof of \Cref{BEKthm} will be carried out as follows. First, we give an
analogous form of the Beilinson fiber square when we work relative to
$\mathcal{O}_K$ (\Cref{BfibOKprop}). Next, we will show that 
the $p$-adic Chern character can be defined entirely in terms of the
special fiber (which will use some Kan extension techniques from
\Cref{app:Kan}), and then identify it with the crystalline Chern character
(\Cref{chpiscrysch}). \Cref{BEKthm} will then follow directly.  

In the next result, we 
will use the continuous Hochschild (resp.~negative cyclic, periodic cyclic)
homology of a formal scheme over $\mathcal{O}_K$, defined as 
in \Cref{continuousinv}; note that \Cref{contcompmapeq} applies to these
relative theories too, since they can be recovered from $\THH$.  

\begin{proposition}[The fiber square relative to $\mathcal{O}_K$] 
\label{BfibOKprop}
Let $\X$ 
be a smooth formal $\mathcal{O}_K$-scheme. Then there are natural fiber
squares
 \begin{equation} \label{BfibOK}\begin{gathered}  \xymatrix{
\K^{\mathrm{cts}}(\X; \mathbb{Q}_p) \ar[d]  \ar[r] &  \K(\X_k; \mathbb{Q}_p) \ar[d]
\\
\TC^{\mathrm{cts}}(\X; \mathbb{Q}_p) \ar[d]  \ar[r] &  \TC(\X_k; \mathbb{Q}_p)
\ar[d]  \\
\HC^{-, \mathrm{cts}}(\X/\mathcal{O}_K; \mathbb{Q}_p) \ar[r] &
\HP^{\mathrm{cts}}(\X/\mathcal{O}_K; \mathbb{Q}_p).
}\end{gathered} \end{equation}
\end{proposition} 
\begin{proof} 
This will follow from the Beilinson fiber square. 
By Zariski descent of all terms in 
the formal scheme $\X$, we can assume that $\X = \spf(R)$ for $R$ a formally
smooth, $p$-complete $\mathcal{O}_K$-algebra. 
First, 
$ \HH( \mathcal{O}_K; \mathbb{Z}_p) \simeq \HH( \mathcal{O}_K/W(k);
\mathbb{Z}_p)$. Since
$L_{\mathcal{O}_K/W(k)}$ is quasi-isogenous to zero, we find that  
the map 
$\HH( \mathcal{O}_K; \mathbb{Q}_p) \to K$ given by truncation
is an equivalence. 
We thus conclude (via Hochschild--Kostant--Rosenberg) $\HH^{\mathrm{cts}}(R; \mathbb{Q}_p) \to \HH^{\mathrm{cts}}(R/\mathcal{O}_K; \mathbb{Q}_p)$ 
is an equivalence, whence $\HC(R; \mathbb{Q}_p) \to \HC(R/\mathcal{O}_K;
\mathbb{Q}_p)$ is an equivalence too by taking $S^1$-coinvariants. 
Here we use that 
$p$-adic completion commutes with taking $S^1$-coinvariants on connective
spectra, since taking $S^1$-coinvariants behaves as a finite colimit in any
ranges of degrees. 
Therefore, the diagram
\[ \xymatrix{
\HC^-(R; \mathbb{Q}_p) \ar[d]  \ar[r] &  \HP(R; \mathbb{Q}_p) \ar[d]  \\
\HC^-(R/\mathcal{O}_K ; \mathbb{Q}_p) \ar[r] &
\HP(R/\mathcal{O}_K; \mathbb{Q}_p) 
}\]
is homotopy cartesian. Combining with the Beilinson fiber square, the
result now follows. 
\end{proof} 

\begin{construction}[The $p$-adic Chern character map] 
Since $\X/\mathcal{O}_K$ is smooth, we obtain from 
Hochschild--Kostant--Rosenberg type filtrations (as in \cite{antieauperiodic},
using Adams operations as in \cite[Sec.~9.4]{BMS2} to split the filtration)
natural decompositions
$$\HP^{\mathrm{cts}}(\X/\mathcal{O}_K; \mathbb{Q}_p) \simeq \prod_{i \in \mathbb{Z}} 
R \Gamma_{\dR}(\X_K/K)[2i]$$ 
and $$\HC^{-,\mathrm{cts}}(\X/\mathcal{O}_K; \mathbb{Q}_p) \simeq \prod_{i \in \mathbb{Z}}
\mathrm{Fil}^{\geq i} R \Gamma_{\mathrm{dR}}(\X_K/K)[2i].
$$
It follows that we obtain from \eqref{BfibOK} a natural map 
\begin{equation} \label{XChmapOK}  \K(\X_k; \mathbb{Q}_p) \to \TC(\X_k;
\mathbb{Q}_p) \to  \prod_{i \in \mathbb{Z}} R
\Gamma_{\dR}(\X_K/K)[2i]  \end{equation}
for every smooth $p$-adic formal scheme $\X \to \spf(\mathcal{O}_K)$. 
We observe that both the source and target actually depend only on the special
fiber $\X_k$ of $\X$, thanks to \Cref{comparisondRcrys}. 
Furthermore, to construct 
\eqref{XChmapOK}, it suffices to work with affine formal schemes over
$\mathcal{O}_K$, by Zariski descent of the target, so we can assume 
$\X = \mathrm{Spf}(R)$ for $R$ a formally smooth $\mathcal{O}_K$-algebra. 
That is, we have a natural  map $\TC(R \otimes_{\mathcal{O}_K} k; \mathbb{Q}_p) \to \prod_{i \in \mathbb{Z}} (R
\Gamma_{\mathrm{crys}}( \mathrm{Spec}(R \otimes_{\mathcal{O}_K}
k);{\mathbb{Q}_p}) \otimes_{K_0} K)[2i]$.

Consider the functors
on smooth $k$-algebras 
\begin{equation} \label{functorsa} A \mapsto \TC(A; \mathbb{Q}_p)\quad\text{and}\quad
A\mapsto\prod_{i \in \mathbb{Z}}  \left( R\Gamma_{\mathrm{crys}}( \spec(A);
\mathbb{Q}_p) \otimes_{K_0} K \right)[2i]. \end{equation}
The left Kan extension of $\TC(-; \mathbb{Q}_p)$ to 
almost finitely presented objects of $\mathrm{SCR}_{k}$ (as in \Cref{almostfp})
is $\TC(-; \mathbb{Q}_p)$ again, since this functor commutes with 
geometric realizations in $\mathrm{SCR}_k$. By 
\Cref{isogenyTHH}, $\TC(-; \mathbb{Q}_p) = \TC( \pi_0(-); \mathbb{Q}_p)$ on
    $\mathrm{SCR}_k$, so hypotheses (1) and (2) of
\Cref{TCKanextlemma} are satisfied when applied to the left Kan extensions
of the functors \eqref{functorsa} on smooth $k$-algebras
(and Zariski sheafifying again). It follows that \eqref{XChmapOK} actually upgrades to a natural
transformation of functors in the special fiber alone. That is, for 
every smooth $k$-scheme $Z$, we obtain a natural map 
\begin{equation} 
\label{ZChmapk}
\K(Z; \mathbb{Q}_p) \xrightarrow{\tr} \TC(Z; \mathbb{Q}_p) \to \prod_{i \in \mathbb{Z}}(  R
\Gamma_{\crys}(Z; \mathbb{Q}_p) \otimes_{K_0} K)[2i],
\end{equation} 
such that \eqref{XChmapOK} is obtained by taking $Z = X_k$.

\end{construction} 
Next, we identify (up to scaling factors) the map \eqref{ZChmapk} on $\pi_0$ with
the crystalline Chern character.

\begin{proposition} 
\label{chpiscrysch}
There exists a scalar $\lambda \in K^{\times}$ such that 
for every smooth separated $k$-scheme $Z$, the map 
$\K_0(Z; \mathbb{Q}_p) \to \prod_{i \geq 0} H^{2i}_{\crys}(Z;\QQ_p)
\otimes_{K_0} K$ of
\eqref{ZChmapk} is given by 
the crystalline Chern character composed with the automorphism that multiplies
the $i$th factor by $\lambda^i$. 
\end{proposition} 
\begin{proof} 
 It suffices (by the
resolution property) to evaluate 
\eqref{ZChmapk} on the class of a vector bundle on $Z$, and for this we will
reduce to the universal case. 
For this it
will be convenient to extend to stacks over $k$ as well.
Given a smooth scheme or stack $Z$ over $k$, 
let $\mathrm{Vect}_n(Z)$ denote the groupoid of $n$-dimensional vector bundles
on $Z$. 
It follows that we obtain from \eqref{ZChmapk} a natural transformation (of spaces)
for all smooth $k$-schemes $Z$,
\[ f_n\colon \mathrm{Vect}_n(Z) \to \Omega^\infty \left( \prod_{i \in \mathbb{Z}} \left(R
\Gamma_{\crys}(Z; \mathbb{Q}_p) \otimes_{K_0} K \right)[2i] \right)  \]
such that the $\{f_n\}$ are additive and  multiplicative. 

Both the source and target of the $f_n$ are sheaves of spaces for the smooth or
\'etale topology on smooth $k$-schemes. 
Sheafifying for the smooth topology, we obtain such a natural transformation for
any smooth Artin stack, which still satisfies the additivity and
multiplicativity properties. 
By naturality, it suffices to show that 
for $Z = BGL_n$ and for $\mathcal{E}$ the tautological $n$-dimensional vector
bundle, $f_n(\mathcal{E})$ is given by (up to scalars) the crystalline Chern
character of $\mathcal{E}$. 
It follows that for each $n$, $f_n(\mathcal{E})$ is given by a power series
(with $K$ coefficients) in the 
crystalline 
Chern classes of $\mathcal{E}$, since 
$R\Gamma_{\mathrm{crys}}(BGL_n; \mathbb{Q}_p)$ (defined via sheafification) is the polynomial ring
$K_0[c_1, \dots, c_n]$, e.g., as in the calculations of de Rham and
Hodge cohomology of $BGL_n$ in \cite{Totaro18}. 
By additivity, multiplicativity, and the splitting principle to reduce to the
case of line bundles, we find easily that $f_n$ must be the Chern
character up to normalization by powers of some constant $\lambda$. 
Moreover, $\lambda \neq 0$ by comparison with the left-hand-side of
\eqref{BfibOK}. 
\end{proof}

\begin{proof}[Proof of \Cref{BEKthm}] 
We use the fiber square of \eqref{BfibOK}. 
As before, we have identifications
$\HP^{\mathrm{cts}}(\mathfrak X/\mathcal{O}_K; \mathbb{Q}_p) \simeq \prod_{i \in \mathbb{Z}} R
\Gamma_{\dR}(\mathfrak X_K/K)[2i]$ and 
$\HC^{-,\mathrm{cts}}(\mathfrak X/\mathcal{O}_K; \mathbb{Q}_p) \simeq \prod_{i \in \mathbb{Z}}
\mathrm{Fil}^{\geq i} R \Gamma_{\dR}(\mathfrak X_K/K)[2i]$. 
Using the crystalline-to-de Rham comparison (\Cref{comparisondRcrys}) and
\Cref{chpiscrysch}, we see that the map 
$$\K_0(\mathfrak X_k; \mathbb{Q}_p) \to \prod_{i \in \mathbb{Z}} H^{2i}_{\dR}(\mathfrak X_K/K) \simeq 
\prod_{i \in \mathbb{Z}} H^{2i}_{\crys}(\mathfrak X_k; \mathbb{Q}_p) \otimes_{K_0} K$$
is given up to scalar factors by the crystalline Chern character. 
The result now follows from 
\Cref{BfibOKprop}. 
\end{proof} 

\subsection{Generalization of Beilinson's obstruction; proof of
Theorem~E}
\label{sec:beilinsonobstruction}

Let $K, \mathcal{O}_K, k$ be as in the preceding subsection. 
Consider a proper 
scheme $X \to \spec(\mathcal{O}_K)$ with smooth generic fiber $X_K$ and
possibly singular special fiber $X_k$. 
In \cite{Beilinson}, Beilinson 
considers more generally the 
deformation problem for classes in higher $\K$-theory, and proves: 

\begin{theorem}[Beilinson \cite{Beilinson}] 
\label{bekrefinement}
Given $x \in \K_i(X_k)_{\mathbb{Q}}$, there is a natural  obstruction class in 
$\bigoplus_{r \geq 0} H^{2r-i}_{\mathrm{dR}}(X_K/K)/ \mathrm{Fil}^r
H^{2r-i}_{\mathrm{dR}}(X_K/K)
$ which vanishes
if and only if $x$ lifts to $(\varprojlim \K_i(X/\pi^n))_{\mathbb{Q}}$. 
More precisely, there is a natural equivalence of spectra
\[ \mathrm{cofib}( \K^{\mathrm{cts}}(X; \mathbb{Q}_p) \to \K(X_k;\mathbb{Q}_p))
\simeq \bigoplus_{r \geq 0} R \Gamma_{\mathrm{dR}}(X_K/K)/\mathrm{Fil}^{\geq r}
R \Gamma_{\mathrm{dR}}(X_K/K) [2r].
\]
\end{theorem}

In particular, 
\Cref{bekrefinement} applies for $i = 0$ and overlaps with the results of
\cite{bek1}, although it does not identify the
obstruction class with the crystalline Chern character. 
In this subsection, we observe that \Cref{bekrefinement} can be extended to
essentially arbitrary formal schemes, using comparisons between cyclic and de
Rham cohomology as in \cite{antieauperiodic}. 
This argument will not essentially rely on having a fiber square as in
Theorem~\hyperref[thm:a]{A}
(versus a fiber sequence), and could be deduced from the results of
\cite{Beilinson}. 

\begin{theorem}\label{thm:bek}
    Let $\mathfrak{X}$ be a qcqs $p$-adic formal
	 scheme with bounded $p$-power torsion.
    Given $i\in\mathbb Z$ and a class $x\in\K_i(\X_1;\QQ)$ there is a natural
 class
    $$c(x)\in\bigoplus_{r\ge0} H^{2r-i}\left( 
	L \Omega_{\mathfrak{X}}/ L \Omega_{\mathfrak{X}}^{\geq r} 
	 \right)_{\mathbb{Q}_p}$$ with the property that $x$ lifts to
    $\K_i^{\mathrm{cts}}( \mathfrak{X}; \mathbb{Q})$ if and only if $c(x)=0$.
	 More precisely, there is a natural equivalence
	 of spectra
	\begin{equation}  \mathrm{cofib}\left( \K^{\mathrm{cts}}(\mathfrak{X};
	\mathbb{Q}) \to \K( \mathfrak{X}_1; \mathbb{Q}) \right) \simeq \bigoplus_{r \geq 0} \left(
	L \Omega_{\mathfrak{X}}/L \Omega_{\mathfrak{X}}^{\geq
	r}[2r]\right)_{\mathbb{Q}_p }. 
	\label{naturalequiv2}
	\end{equation} 
\end{theorem}

\newcommand{\gr}{\mathrm{gr}}

\begin{proof}[Proof of Theorem~\ref{thm:bek}]
It clearly suffices to exhibit the natural equivalence \eqref{naturalequiv2}, and for this we may assume
$\mathfrak{X} = \spf(R)$ is affine, since all terms satisfy Zariski descent. 
Now we have seen 
that the cofiber in \eqref{naturalequiv2} 
can be identified with the 
cofiber of 
\( \TC(R; \mathbb{Q}_p) \to \TC(R/p; \mathbb{Q}_p),  \)
or equivalently with $ \HC(R; \mathbb{Q}_p)[2]$ by \Cref{thm_a_text}. 
We invoke the result of~\cite{antieauperiodic} which constructs on
    $\HC(R;\ZZ_p)[2]$
    a natural exhaustive decreasing filtration
	 $\fil^{\geq \star}\HC(R;\ZZ_p)[2]$ with graded pieces
    $$\gr^n\HC(R;\ZZ_p)[2]\we L \Omega_R/ L \Omega_R^{\geq n}[2n],$$
    where, as before,  $L \Omega_R$ is the $p$-adic derived de Rham cohomology of $R$ and
    $L \Omega^{\geq \star}_R$ is the Hodge filtration on the derived de Rham cohomology
    (see specifically the proof
    of~\cite[Corollary~4.11]{antieauperiodic}).\footnote{The work
    of~\cite{antieauperiodic} was essentially motivated by that
    of~\cite{BMS2} which among many other things established such
    filtrations for quasisyntomic rings by descent.}
   It follows that on $\HC(R;\QQ_p)[2]$ there is a natural exhaustive
    decreasing filtration $\mathrm{Fil}^{\geq \star}\HC(R;\QQ_p)[2]$ with graded pieces
    $$\gr^n\HC(R;\QQ_p)[2] \simeq (L \Omega_R/L \Omega_R^{\geq
	 n})_{\mathbb{Q}_p}[2n].$$
    An argument as in~\cite[Section~9.4]{BMS2} can be
    used to show that there is an action of Adams operations on
    $\HC(R;\ZZ_p)[2]$
    where $\lambda\in\ZZ_p^\times$ acts via $\lambda^n$ on
    $\gr^n\HC(R;\ZZ_p)$. In particular, these split $\HC(R;\QQ_p)[2]$ into
    eigenspaces so that there is a natural decomposition
    $$\HC(R;\QQ_p)[2]\we\bigoplus_n (L \Omega_R/L \Omega_R^{\geq
	 n})_{\mathbb{Q}_p}[2n].$$
    The result now follows from the Beilinson fiber sequence.
\end{proof}

\begin{remark}[Changing the base ring] 
In the above work, $\mathbb{Z}_p$ was used as the base for cyclic and de Rham
cohomology, but often this is not essential. 
Suppose now that $A$ is a commutative $\ZZ$-algebra with
$\widehat{L_{A/\ZZ}}$ quasi-isogenous to zero. Then it is not difficult to see that, for formal schemes over $A$, we can replace all occurrences of derived
de Rham cohomology relative to $\mathbb{Z}_p$ with such occurrences relative to
$A$. 
\end{remark}

\section{The motivic filtration on $\mathrm{TC}$}

In this section (which will not use the Beilinson fiber square), we prove some
general structural results on topological cyclic homology $\mathrm{TC}$ (in
particular, Theorem G from the introduction) and on
the ``motivic'' filtration constructed by Bhatt--Morrow--Scholze \cite{BMS2}.

Recall that, according to \cite{BMS2}, for $R$ a quasisyntomic ring (see
\Cref{def:qsyn} below for a review), 
$\mathrm{TC}(R; \mathbb{Z}_p)$ admits a 
complete descending $\mathbb{Z}_{\geq 0}$-indexed filtration $\mathrm{Fil}^{\geq \star}
\mathrm{TC}(R;\mathbb{Z}_p)$ with associated graded terms given as
$\gr^i\TC(R;\ZZ_p)\we\mathbb{Z}_p(i)(R)[2i]$. In this section, we will prove some
structural properties of this filtration. 
Our main results are as follows.

\begin{theorem}[Connectivity properties] 
\label{Zpibound}
\begin{enumerate}
    \item[{\rm (1)}]
Let $R \in \qsyn$ be a quasisyntomic ring. 
Then $\mathbb{Z}_p(i) (R) \in D^{\leq i+1}(\mathbb{Z}_p)$. 
Consequently, we have $\mathrm{Fil}^{\geq i} \TC(R; \mathbb{Z}_p) \in \sp_{\geq
i-1}$. 
\item[{\rm (2)}]
The functors $R \mapsto \mathbb{Z}_p(i)(R)$ and $R\mapsto\mathrm{Fil}^{\geq i} \TC(R;
\mathbb{Z}_p)$ are left Kan extended from finitely generated
$p$-complete polynomial $\mathbb{Z}_p$-algebras. 
\end{enumerate}
\end{theorem} 

Part (2) was indicated to us by Scholze. 
In view of it, we can extend the construction of the $\mathbb{Z}_p(i)$ to all ($p$-complete) 
rings. 

\begin{theorem}[Rigidity] 
\label{rigidity}
Let $(R, I)$ be a henselian pair {where $R$ and $R/I$ are $p$-complete}.
Then $\mathrm{fib}( \mathbb{Z}_p(i)(R) \to \mathbb{Z}_p(i)(R/I)) \in D^{\leq
i}(\mathbb{Z}_p)$. 
\end{theorem} 

In particular, using the known description in characteristic $p$, we obtain that
for any $R$ there is a complete description of the top cohomology $H^{i+1}(
\mathbb{F}_p(i)(R))$ and that this vanishes \'etale locally.

\subsection{Review of \cite{BMS2}}\label{sec:review}
Here we recall some of the major results and techniques of \cite{BMS2}. 
We recall first  the quasisyntomic site $\qsyn$ 
 (a non-noetherian version of the syntomic site used by
Fontaine--Messing~\cite{FM87})
 and the subcategory $\qrsp \subset \qsyn$ of quasiregular semiperfectoid rings. 

\begin{definition}[{$p$-complete (faithful) flatness and
$\mathrm{Tor}$-amplitude, \cite[Def. 4.1]{BMS2}}] 
Let $R$ be a commutative ring. An $R$-module $M$
is called \emph{$p$-completely flat} (resp.~\emph{$p$-completely faithfully
flat}) if 
$M \otimes^{\LL}_R (R/p) \in D(R/p)$ is a flat (resp.~faithfully flat)
$R/p$-module concentrated in degree zero. 
Similarly, an object $N \in D(R)$ has \emph{$p$-complete
$\mathrm{Tor}$-amplitude in $[a,b]$} if $N \otimes^{\LL}_R R/p \in D(R/p)$ has
$\mathrm{Tor}$-amplitude in $[a, b]$. 
\end{definition}

\begin{definition}[{The quasisyntomic site, cf.~\cite[Sec.~4]{BMS2}}] 
\label{def:qsyn}
\begin{enumerate}
\item  
A commutative ring $R$ is called \emph{quasisyntomic}
if it is $p$-complete, has bounded $p$-power torsion, and $\L_{R/\mathbb{Z}_p}$
 has $p$-complete $\mathrm{Tor}$-amplitude in $[-1, 0]$ (indexing conventions for the derived category are
cohomological).
We let $\qsyn$ be the category of quasisyntomic rings, with all ring homomorphisms.
\item
The category $\qsyn$ (or more precisely its opposite) acquires the structure of a site as follows: a map $A\to
B$ in $\qsyn$ is a cover
if $A \to B$ is $p$-completely faithfully flat and if $\L_{B/A} \in D(B)$ has  $p$-complete $\mathrm{Tor}$-amplitude in
$[-1, 0]$. 
We call a map with all of the above properties, except that $A \to B$ only
assumed $p$-completely flat (rather than $p$-completely faithfully flat), a \emph{quasisyntomic map}. 
\item 
An object $R \in \qsyn$ is \emph{quasiregular semiperfectoid}
if $R$ admits a map from a perfectoid ring and the Frobenius 
on $R/p$ is surjective. 
We let $\qrsp \subset \qsyn$ be the full subcategory spanned by quasiregular
semiperfectoid rings. 
If $R$ is additionally an $\mathbb{F}_p$-algebra, then $R$ is called
\emph{quasiregular semiperfect.}
\end{enumerate}
\end{definition} 

\newcommand{\qqsyn}[1]{\mathscr{Q}\mathrm{Syn}_{#1}}
\newcommand{\qqrsp}[1]{\mathscr{Q}\mathrm{RSPerfd}_{#1}}
For future reference, we will also 
need the relative versions
$\mathrm{qSyn}_A$ and $\qqsyn{A}$ of the quasisyntomic sites (of which the first is
considered in \cite{BMS2}). 

\begin{definition}[{Relative quasisyntomic sites, cf.~\cite[Sec.~4.5]{BMS2}}] 
Fix a quasisyntomic ring $A \in \qsyn$. 
We define the sites $\mathrm{qSyn}_A$ and $\qqsyn{A}$ as follows. 
\begin{enumerate}
\item We let $\qqsyn{A}$ denote the category of
$A$-algebras
$B$ which are quasisyntomic as underlying rings and such that $L_{B/A}
\in D(B)$ has
$p$-complete $\mathrm{Tor}$-amplitude in $[-1, 0]$.  
We let $\mathrm{qSyn}_A \subset \qqsyn{A}$ be the full subcategory spanned by
the quasisyntomic $A$-algebras (i.e., those $B$ such that $B$ is additionally
$p$-completely flat over $A$). 
\item 
We make $\qqsyn{A}$ and $\mathrm{qSyn}_A$ into sites by declaring a cover to be a map which is a cover
in $\qsyn$. 
\item We let $\qqrsp{A}$ (resp.~$\qlrsp_A$) denote the subcategory of $\qqsyn{A}$
(resp.~$\mathrm{qSyn}_A$) spanned by
$A$-algebras 
whose underlying ring is quasiregular semiperfectoid. 
Note that if $B \in \qqrsp{A}$, then the $p$-completion of $L_{B/A}[-1]$ is a
$p$-completely flat, discrete $B$-module by \cite[Lem.~4.7(1)]{BMS2}.
\end{enumerate}
\end{definition} 

Note that in the case $A = \mathbb{Z}_p$, 
$\mathrm{qSyn}_{\mathbb{Z}_p}$ is the category of $p$-torsion free quasisyntomic
rings and $\qqsyn{\mathbb{Z}_p} = \qsyn$. 
For $A = \mathbb{F}_p$, 
$\qqsyn{\mathbb{F}_p}$ and $\mathrm{qSyn}_{\mathbb{F}_p}$ are both simply the
subcategory of $\qsyn$ spanned by those quasisyntomic rings which are
$\mathbb{F}_p$-algebras \cite[Lemma 4.34]{BMS2}; more generally $\qqsyn{A}$ is
the category of $A$-algebras which are quasisyntomic 
for any perfectoid ring $A$. 

The site $\qsyn$ has a basis given 
by $\qrsp$, and similarly in the relative cases. All the functors below will be sheaves on $\qsyn$; to describe them,
it therefore suffices to describe them as sheaves on $\qrsp$ \cite[Prop. 4.31]{BMS2}. 

We now review the prismatic sheaves  on $\qsyn$, constructed via topological Hochschild and cyclic homology. A
purely algebraic construction via the prismatic cohomology of Bhatt--Scholze
is given in 
\cite{Prisms} (at least for algebras over a base perfectoid ring), which also produces objects before Nygaard completion. 

\begin{definition}[{Prismatic sheaves on $\qsyn$,
\cite[Sec.~7]{BMS2}}]\label{def_prism_sheaves}
The objects $\Prism_R\left\{i\right\}$ and $\mathcal{N}^{\geq n} \Prism_R$ define
sheaves on $\qsyn$ with values in $D(\mathbb{Z}_p)^{\geq 0}$. Each of these sheaves is constructed  via descent \cite[Prop.
4.31]{BMS2} from $\qrsp \subset
\qsyn$, on which they take discrete values defined via topological Hochschild
homology.

\begin{enumerate}
\item  
For $R \in \qrsp$, $\THH(R;
\mathbb{Z}_p)$ is concentrated in even degrees, so the homotopy fixed
point and Tate spectral sequences for $\TC^-(R; \mathbb{Z}_p)$ and $\TP(R;
\mathbb{Z}_p)$ 
degenerate and $\TP(R; \mathbb{Z}_p)$ is 2-periodic.
For $R \in \qrsp$, we have \begin{gather} 
\Prism_R = \pi_0( \TC^-(R; \mathbb{Z}_p)) = \pi_0( \TP(R; \mathbb{Z}_p)).
\end{gather} 
\item 
For $R \in \qrsp$ and $n\in\mathbb Z$, the ideal $\mathcal{N}^{\geq n} \Prism_R \subset \Prism_R$ is the one
defined by the homotopy fixed point spectral sequence, i.e.,
 \begin{equation} \mathcal{N}^{\geq n} \Prism_R = \mathrm{im}\left( \pi_0(
 (\tau_{\geq 2n} \THH(R;
\mathbb{Z}_p))^{hS^1}) \to \pi_0 (\THH(R; \mathbb{Z}_p)^{hS^1})
\right). \end{equation}
\item 
For $i\in\mathbb Z$ we further have the invertible $\Prism$-modules (as
sheaves on $\qsyn$) $\Prism\left\{i\right\}$, called \emph{Breuil--Kisin twists}. For $R \in \qrsp$,
\begin{gather}
\Prism_R\left\{i\right\} = \pi_{2i} ( \TP(R; \mathbb{Z}_p)),\end{gather}
and by 2-periodicity $\Prism_R\left\{i\right\} =
\Prism_R\left\{1\right\}^{\otimes i}$. 
We have a natural isomorphism 
 \begin{equation}  \mathcal{N}^{\geq i}
\Prism_R \left\{i\right\} \simeq \pi_{2i}( \TC^-(R;\mathbb{Z}_p)).
\end{equation}
More generally,
$\mathcal N^{\ge n + i}\Prism_R\{i\}$
is the image of the injection $\pi_{2i}((\tau_{\ge 2n + 2i}\THH(R;\mathbb
Z_p))^{hS^1}) \to 
\pi_{2i}((\THH(R;\mathbb
Z_p))^{hS^1})
$ for $n\in\mathbb Z$. 
\item 
There are two maps 
of sheaves on $\qsyn$,  \begin{equation} \mathrm{can}, \varphi_i\colon \mathcal{N}^{\geq i} \Prism_R
\left\{i\right\} \rightrightarrows
\Prism_R \left\{i\right\} \end{equation} arising from the canonical and Frobenius maps
$\mathrm{TC}^-(R; \mathbb{Z}_p) \rightrightarrows \mathrm{TP}(R; \mathbb{Z}_p)$;
in particular, we obtain an endomorphism $\varphi = \varphi_0\colon \Prism_R \to
\Prism_R$.
\item 
Finally, the map $\TC^-(R; \mathbb{Z}_p) \to R$ yields a projection map 
$a_R\colon \Prism_R \to R$, a surjection with kernel $\mathcal{N}^{\geq 1}\Prism_R
\subset \Prism_R$. 
\end{enumerate}

\end{definition}

\begin{example}[{Perfectoid rings, \cite[Sec.~6]{BMS2}}] 
Let $R_0$ be a perfectoid ring. 
    In this case, we have Fontaine's ring $\ainf(R_0)=W(R_0^\flat)$ and the surjective map
$\theta\colon \ainf(R_0) \to R_0$, whose kernel is a principal ideal generated by a
nonzerodivisor $\xi \in \ainf(R_0)$; $\theta$ is the universal pro-nilpotent,
$p$-complete thickening of $R_0$. 
We have a canonical isomorphism $\Prism_{R_0} \simeq \ainf(R_0)$
such that the projection map $a_R\colon \Prism_{R_0} \to R_0$ is $\theta$. 
There are also non-canonical isomorphisms $\Prism_{R_0}\left\{i\right\} \simeq
\ainf(R_0)$ for each $i$. 
The map $\varphi = \varphi_0\colon \Prism_{R_0} \to \Prism_{R_0}$ is given by the
Witt vector Frobenius on $\ainf(R_0)$. 
The Nygaard filtration on $\Prism_R = \ainf(R_0)$ is the $\xi$-adic filtration. 
The map $\phi_i$ is injective (and $\phi_0$-semilinear) and its image
is given by $( \widetilde{\xi}^{-\mathrm{min}(i, 0)})$ for $\widetilde{\xi}  =
\varphi(\xi)$. 
\label{BMS2perfectoid}
\end{example}

\begin{definition}[The sheaves $\mathbb{Z}_p(i)$]\label{defBMSZPi}
For $i \geq 0$, 
the sheaf $\mathbb{Z}_p(i)$ is defined as the homotopy equalizer of 
$\can, \phi_i$, i.e., via 
$$\ZZ_p(i)\we\mathrm{fib}\left(\Nscr^{\geq i}\Prism_R\left\{i\right\}\xrightarrow{\can-\varphi_i}\Prism_R\left\{i\right\}\right).$$
Consequently, since $\mathrm{TC}$ is itself a homotopy equalizer,
we also have
for $R \in \qrsp$,
\begin{equation} \mathbb{Z}_p(i)(R) = (\tau_{[2i-1, 2i]} \mathrm{TC}(R;
\mathbb{Z}_p))[-2i]. 
\end{equation}
We also define quasisyntomic sheaves $\mathbb{F}_p(i)$ and
$\mathbb{Q}_p(i)$ by reducing $\mathbb{Z}_p(i)$ modulo $p$ or inverting $p$ on
$\ZZ_p(i)$,
respectively. 
\end{definition}

Using the result \cite[Sec.~3]{BMS2} that $\mathrm{TC}$ defines a sheaf for the fpqc topology on
rings, one sheafifies the Postnikov filtration and obtains the following fundamental result.

\begin{theorem}[{``Motivic filtrations,'' \cite[Theorem 1.12]{BMS2}}] 
\label{BMSfilt}
Let $R \in \qsyn$. Then 
$\THH(R; \mathbb{Z}_p)$, $\TC^-(R; \mathbb{Z}_p), \TP(R; \mathbb{Z}_p)$, and $\TC(R;
\mathbb{Z}_p)$
naturally upgrade to filtered spectra  with complete,
multiplicative descending filtrations
$\fil^{\geq\star}\THH(R; \mathbb{Z}_p)$, $\fil^{\geq\star}\TC^-(R;
\mathbb{Z}_p), \fil^{\geq\star}\TP(R; \mathbb{Z}_p)$, and
$\fil^{\geq\star} \TC(R; \mathbb{Z}_p)$, indexed by $\mathbb Z_{\ge0}, \mathbb
{Z}, \mathbb Z$, and $\mathbb Z_{\ge0}$ respectively,
such that the associated graded pieces are given by
\begin{enumerate}
    \item[{\rm (1)}] 
$\gr^i \THH(R; \mathbb{Z}_p) \simeq \mathcal{N}^i \Prism_R\{i\}[2i]
\stackrel{\mathrm{def}}{=}  \left( \mathcal{N}^{\geq i} \Prism_R\{i\}/\mathcal{N}^{\geq i+1}
\Prism_R\{i\}\right) [2i]$ for all $i\ge0$; moreover, in this case the
Breuil--Kisin twists can be trivialized, so 
also $\gr^i \THH(R; \mathbb{Z}_p)\simeq \mathcal{N}^i \Prism_R[2i]$,
\item[{\rm (2)}] $\gr^i \TC^-(R; \mathbb{Z}_p) = \mathcal{N}^{\geq i}
\Prism_R\left\{i\right\} [2i]$ for all $i \in \mathbb{Z}$,
\item[{\rm (3)}]  $\gr^i \TP(R; \mathbb{Z}_p) = \Prism_R\left\{i\right\}[2i]$ for all $i\in\mathbb Z$,
\item[{\rm (4)}]  $\gr^i \TC(R; \mathbb{Z}_p) = \mathbb{Z}_p(i)(R)[2i]$ for all
$i \geq 0$.
\end{enumerate}

\end{theorem} 

\begin{remark}[Comparison with $\K$-theory] 
Recall that for $p$-adic rings, $\mathrm{TC}$ and $p$-adic \'etale $\K$-theory
agree in nonnegative degrees. Cf.~\cite{GH99} for smooth algebras in characteristic $p$, and in general \cite{CMM, CM}. 
One may thus expect 
the filtration of 
\Cref{BMSfilt} to be the \'etale sheafification of the filtration on algebraic
$\K$-theory 
 with associated graded motivic cohomology,
cf.~\cite{FS02, Lev08} (for smooth schemes over fields). 
In particular, one expects the $\mathbb{Z}_p(i)$ to be some form 
of $p$-adic \'etale motivic cohomology. This is essentially understood in equal
characteristic (already by \cite{BMS2}), as we review below, but has not yet
appeared in mixed characteristic. In mixed characteristic and under
finiteness assumptions (e.g., smooth schemes over a DVR), many
authors have studied \'etale motivic cohomology \cite{Ge04} and similar ``$p$-adic \'etale Tate twists,'' e.g.,
those of 
\cite{FM87, Sch94, Sa07}, though the construction is very different from that
of \cite{BMS2}; one ultimately hopes to compare all of them, and we will at least
offer some information in this and the next section.  
\end{remark} 

We review the discreteness property of the $\mathbb{Z}_p(i)$. By construction, 
the objects $\mathcal{N}^{\geq i} \Prism_R\left\{i\right\}$ are sheaves on
$\qsyn$ with values in $D(\mathbb{Z}_p)^{\geq 0}$ (recall \cite[Cor.~2.1.2.3]{SAG} that such sheaves form the coconnective
part of the derived $\infty$-category of the category of abelian sheaves on $\qsyn$);
as objects of this category, they are in fact
\emph{discrete}, since they take discrete values on the basis $\qrsp$.  
A deep result of Bhatt--Scholze (conjectured in \cite{BMS2} and proved in the
characteristic $p$ case there) is that this
discreteness also holds for the $\mathbb{Z}_p(i)$,
although they in general take values in cohomological degrees $[0, 1]$ for rings in $\qrsp$.

\begin{theorem}[Bhatt--Scholze, {\cite[Theorem 14.1]{Prisms}}] 
\label{oddvanishingconj}
The $D(\mathbb{Z}_p)^{\geq 0}$-valued sheaf $\mathbb{Z}_p(i)$ on $\qsyn$
is discrete and torsion free. More precisely, given $R \in \qsyn$, there is a
cover $R \to R'$ in $\qsyn$ such that $\mathbb{Z}_p(i)(R')$ is discrete and
torsion free. 
\end{theorem}

Finally, we review the prism
structure on $\Prism_R$, for $R$ quasiregular semiperfectoid. For simplicity, we
will assume $R$ to be $p$-torsion free.

\begin{proposition} 
\label{TPTHHtcp}
Let $R \in \qlrsp_{\mathbb{Z}_p}$. Suppose $R$ is an algebra over the perfectoid
ring $R_0$, with notation as in 
\Cref{BMS2perfectoid}. 
Then $\TP(R; \mathbb{Z}_p)/\widetilde{\xi} \simeq
\THH(R; \mathbb{Z}_p)^{tC_p}$, and this is concentrated in even
degrees and $p$-torsion free. 
\end{proposition} 
\begin{proof} 
We have that 
$$\TP(R; \mathbb{Z}_p)/{\xi} \simeq \HP(R / R_0 ; \mathbb{Z}_p) $$ by \cite[Theorem 6.7]{BMS2}. 
Since $R$ is $p$-torsion free and quasiregular semiperfectoid, we find that
$\HP(R/R_0; \mathbb{Z}_p)$ is concentrated in even degrees and is $p$-torsion free, where it
is given by Hodge-complete $p$-adic derived de Rham cohomology by
\cite[Prop. 5.15]{BMS2}. 
In particular, it follows that $({\xi}, p)$ defines a regular sequence
on $\Prism_R$. 
Since $\Prism_R$ is complete with respect to this ideal, it follows that $(p,
{\xi})$ is a regular sequence; since $\xi^p \equiv \widetilde{\xi}
(\mathrm{mod} \ p)$, we get that $(p, \widetilde{\xi})$ is a regular sequence,
and hence so is $(\widetilde{\xi}, p)$. 
Now, the equivalence 
$\TP(R; \mathbb{Z}_p)/\widetilde{\xi} \simeq
\THH(R; \mathbb{Z}_p)^{tC_p}$ is \cite[Prop. 6.4]{BMS2}, from which the
remainder now follows. 
\end{proof} 

\begin{construction}[{The prismatic structure on $\Prism_R$, cf.~\cite[Sec.~13]{Prisms}}] \label{consprism}
Let $R \in \qlrsp_{\mathbb{Z}_p}$. Suppose $R$ is 
an algebra over the perfectoid ring $R_0$. 
Then the  ring $\Prism_R = \pi_0( \TP(R; \mathbb{Z}_p))$ has the structure 
of a prism (in the sense of \cite{Prisms}).

\begin{enumerate}
\item  
We have the endomorphism $\varphi  = \varphi_0$, which is
congruent to the Frobenius modulo $p$, by \cite[Sec.~13]{Prisms}, and thus defines a $\delta$-structure on
$\Prism_R$ by $p$-torsionfreeness. 

\item
We have the ideal $I \subset \Prism_R$ given by $I = (\widetilde{ \xi})$; $I$ is the kernel of
$\pi_0( \TP(R; \mathbb{Z}_p)) \to \pi_0 ( \THH(R; \mathbb{Z}_p)^{tC_p})$ and
therefore does not depend on the choice of $R_0$. 

\end{enumerate}
 Finally, there is a natural map 
$$\eta_R\colon  R \to \Prism_R/I,$$ 
given via the cyclotomic Frobenius $\THH(R; \mathbb{Z}_p) \to \THH(R;
\mathbb{Z}_p)^{tC_p} = \TP(R; \mathbb{Z}_p)/\widetilde{\xi}$ upon applying $\pi_0$. 
\end{construction} 
\begin{remark} 
In fact, by \cite[Theorem 13.1]{Prisms}, $\Prism_R$ is the Nygaard completion of
the absolute 
prismatic cohomology of $R$, although we will not need this fact. 
\end{remark}

\subsection{Relative $\THH$ and its filtration}
\label{sec:relativeTHH}
In this subsection and the next, we will prove connectivity bounds for the 
motivic filtration on $\THH$. 
We will prove that for any $R \in \qsyn$, we have $\fil^{\geq n} \THH(R;
\mathbb{Z}_p) \in \sp_{\geq n}$ and  
$\mathcal{N}^{n} \Prism_R \in D^{\leq n}(\mathbb{Z})$.
It is not difficult to deduce the above connectivity bound in the case $R$ is
an algebra over a fixed perfectoid ring, using methods as in \cite[Sec.~6--7]{BMS2}; see in particular \cite[Const.~7.4]{BMS2}. To verify the connectivity bound in the general case, we will use additionally a
fiber sequence which arises from the work of 
Krause--Nikolaus~\cite{KN}, which gives a comparison between relative and
absolute $\THH$. 

Let $\mathcal{O}_K$ denote a complete discrete valuation ring of mixed characteristic
$(0, p)$ with perfect residue field $k$; let $\pi \in \mathcal{O}_K$ be a
uniformizer. 
The primary case of interest is $\mathcal{O}_K = \mathbb{Z}_p$ and $\pi = p$. 

\begin{construction}[Relative topological Hochschild homology] 
Let $R \in \qqsyn{\mathcal{O}_K}$. 
We consider the $\mathbb{E}_\infty$-ring $\mathbb{S}[z]$ and consider $R$ as an
$\mathbb{S}[z]$-algebra
via $z \mapsto \pi$. Using this, we can form 
the relative topological Hochschild homology (with $p$-adic coefficients)
$\THH(R/\mathbb{S}[z]; \mathbb{Z}_p)$. 
The construction $R \mapsto 
\THH(R/\mathbb{S}[z]; \mathbb{Z}_p)$ defines a sheaf of spectra on
$\qqsyn{\mathcal{O}_K}$, thanks to \cite[Sec.~3]{BMS2} (which gives that 
$R \mapsto \HH(R/\mathcal{O}_K; \mathbb{Z}_p)$ is a sheaf)
and
\eqref{relTHHdeformsHH} below. 
We observe the following two comparisons for relative $\THH$.
\begin{enumerate}
\item  
When we base-change along the map $\mathbb{S}[z] \to \mathbb{S}$ where $z\mapsto 0$, we find that 
 \begin{equation} \THH(R/\mathbb{S}[z]; \mathbb{Z}_p) \otimes_{\mathcal{O}_K} k \simeq
 \THH(R \otimes^{{L}}_{\mathcal{O}_K} k;
\mathbb{Z}_p).\end{equation}  
\item 
We have an equivalence
 \begin{equation}  \label{relTHHdeformsHH} \THH(R/\mathbb{S}[z]; \mathbb{Z}_p)
     \otimes_{\THH( \mathcal{O}_K/\mathbb{S}[z]; \mathbb{Z}_p)}
\mathcal{O}_K \simeq \HH(R/\mathcal{O}_K; \mathbb{Z}_p)  . \end{equation} 
Thus, 
$\THH(R/\mathbb{S}[z]; \mathbb{Z}_p)$ is a deformation of Hochschild homology relative
to $\mathcal{O}_K$. 
\end{enumerate}

\end{construction}

Next, we need an analog of  the
Hochschild--Kostant--Rosenberg theorem for relative $\THH$ (in the absolute case
for algebras over a perfectoid ring,
this is \cite[Theorem B]{He96} and \cite[Cor.~6.9]{BMS2}). 

\begin{proposition} 
Let $R$ be a formally smooth $\mathcal{O}_K$-algebra. 
Then we have a natural isomorphism of graded rings
\[ \THH_*(R/\mathbb{S}[z];\mathbb{Z}_p) \simeq
\widehat{\Omega^{\ast}_{R/\mathcal{O}_K}}
[\sigma] , \quad|\sigma| = 2, \]
    where $\widehat{\Omega^{\ast}_{R/\mathcal{O}_K}}$ denotes the
    $p$-completion of the de Rham complex of $R$ over $\Oscr_K$.
\end{proposition} 
\begin{proof} 
In the case $R = \mathcal{O}_K$, this follows from B\"okstedt's
calculation of $\THH(\mathbb{F}_p)$, cf.~\cite[Prop. 11.10]{BMS2},
\cite[Theorem 3.5]{amn1},  or \cite[Theorem 3.1]{KN}. 
Now \eqref{relTHHdeformsHH} shows that 
$\THH(R/\mathbb{S}[z]; \mathbb{Z}_p)/\sigma \simeq \HH(R/\mathcal{O}_K;
\mathbb{Z}_p)$, and the Hochschild--Kostant--Rosenberg theorem yields $\HH_*(R/\mathcal{O}_K;
\mathbb{Z}_p) \simeq  \widehat{\Omega^{\ast}_{R/\mathcal{O}_K}}$. 
It remains to show that the induced Bockstein spectral sequence  for 
$\THH(R/\mathbb{S}[z]; \mathbb{Z}_p)$
(with respect to
taking the cofiber of $\sigma$) degenerates, or equivalently that the map of the
HKR isomorphism lifts to a map $\widehat{\Omega_{R/\mathcal{O}_K}^{\ast}} \to
\THH_*(R/\mathbb{S}[z]; \mathbb{Z}_p)$.  
Indeed, since $\THH(R/\mathbb{S}[z]; \mathbb{Z}_p)$ is an $\mathbb{E}_\infty$-algebra with an
$S^1$-action receiving a map from $\THH(\mathcal{O}_K/\mathbb{S}[z]; \mathbb{Z}_p)$, we obtain the structure of a commutative differential graded
algebra on 
$\THH_*(R/\mathbb{S}[z]; \mathbb{Z}_p)$, and it receives a map (of cdgas) from
$\mathcal{O}_K[\sigma]$ with trivial differential. The universal property of the
de Rham complex now produces the desired map 
$\widehat{\Omega^{\ast}_{R/\mathcal{O}_K}} \to \THH(R/\mathbb{S}[z]; \mathbb{Z}_p)$. 
\end{proof} 

Left Kan extending from finitely generated $p$-complete polynomial $\mathcal{O}_K$-algebras, we obtain the following result, which is proved exactly as in
\cite[Prop. 7.5]{BMS2}; the key point is that any $R \in \qqrsp{\mathcal{O}_K}$
has the property that the $p$-completion of $L_{R/\mathcal{O}_K}$ is the
shift of a $p$-completely flat $R$-module.

\begin{corollary} 
\label{grrelativeTHH}
If $R \in \qqrsp{\mathcal{O}_K}$,
then $\THH(R/\mathbb{S}[z]; \mathbb{Z}_p)$ is concentrated in even degrees, and
each $\pi_{2n}
\THH(R/\mathbb{S}[z]; \mathbb{Z}_p)$ is a $p$-completely flat $R$-module. 
Furthermore, the $R$-module
$\pi_{2n}\THH(R/\mathbb{S}[z]; \mathbb{Z}_p)$
admits a natural finite increasing filtration with graded pieces
the (discrete and $p$-completely flat) $R$-modules
$\widehat{(\bigwedge^j L_{R/\mathcal{O}_K})} [-j] $ for $j \leq n$,
where $\widehat{\bigwedge^j L_{R/\mathcal{O}_K}}$ denotes the $p$-completion
    of $\bigwedge^j L_{R/\mathcal{O}_K}$.
\end{corollary} 

\begin{construction}[The filtration on relative $\THH$] 
Let $R \in \qqsyn{\mathcal{O}_K}$. 
In \cite[Sec.~11]{BMS2}, a multiplicative, convergent $\mathbb{Z}_{\geq 0}$-indexed filtration 
$\fil^{\geq \star} \THH(R/\mathbb{S}[z]; \mathbb{Z}_p)$
on $\THH(R/\mathbb{S}[z]; \mathbb{Z}_p)$
in sheaves of spectra on $\qqsyn{\mathcal{O}_K}$ is defined.\footnote{Actually,
in \emph{loc.~cit}, the filtration is defined only on those objects which are flat
over $\mathcal{O}_K$, but the arguments do not require this.} This filtration 
is defined such that it restricts to the double speed Postnikov filtration for $R \in
\qqrsp{\mathcal{O}_K}$, i.e., $\fil^{\geq n} \THH(R/\mathbb{S}[z]; \mathbb{Z}_p) =
\tau_{\geq 2n} \THH(R/\mathbb{S}[z]; \mathbb{Z}_p)$ for such $R$. 
By \Cref{grrelativeTHH} and \cite[Theorem 3.1]{BMS2}, the associated graded
    pieces of the Postnikov filtration on 
$\THH(-/\mathbb{S}[z]; \mathbb{Z}_p)$ on $\qqrsp{\mathcal{O}_K}$ are
sheaves; thus, one unfolds and  obtains the filtration for all 
$R \in \qqsyn{\mathcal{O}_K}$. 
\end{construction}

\begin{corollary} 
\label{filtongrTHH}
Let $R \in \qqsyn{\mathcal{O}_K}$.
Then $\mathrm{gr}^n \THH(R/\mathbb{S}[z]; \mathbb{Z}_p)$ admits a natural finite
increasing filtration
with associated graded 
$( \widehat{\bigwedge^j L_{R/\mathcal{O}_K})}[2n-j]$ for $j \leq n$. 
In particular, we find that $\mathrm{gr}^n\THH(R/\mathbb{S}[z]; \mathbb{Z}_p) \in
\sp_{\geq n}$ and $\fil^{\geq n} \THH(R/\mathbb{S}[z]; \mathbb{Z}_p) \in \sp_{\geq n}$. 
Furthermore, the constructions $$R \mapsto \mathrm{gr}^n \THH(R/\mathbb{S}[z];
    \mathbb{Z}_p)\quad\text{and}\quad
R\mapsto\fil^{\geq n} \THH(R/\mathbb{S}[z]; \mathbb{Z}_p)$$ (as functors on
$\qqsyn{\mathcal{O}_K}$ to
$p$-complete spectra) are left Kan extended from
finitely generated  $p$-complete polynomial $\mathcal{O}_K$-algebras. 
\label{2grrelativeTHH}
\end{corollary} 
\begin{proof} 
The first assertion follows from \Cref{grrelativeTHH} by unfolding; the
connectivity assertions then follow in turn. 
Since the cotangent complex and its wedge powers are left Kan extended from
finitely generated polynomial algebras, the last assertion follows too. 
\end{proof}

\subsection{Preliminary connectivity bounds}

We use the spectral sequence of Krause--Nikolaus~\cite{KN} to obtain a
relationship between the 
relative and absolute $\THH$. 

\begin{proposition}[Relative versus absolute $\THH$] 
\label{relativevsabsolutegr}
If $R \in \qqrsp{\mathcal{O}_K}$,
then there exist natural surjective maps 
$f_n \colon \pi_{2n} \THH(R/\mathbb{S}[z];\mathbb{Z}_p) \to \pi_{2n-2}
\THH(R/\mathbb{S}[z];\mathbb{Z}_p)$
and natural isomorphisms 
$\pi_{2n} \THH(R; \mathbb{Z}_p) \simeq \mathrm{ker}(f_n)$. 
\end{proposition} 
\begin{proof} 
Recall that both $\THH(R; \mathbb{Z}_p)$ and $\THH(R/\mathbb{S}[z]; \mathbb{Z}_p)$ are
concentrated in even degrees since $R \in \qqrsp{\mathcal{O}_K}$ (see
\cite[Theorem 7.1]{BMS2} and \Cref{grrelativeTHH}). 
Therefore, the result follows directly from \cite[Prop.~4.1]{KN}; the spectral
sequence of \emph{loc.~cit.} must degenerate after the first differential, and the
maps of the first differential must be surjective or one would have odd degree
contributions to $\THH(R; \mathbb{Z}_p)$. 
\end{proof} 

The following fiber sequence \eqref{grfibseq} will be the basic tool in obtaining connectivity bounds
on the filtration on $\THH$ and its variants.

\begin{corollary}[Connectivity of the filtration on $\THH$] 
\label{filtconnTHHbounds}
If $R \in \qqsyn{\mathcal{O}_K}$, then for each $n$ there is a natural fiber sequence
\begin{equation} \label{grfibseq}
\mathrm{gr}^n \THH(R; \mathbb{Z}_p) \to 
\mathrm{gr}^n\THH(R/\mathbb{S}[z]; \mathbb{Z}_p) \to \mathrm{gr}^{n-1} \THH(R/\mathbb{S}[z];
\mathbb{Z}_p)[2].\end{equation}
In particular, we have $\mathrm{gr}^n \THH(R; \mathbb{Z}_p)$,
$\mathrm{Fil}^{\geq n} \THH(R; \mathbb{Z}_p) 
\in \sp_{\geq n}$ for any $R \in \qqsyn{\mathcal{O}_K}$.  
Finally, the functors $R \mapsto 
\mathrm{gr}^n \THH(R; \mathbb{Z}_p)$ and
$R\mapsto\mathrm{Fil}^{\geq n} \THH(R; \mathbb{Z}_p) $ on $\qsyn$ are left Kan extended 
from finitely generated $p$-complete polynomial $\mathbb{Z}_p$-algebras, as functors to $p$-complete spectra. 
\end{corollary} 

\begin{proof} 
The fiber sequence follows from 
\Cref{relativevsabsolutegr} by unfolding in $R$. 
The connectivity assertion 
for $\mathrm{gr}^n \THH(R; \mathbb{Z}_p)$
then follows from \Cref{2grrelativeTHH}; the assertion for $\mathrm{Fil}^{\geq
n} \mathrm{THH}(R; \mathbb{Z}_p)$ then follows since the filtration is complete. 
The Kan extension assertion for $\mathrm{Fil}^{\geq n} \THH(R;
\mathbb{Z}_p)$ also follows from 
the one for $\mathrm{Fil}^{\geq n} \THH(R/\mathbb{S}[z]; \mathbb{Z}_p)$
as in 
\Cref{filtongrTHH} (taking $\mathcal{O}_K = \mathbb{Z}_p$). 
\end{proof} 

\begin{corollary}[Connectivity bounds for $\mathcal{N}^i \Prism_R$] 
\label{connectivityNi}
\begin{enumerate}
    \item[{\rm (1)}] 
If $R \in \qsyn$,
then $\mathcal{N}^n \Prism_R \in D^{\leq n}(\mathbb{Z}_p)$. 
\item[{\rm (2)}]
If $R \to R'$ is a surjective map in $\qsyn$, then 
$\mathrm{fib}( \mathcal{N}^n \Prism_R \to \mathcal{N}^n \Prism_{R'}) \in D^{\leq
n}(\mathbb{Z}_p)$. 
\end{enumerate}
\end{corollary} 
\begin{proof} 
Part (1) is a special case of \Cref{filtconnTHHbounds} (take $\mathcal{O}_K = \mathbb{Z}_p$ and $\pi = p$). 

For part (2), note that the hypothesis implies that $R\to R'$ induces a surjection on $H^0$ of $p$-completed cotangent complexes over $\mathbb Z_p$, and similarly on any wedge power. It then follows from \Cref{filtongrTHH} that
$\mathrm{fib}( \gr^n \THH(R/\mathbb{S}[z]; \mathbb{Z}_p)[-2n] \to 
\gr^n \THH(R'/\mathbb{S}[z]; \mathbb{Z}_p)[-2n]) \in D^{\leq n}(\mathbb{Z}_p)$, whence we conclude by \eqref{grfibseq}.
\end{proof} 

Our main general connectivity bound is \Cref{cofiberbound1} below. 
To formulate it, we need to be able to twist the ideal 
$I \subset \Prism_R$ from \Cref{consprism}.
First, we observe that this ideal is also trivialized after base-change along
$a_R$. 

\begin{lemma} \label{lemmatrivialised}
Let $R \in \qlrsp_{\mathbb{Z}_p}$, and let $I \subset \Prism_R$ denote the ideal
defining the prism structure. 
Then there is a natural isomorphism $I \otimes_{\Prism_R} R \simeq R$, i.e.,
the ideal is naturally trivialized after base-change along $a_R\colon \Prism_R \to R$. 
\end{lemma} 
\begin{proof} 
Observe that the base-change $I \otimes_{\Prism_R} R $ defines a functorial
choice of invertible $R$-module, for any $R \in \qlrsp_{\mathbb{Z}_p}$. 
By faithfully flat descent, we obtain for any $R \in
\mathrm{qSyn}_{\mathbb{Z}_p}$ a 
choice of invertible $R$-module, which is functorial in $R$. 
Choosing a trivialization over $R = \mathbb{Z}_p$, we obtain a functorial
trivialization everywhere. 
\end{proof} 

\begin{definition}[Twisting by $I$] \label{defItwist}
For $s, i , n \geq 0$, 
we let $R \mapsto I^s \mathcal{N}^{\geq n} \Prism_R \left\{i\right\}$ denote
the $D(\mathbb{Z}_p)^{\geq 0}$-valued sheaf on 
$\mathrm{qSyn}_{\mathbb{Z}_p}$
defined by unfolding 
the discrete sheaf on $\qlrsp_{\mathbb{Z}_p}$ defined by the aforementioned formula, for
$I \subset \Prism_R$ the ideal defining the prismatic structure. 
For $R \in \qlrsp_{\mathbb{Z}_p}$, 
since $I$ defines a Cartier divisor in $\Prism_R$, we have
$I^s \mathcal{N}^{\geq n} \Prism_R \left\{i\right\} \simeq 
I^s \otimes_{\Prism_R }\mathcal{N}^{\geq n} \Prism_R \left\{i\right\}$. 
\end{definition} 

\begin{proposition}[Connectivity of Nygaard quotients] 
\label{cofiberbound1}
Let $R \in \mathrm{qSyn}_{\mathbb{Z}_p}$ and $i, n, s \geq 0$. Then the cofiber
$I^s\Prism_R\left\{i\right\}/ I^s \mathcal{N}^{\geq n} \Prism_R \left\{i\right\}$
belongs to $D^{\leq n-1}( \mathbb{Z}_p)$. 
Moreover, this cofiber is left Kan extended from finitely generated
$p$-complete polynomial
$\mathbb{Z}_p$-algebras. 
\end{proposition} 
\begin{proof}
By d\'evissage it suffices to show that 
    $I^s\mathcal{N}^n\Prism_R\left\{i\right\} \in D^{\leq n}(R)$ for each $n\ge0$ and
that this is left Kan extended from finitely generated $p$-complete polynomial $\mathbb{Z}_p$-algebras. Here we write $I^s\mathcal{N}^n\Prism_R\left\{i\right\}$ for the unfolding from $\qlrsp_{\mathbb{Z}_p}$ of $I^s \otimes_{\Prism_R }\mathcal{N}^{n} \Prism_R \left\{i\right\}$.
However, the twists here are trivialized by \Cref{lemmatrivialised} since
$\mathcal{N}^n \Prism_R$ is an $R$-module, 
so that
$I^s\mathcal{N}^n\Prism_R\left\{i\right\} \simeq \mathcal{N}^{n}
\Prism_R$.
Thus, the result follows from \Cref{connectivityNi} as well as
\Cref{filtconnTHHbounds} (for the left Kan extension assertion). 
\end{proof}

We finish this subsection by recording a connectivity bound that depends on the number of generators of the cotangent complex (we will not use this result in the paper, but note that it implies in particular that $\Prism_{\mathcal O_K}\in D^{\le 1}(\mathbb Z_p)$).

\begin{lemma} 
Let $R$ be a commutative ring, and let $M \in D^{\leq 0}(R)$. 
Suppose $H^0(M)$ is generated by $d$ elements. 
Then for all $j$, $(\bigwedge^j M) [-j] \in D^{\leq d}(R)$. 
\label{Easyconnestimate}
\end{lemma} 

\begin{proof} 
The result is clear if $M = R^d$ itself.
In general, we have a map $R^d \to M$ inducing a surjection on $H^0$, so the cofiber $F$ of
the map satisfies $F \in D(R)^{\leq -1}$. 
It follows that $\bigwedge^{j'} F [-j'] \in D(R)^{\leq 0}$ for all $j'$ by
standard connectivity estimates (see \cite[Sec.~25.2.4]{SAG} for an account).  
Using the natural filtration on $\bigwedge^j  M [-j]$ with associated graded
terms
$\bigwedge^{j'} F[-j'] \otimes_R \bigwedge^{j-j'} R^d [-(j-j')]$, 
the result easily follows. 
\end{proof} 

\begin{proposition} 
\label{connboundgenerators}
Let $R  \in \qqsyn{\mathcal{O}_K}$, let $n,i\ge0$, and suppose that $H^0( \widehat{L_{R/\mathcal{O}_K}})$ is generated by $d$ elements. Then $\mathcal{N}^n \Prism_R$ and  $\mathcal{N}^{\geq n} \Prism_R \left\{i\right\}$ lie in $D^{\leq d+1}(R)$. 
\end{proposition}
\begin{proof}
By \Cref{filtongrTHH},  $\gr^n \THH(R/\mathbb{S}[z]; \mathbb{Z}_p)$
has a finite filtration with graded pieces
$ \widehat{\bigwedge^j L_{R/\mathcal{O}_K}} [2n-j]$ for $0 \leq j \leq n$. 
By \Cref{Easyconnestimate}, we find that 
$\gr^n \THH(R/\mathbb{S}[z]; \mathbb{Z}_p) \in D^{\leq d-2n}(R)$. 
Using the fiber sequence \eqref{grfibseq}, we find now that $\gr^n \THH(R;
\mathbb{Z}_p) \in D^{\leq d - 2n+1}(R)$. Shifting by $2n$ the result now
follows for $\mathcal N^n\Prism_R$. The same connectivity bound then follows for each $\mathcal{N}^{\geq n} \Prism_R \left\{i\right\}/\mathcal{N}^{\geq n+r} \Prism_R \left\{i\right\}$ by d\'evissage, and then for $\mathcal{N}^{\geq n} \Prism_R \left\{i\right\}$ by passing to the limit.
\end{proof}

\subsection{Frobenius nilpotence on $\Prism_R/p$, and proof of \Cref{Zpibound}
(2)}
In this subsection, we record some results about the contracting property of Frobenius on
$\Prism_R/p$ and use it to prove part of \Cref{Zpibound}. 
If $R \in \qlrsp_{\mathbb{Z}_p}$ is a $p$-torsion free quasiregular
semiperfectoid ring, then both $\Prism_R$ and all graded steps $\mathcal N^n\Prism_R$ of the Nygaard filtration are $p$-torsion free (e.g., because $\THH_*(R; \mathbb{Z}_p)$ is $p$-torsion free and concentrated in even degrees). 
 For $i, r \geq 0$, we will consider the maps 
\begin{equation} \label{canphiIadic} \can, \phi_i\colon \mathcal{N}^{\geq i+r} \Prism_R\left\{i\right\}/p  \to
\Prism_R\left\{i\right\}/p . \end{equation}
and show that both maps respect the $I$-adic filtration from \Cref{defItwist}, with
$\phi_i$ inducing the zero map on associated graded pieces in positive degrees
(\Cref{Iadicfiltlemma}). We will show in addition that
$\mathrm{can} - \phi_i$ induces an automorphism of $\mathcal{N}^{\geq i + r}
\Prism_R\left\{i\right\}/p$ for $r \gg 0$ (\Cref{Nygmodulop}). 

\begin{proposition} 
\label{contractingFrob}
Let $R \in \qlrsp_{\mathbb{Z}_p}$ and $i,r\ge0$. Then
the map 
\( \phi_i\colon  \mathcal{N}^{\geq i} \Prism_R\left\{i\right\} \to
\Prism_R\left\{i\right\}
\)
 carries 
$\mathcal{N}^{\geq i + r} \Prism_R\left\{i\right\}  $ into 
$I^{r}\Prism_R\left\{i\right\}$. 
\end{proposition} 
\begin{proof} 
Let $R_0$ be a perfectoid ring mapping to $R$ and fix $\xi, \widetilde{\xi} \in \ainf(R_0)$ as usual. 
Then \cite[Sec.~6]{BMS2} we have an isomorphism
$\TC^-_*(R_0; \mathbb{Z}_p) \simeq \ainf(R_0)[u, v]/(uv - \xi)$ for $|u| = 2, |v|
= -2$. 
In this case, the filtration on $\mathcal{N}^{\geq i} \Prism_R\left\{n\right\}
\simeq \pi_{2i}( \TC^-(R; \mathbb{Z}_p))$
is the filtration by powers of $v$:
$$\mathcal{N}^{\geq i+r} \Prism_R \left\{i\right\} = v^r  \pi_{2i+2r} \TC^-(R;
\mathbb{Z}_p) \subset \pi_{2i} \TC^-(R; \mathbb{Z}_p).$$
But (as in \emph{loc.~cit.}) the cyclotomic Frobenius carries $v$ to a multiple of $\phi(\xi) =
\widetilde{\xi}$ in
$\pi_{-2} \TP(R_0; \mathbb{Z}_p)$; recalling that $I=(\tilde\xi)$, the result follows.  
\end{proof} 

\begin{construction}[The $I$-adic filtrations modulo $p$] 
\label{Iadicmodulop}
Let $R \in \qlrsp_{\mathbb{Z}_p}$. For each $i,s,r\ge0$, the map $\phi_i\colon\mathcal N^{\ge i+r}\Prism_R\{i\}/p\to\Prism_R\{i\}/p$ is Frobenius semi-linear by \Cref{consprism}(1), and so carries $I^s(\mathcal{N}^{\geq i+r}
 \Prism_R \left\{i\right\}/p)$ to $I^{ps + r}(\Prism_R\left\{i\right\}/p)$ by
 \Cref{contractingFrob}. But we have seen in \Cref{TPTHHtcp} that
 $(p,\tilde\xi)$ and $(\tilde\xi,p)$ are regular sequences on $\Prism_R$, whence
 the canonical maps are isomorphisms $I\otimes_{\Prism_R}\Prism_R/p\simeq
 I(\Prism_R/p)\simeq I/p$, and similarly for any power of $I$ and Breuil--Kisin
 twist of $\Prism_R$.  We thus get maps 
\begin{equation} \label{canphinfilt} \can, \phi_i\colon I^s \otimes_{\Prism_R}   \mathcal{N}^{\geq i+r} \Prism_R
\left\{i\right\}/p \to
I^s \otimes_{\Prism_R} 
\Prism_R \left\{i\right\}/p.
\end{equation}
\end{construction} 

For convenience, we record what we have proved about the interaction of the
Frobenius and the $I$-adic filtration, as it will be used to prove 
\Cref{connboundtransferfree}:

\begin{proposition}[The canonical and Frobenius map are $I$-adically filtered modulo
$p$] 
\label{Iadicfiltlemma}
Let $R \in \qlrsp_{\mathbb{Z}_p}$, and let $i,r\ge0$. 
The maps 
\eqref{canphiIadic} upgrade to the structure of filtered maps with respect to the
$I$-adic filtrations on both sides, i.e., there are compatible 
maps for each $s \geq 0$, 
\[\can, \phi_i\colon I^s \otimes_{\Prism_R} \mathcal{N}^{\geq i+r}
\Prism_R\left\{i\right\}/p \to 
I^s \otimes_{\Prism_R}\Prism_R\left\{i\right\}/p.\] 
Furthermore, the map  $\phi_i$ induces the zero 
map on associated graded pieces unless $s = r = 0$. 
\end{proposition} 
\begin{proof} 
In \Cref{Iadicmodulop} we constructed the maps and showed that in fact $\phi_i$  has image in $I^{ps+r}\otimes_{\Prism_R}\Prism_R\left\{i\right\}/p$.
\end{proof} 

For the moment we need the following consequence of our arguments.

\begin{corollary}[The Nygaard filtrations modulo $p$]
\label{Nygmodulop}
Let $R \in \qlrsp_{\mathbb{Z}_p}$ and $i\ge0$. For $r\gg0$ (independent of $R$), the map
$\can - \phi_i\colon 
\mathcal{N}^{\geq i} \Prism_R\left\{i\right\}/p \to \Prism_R\left\{i\right\}/p$ 
 carries $\mathcal{N}^{\geq i+r} \Prism_R\left\{i\right\}/p $
 isomorphically onto itself. Consequently, for such $r$, one has a natural isomorphism
 \begin{equation} \label{nilpFpn} \mathbb{F}_p(i)(R) \simeq 
\mathrm{fib}\left( \can - \phi_i\colon (\mathcal{N}^{\geq i} \Prism_R\left\{i\right\}/
\mathcal{N}^{\geq i+r} \Prism_R \left\{i\right\})/p \to 
( \Prism_R\left\{i\right\}/
\mathcal{N}^{\geq i+r} \Prism_R \left\{i\right\}) /p\right)
 .\end{equation}
\end{corollary} 
\begin{proof} 
This is \cite[Lemma 7.22]{BMS2}. 
By \Cref{contractingFrob} above (and choosing as usual a
perfectoid ring $R_0$ mapping to $R$), 
we find that $\varphi_i$ carries $\mathcal{N}^{\geq i +
r } \Prism_R \left\{i\right\}/p$ 
(equivalently, $v^r 
\mathcal{N}^{\geq i+r} \Prism_R\left\{i+r\right\}/p $)
into multiples of $\widetilde{\xi}^{r} \Prism_R
\left\{i\right\}/p$. Since we are working modulo $p$, we have 
$\widetilde{\xi}^{r} \Prism_R
\left\{i\right\}/p = 
\xi^{rp} \Prism_R\left\{i\right\}/p
 \subset \mathcal{N}^{\geq rp} \Prism_R\left\{i\right\}/p$.  For $r \gg 0$, this
 is contained in $\mathcal{N}^{\geq i  + r + 1} \Prism_R\left\{i\right\}/p$, whence we have shown that 
$\phi_i$ carries $\mathcal{N}^{\geq i+r} \Prism_R\left\{i\right\}/p$ to 
$\mathcal{N}^{\geq i+r+1} \Prism_R\left\{i\right\}/p$. 

It follows that $\can - \phi_i$ carries 
$\mathcal{N}^{\geq i+r} \Prism_R\left\{i\right\}/p$ into itself, and it differs from the identity by a topologically nilpotent endomorphism
of 
$\mathcal{N}^{\geq i+r}\Prism_R\left\{i\right\}/p$ with respect to the Nygaard
filtration. Therefore it is an isomorphism and the result follows.
\end{proof} 

\begin{proposition}[A criterion for being left Kan extended] 
\label{LKElem}
Let $F, G\colon \qsyn \to D(\mathbb{Z}_p)$ be $p$-complete quasisyntomic sheaves 
equipped with complete descending $\mathbb{Z}_{\geq 0}$-indexed filtrations $\mathrm{Fil}^{\geq \star} F$ and
$\mathrm{Fil}^{\geq \star} G$. 
Let $F \to G$ be  a map of functors (not necessarily filtration-preserving). If
\begin{enumerate}
    \item[{\rm (1)}] for $R \in \qlrsp_{\mathbb{Z}_p}$, the objects
$\gr^r F(R)$ and $\gr^r G(R)$ are discrete, $p$-complete, and $p$-torsion free (and therefore so
are $F(R), G(R)$),
\item[{\rm (2)}]
each of the associated graded terms $\mathrm{gr}^r F$ and $\mathrm{gr}^r G$ is
 left
Kan extended from finitely generated $p$-complete polynomial $\mathbb{Z}_p$-algebras to the
$p$-complete derived category, and
\item[{\rm (3)}] 
there exists $N$ such that for $r \geq N$ and 
for $R \in \qlrsp_{\mathbb{Z}_p}$, the map $F(R)/p \to G(R)/p$ 
carries the submodule $\mathrm{Fil}^{\geq r} F(R)/p $ isomorphically to $\mathrm{Fil}^{\geq
r} G(R)/p$ (so that $\mathrm{Fil}^{\geq N}F(R)/p \to \mathrm{Fil}^{\geq
N}G(R)/p$ is an equivalence of filtered
abelian groups),
\end{enumerate}
then $\mathrm{fib}( F \to G)$  is $p$-completely left Kan extended
from finitely generated $p$-complete polynomial $\mathbb{Z}_p$-algebras. 
\end{proposition} 
\begin{proof} 
It suffices to check that 
$\mathrm{fib}( F \to G)/p$ is left Kan extended from finitely generated
$p$-complete polynomial algebras.
Let $F', G'$ denote the functors on $\qsyn$ obtained by restricting $F, G$ to
finitely generated $p$-complete polynomial algebras and then left Kan extending to
$\qsyn$, with their left Kan extended filtrations.  
Our assumptions imply that $F, G$ are the respective completions of $F', G'$ with respect
to their filtrations. 

It suffices to check that 
the natural map induces an equivalence $\mathrm{fib}(F' \to G')/p \simeq
\mathrm{fib}(F \to G)/p$. 
We claim that for any $R \in \qsyn$, there are natural commutative diagrams,
compatible in $r \geq N$,
    \begin{equation} \label{diags}\begin{gathered} \xymatrix{
\mathrm{Fil}^{\geq r} F'(R)/p \ar[d]  \ar[r]^{\simeq} & \mathrm{Fil}^{\geq r}
G'(R)/p \ar[d]  \\
F'(R)/p \ar[r] &  G'(R)/p.
    }\end{gathered}\end{equation}
In fact, it suffices to prove this by left Kan extension for $R$
finitely generated $p$-complete polynomial
over $\mathbb{Z}_p$, and then  we can replace $F', G'$ by $F, G$. By descent for
$F, G$,
we can then reduce to $R \in \qlrsp_{\mathbb{Z}_p}$,  whence we have the desired
diagrams by hypothesis. 

Using the diagrams \eqref{diags}, we find that 
there is a natural commutative diagram
\[ \xymatrix{
F'(R)/p \ar[d]  \ar[r] &  G'(R)/p \ar[d]  \\
F(R)/p \ar[r] &  G(R)/p 
}\]
which is homotopy cartesian (taking the inverse limit over $r$). 
We finally find that $\mathrm{fib}(F \to G) = \mathrm{fib}(F' \to G')$, which
is left Kan extended from finitely generated $p$-complete polynomial algebras as desired. 
\end{proof} 

\begin{proof}[Proof of \Cref{Zpibound}(2)]
We show that  $R \mapsto
\mathbb{Z}_p(i)(R)$, as a functor on $\qsyn$, is left Kan extended from
finitely generated 
$p$-complete polynomial $\mathbb{Z}_p$-algebras. 
Since $\mathbb{Z}_p(i)(R) = \mathrm{fib}( \can  - \phi_i\colon \mathcal{N}^{\geq i}
    \Prism_R\left\{i\right\} \to \Prism_R\left\{i\right\})$, we will apply
    \Cref{LKElem}
with $F = \mathcal{N}^{\geq i}\Prism\left\{i\right\}, G =
\Prism\left\{i\right\}$
using the Nygaard filtrations
$R \mapsto \mathcal{N}^{\geq i+r}\Prism_{R}\left\{i\right\}$ on $\qsyn$.
Indeed, the associated graded terms for the Nygaard filtration are left Kan extended from
finitely generated $p$-complete polynomial algebras (\Cref{cofiberbound1}), and they are
torsion free on $R \in \qlrsp_{\mathbb{Z}_p}$. 
The last hypothesis follows from 
\Cref{Nygmodulop}. 
Then \Cref{LKElem} gives that the $\mathbb{Z}_p(i)$ are left Kan extended from
finitely generated $p$-complete polynomial algebras as desired. 
Since $R \mapsto \TC(R; \mathbb{Z}_p)$ is left Kan extended from finitely
generated $p$-complete polynomial rings (by \cite[Theorem G]{CMM} and since
$\mathrm{TC}(-; \mathbb{Z}_p)$ commutes with geometric realizations on
simplicial commutative rings), it follows inductively
that the constructions $R \mapsto \fil^{\geq i} \TC(R; \mathbb{Z}_p)$ are also left
Kan extended from finitely generated $p$-complete $\mathbb{Z}_p$-algebras. 
\end{proof}

For future reference, we observe that we can obtain a motivic filtration on
$\mathrm{TC}$ for any simplicial commutative ring, by left Kan extension. 
We let $\mathrm{SCR}$ denote the $\infty$-category of simplicial commutative
rings. 

\begin{construction}[{Left Kan extending to $\mathrm{SCR}$}]
\label{LKEmotivicfilt}
We have seen that the functor which sends $R \in \qsyn$ to the filtration
$\mathrm{Fil}^{\geq \star} \TC(R; \mathbb{Z}_p)$ is left Kan extended
from finitely generated $p$-complete polynomial algebras. 
Thus, we can left Kan extend to all $p$-complete simplicial commutative rings to obtain a functor
\( R \mapsto  \mathrm{Fil}^{\geq \star} \TC(R; \mathbb{Z}_p), \)
from $\mathrm{SCR}$ to $p$-complete filtered spectra, which commutes with
sifted colimits. We define functors $\mathbb{Z}_p(i)$ on $\mathrm{SCR}$ as
$\mathbb{Z}_p(i)(R) = \mathrm{gr}^i \mathrm{TC}(R; \mathbb{Z}_p)[-2i]$, or in other words by left Kan extending $\mathbb Z_p(i)$ from finitely generated $p$-complete polynomial algebras.

Once we complete the proof of \Cref{Zpibound} in the next subsection, it will
follow that $\mathrm{Fil}^{\geq i} \mathrm{TC}(R; \mathbb{Z}_p)$ (resp.~$\mathbb
Z_p(i)$) belongs to $\sp_{\geq
i-1}$ (resp.~$D^{\leq i+1}(\mathbb{Z}_p)$) for each $i$, by left Kan extending the connectivity estimate from the quasisyntomic case.

We also emphasize that the above proof of \Cref{Zpibound}(2) shows the following. Given $i\ge 0$, there is $r\gg0$ such that for all $p$-complete rings $R$ there is a natural expression  \begin{equation} \label{nilpFpngeneral} \mathbb{F}_p(i)(R) \simeq 
\mathrm{fib}\left( \can - \phi_i\colon (\mathcal{N}^{\geq i} \Prism_R\left\{i\right\}/
\mathcal{N}^{\geq i+r} \Prism_R \left\{i\right\}) \otimes_{\mathbb{Z}}^L
\mathbb{F}_p \to 
( \Prism_R\left\{i\right\}/
\mathcal{N}^{\geq i+r} \Prism_R \left\{i\right\}) \otimes_{\mathbb{Z}}^L
\mathbb{F}_p\right),
\end{equation}
where the two Nygaard quotients on the right side are defined by left Kan extension from finitely generated $p$-complete polynomial algebras.
\end{construction}

\subsection{Proofs of the connectivity bounds (\Cref{Zpibound}(1)) for the $\mathbb{Z}_p(i)$}
In this subsection, we complete the proof of 
\Cref{Zpibound}. 

\begin{lemma}[Connectivity lemma]
\label{connectivitynilpotencelemma}
Let $\mathrm{can}, \phi\colon \mathrm{Fil}^{\geq \star} M\to
\mathrm{Fil}^{\geq \star } N$ be maps of filtered objects in
$D(\mathbb{Z})$ (with underlying objects $M, N$). Suppose that
\begin{enumerate}
    \item[{\rm (1)}]  both filtrations are complete,
    \item[{\rm (2)}] $\phi$ induces the zero map on associated graded pieces, and
    \item[{\rm (3)}] there is a fixed $r$ such that, for each $s$,
the induced map $\mathrm{can}\colon \mathrm{Fil}^{\geq s} M \to
\mathrm{Fil}^{\geq s} N$ has fiber in
$D(\mathbb{Z})^{\leq r}$.
\end{enumerate}
Then $\mathrm{can} - \phi\colon M \to N$ has fiber in $D(\mathbb{Z})^{\leq r}$. \end{lemma} 
\begin{proof} 
The fiber $\mathrm{fib}( \can - \phi\colon M \to N)$ acquires the natural structure
of a filtered spectrum, since $\can, \phi$ are filtered maps. 
On graded pieces, we find 
$\mathrm{gr}^s \mathrm{fib}( \can - \phi\colon M \to N) \simeq \mathrm{gr}^s
\mathrm{fib}( \can\colon M \to N)$ since $\phi$ vanishes on associated graded
terms. In
particular, the associated graded terms belong to $D(\mathbb{Z})^{\leq r} $. 
Since the filtration on $\mathrm{fib}( \can - \phi)$ is complete, the
connectivity assertion on the fiber now follows from the analogous assertion on
associated graded terms. 
\end{proof}

\begin{proposition}[The $\mathbb{Z}_p(i)$ connectivity bound for $
\mathrm{qSyn}_{\mathbb{Z}_p}$]
\label{connboundtransferfree}
Let $R \in
\mathrm{qSyn}_{\mathbb{Z}_p}$. Then $\mathbb{Z}_p(i)(R) \in D^{\leq i+1}(
\mathbb{Z}_p)$ for each $i \geq 0$. 
\end{proposition} 
\begin{proof} 
First, we recall from the proof of \Cref{cofiberbound1} that 
the inclusion $\mathcal{N}^{ \geq i+1} \Prism_R\left\{i\right\}  \to \mathcal{N}^{\geq
i}\Prism_R \left\{i\right\}$ has cofiber in $D^{\leq i}( \mathbb{Z}_p)$. 
Using the resulting cofiber sequence, it thus suffices to show that the fiber of 
$\mathrm{can} - \phi_i\colon\mathcal{N}^{\geq i+1} \Prism_R\left\{i\right\}\to \Prism_R\left\{i\right\}$
belongs to $D^{\leq i+1}( \mathbb{Z}_p)$. 
Since everything is $p$-complete, it suffices to check this with mod $p$
coefficients.

We consider the two maps
\[ \mathrm{can}, \phi_i\colon \mathcal{N}^{\geq i+1} \Prism_R\left\{i\right\}/p \to
\Prism_R\left\{i\right\}/p.  \]
Unfolding \Cref{Iadicfiltlemma} shows that these upgrade to maps $\mathrm{can}, \phi_i\colon I^s\mathcal{N}^{\geq i+1} \Prism_R\left\{i\right\}/p \to
I^s\Prism_R\left\{i\right\}/p$ for all $s\ge0$, i.e., of $I$-adically filtered
    objects, and that the map $\phi_i$ acts trivially on associated graded
    pieces.
%Here both sides upgrade to filtered objects via the $I$-adic filtrations (i.e., tensoring with powers of $I$), and the map $\phi_i$ acts trivially on associated gradeds with respect to these filtrations (\Cref{Iadicfiltlemma}, via quasisyntomic descent). 
Furthermore, 
for each $s\ge0$,
    the fiber of $\mathrm{can}\colon I^s 
\mathcal{N}^{\geq i+1} \Prism_R\left\{i\right\}/p \to
    I^s \Prism_R\left\{i\right\}/p$
	 belongs to $D^{\leq i+1}( \mathbb{Z}_p)$ by
\Cref{cofiberbound1}.  
\Cref{connectivitynilpotencelemma} (whose hypothesis (1) is satisfied by $\tilde\xi$-adic completeness in the case of $R\in \qlrsp_{\mathbb{Z}_p}$) now shows that 
$\mathrm{can} - \phi_i\colon\mathcal{N}^{\geq i+1} \Prism_R\left\{i\right\}\to \Prism_R\left\{i\right\}$
belongs to $D^{\leq i+1}( \mathbb{Z}_p)$, as desired.
\end{proof}

\begin{proof}[Proof of \Cref{Zpibound}(1)]
We wish to show that $\mathbb{Z}_p(i) \in D^{\leq
i+1}(\mathbb{Z}_p)$. 
But we have already proved part (2) of \Cref{Zpibound}, so the problem reduces to the case of finitely generated $p$-completely polynomial
rings over $\mathbb{Z}_p$, which is covered by \Cref{connboundtransferfree}. 
\end{proof}

\subsection{Rigidity, and proof of \Cref{rigidity}}

In this subsection, we give a proof of \Cref{rigidity}. Our strategy is to first
prove a continuity statement, after which N\'eron--Popescu and left Kan
extension arguments reduce the general case to that of a square-zero extension. In that case we use the automatic gradings that exist and argue with the pro-nilpotence of
Frobenius.
Recall that a ring $R$ is said to be {\em F-finite} if $R/p$ is finitely generated over its subring of $p^{\textrm{th}}$-powers.
The next result is an analog on graded pieces of 
\cite[Th.~4.5]{Morrow_Dundas}. 

\begin{proposition}[Continuity]\label{prop_continuity}
Let $R$ be a noetherian, F-finite, $p$-complete ring, and $I\subseteq R$ an ideal such that $R$ is $I$-adically complete. Then the natural map $\mathbb Z_p(i)(R)\to\varprojlim_s \mathbb Z_p(i)(R/I^s)$ is an equivalence for any $i\ge0$.
\end{proposition}
\begin{proof}
From \Cref{defBMSZPi} and completeness of the Nygaard filtration, it is enough
to prove the analogous continuity for each $\mathcal
N^n\Prism\{i\}\simeq\mathrm{gr}^n \THH(-; \mathbb{Z}_p)[-2n]$. Then \Cref{filtconnTHHbounds} reduces us further to continuity for each $\mathrm{gr}^n\THH(-/\mathbb{S}[z]; \mathbb{Z}_p)$, and finally the filtration of \Cref{filtongrTHH} reduces the problem to continuity for each $\widehat{\bigwedge^n L_{-/\mathbb Z_p}}$.

A cell attachment lemma of Andr\'e and Quillen (see
\cite[Th.~4.4(i)]{Morrow_pro_H_unitality} for a presentation in this context)
shows that all cohomology groups of all wedge powers of $L_{(R/I^s)/R}$ are pro
zero in $s$ (except $H^0$ of $\bigwedge^0 L$). So the transitivity fiber
sequence 
shows that $\bigwedge^nL_{R/\mathbb
Z_p}\otimes^L_RR/I^s\to\bigwedge^nL_{(R/I^s)/\mathbb Z_p}$ is a pro isomorphism on
all cohomology groups. This reduces the problem to showing that
$\bigwedge^nL_{R/\mathbb Z_p}\to \varprojlim_s\bigwedge^nL_{R/\mathbb
Z_p}\otimes_R^LR/I^s$ is an equivalence after $p$-adic completion. Since $R$ has
bounded $p$-power torsion, this derived $p$-adic completion may be equivalently
computed as $\varprojlim_r (-\otimes^L_R R/p^r)$. Exchanging the limits, it is
enough to show that $M\simeq \varprojlim_sM\otimes^L_RR/I^s$, where
$M=\bigwedge^nL_{R/\mathbb Z_p}\otimes^L_RR/p^r$ for any $n\ge0$, $r\ge1$. But
this follows from the facts that $R$ is noetherian, that $M$ is bounded above
(cohomologically), and that its cohomology groups are finitely generated
$R$-modules, cf.~\cite[Th.~3.6]{Morrow_Dundas} (which uses $F$-finiteness).
\end{proof}

\begin{remark}
As usual, one can prove stronger continuity statements when $I=(p)$. For example,
given a $p$-complete ring $R$ with bounded $p$-power torsion, we claim that
$\mathbb Z_p(i)(R)/p^r\to\{\mathbb Z_p(i)(R/p^s)/p^r\}_s$ induces an isomorphism of pro groups on all cohomology groups for any
fixed $r\ge1$. To prove this we reduce to the case $r=1$ and appeal to
(\ref{nilpFpngeneral}) (instead of the completeness of the Nygaard filtration
used at beginning of the proof of \Cref{prop_continuity}) to again reduce to the
analogous assertion for graded pieces of the Nygaard filtration. Then argue as
in the proof of \Cref{prop_continuity} (the assumption that $R$ has bounded
$p$-power torsion implies that the ideal $(p)$ is pro Tor-unital in the sense of
\cite{Morrow_pro_H_unitality}, which is needed to verify pro vanishing of
$\bigwedge^nL_{(R/p^s)/R}$) to reduce to the analogous statement about
$(\bigwedge^nL_{R/\mathbb Z_p})/p\to \{(\bigwedge^nL_{R/\mathbb
Z_p})/p\otimes_R^L R/p^s\}_s$, which follows from boundedness of the $p$-power torsion in $R$.
\end{remark}

\begin{proposition}\label{prop_reduction}
\Cref{rigidity} holds in the  special case of henselian pairs of the form $R=A
\oplus N$, $I=N$, where $A$ is the $p$-completion of a finitely generated
$\mathbb Z_p$-polynomial algebra, $N$ is a finitely generated $A$-module, and
$R$ is the trivial square-zero extension of $A$ by $N$.  
\end{proposition}

The proof of 
\Cref{prop_reduction} will be given below. 
Using 
\Cref{prop_reduction}, we explain how to deduce 
\Cref{rigidity}. 

\begin{proof}[Proof of \Cref{rigidity}]
First, note that it is equivalent to prove that the homotopy fiber of \Cref{rigidity} mod $p$ belongs to $D^{\le i}(\mathbb Z_p)$, i.e., we may replace $\mathbb Z_p(i)$ by $\mathbb F_p(i)=\mathbb Z_p(i)/p$.

We consider the functor on 
$\mathbb{Z}$-algebras, 
$R \mapsto F(R) \stackrel{\mathrm{def}}{=}\mathbb{F}_p(i)( \widehat{R}_p)$ (where $\hat{R}_p$ denotes the
derived $p$-completion of $R$). By \Cref{Zpibound}, the functor $F$ commutes with filtered colimits
in $R$. 
It suffices to show that if $(R, I)$ is a henselian pair, then \begin{equation}
\label{cohbnd} \mathrm{fib}(
F(R) \to F(R/I)) \in D^{\leq i}(\mathbb{F}_p) .\end{equation}

First, we prove \eqref{cohbnd} in case $I \subset R$ is a nilpotent ideal. 
By transitivity and an easy induction, it suffices to assume $I^2 =0$. 
Next we apply a standard trick to reduce to the case that $R\to R/I$ is split:
choose a simplicial resolution $P_\bullet \to R/I$ by polynomial $\mathbb
{Z}$-algebras (possibly in infinitely many variables), and let $Q_\bullet$ be
the fiber product along $R\to R/I$. Then each $Q_j\to P_j$ has kernel $I$ and
admits a section, since $P_j$ is a polynomial algebra. Taking the geometric
realization using \Cref{Zpibound}(2), we thus reduce to proving
\eqref{cohbnd} for each pair $(Q_j=P_j\oplus I,I)$.
But we can write $P_j$ as a filtered colimit of polynomial $\mathbb{Z}$-algebras
on finitely many variables and $I$ as a filtered colimit of finitely generated
modules. 
Using \Cref{Zpibound}(2) again (which shows that $F$ commutes with filtered
colimits), 
and \Cref{prop_reduction}, we conclude \eqref{cohbnd} for $I \subset R$
nilpotent.

Second, we prove \eqref{cohbnd} in case $R$ is noetherian and  $F$-finite, and 
$R$ is $I$-adically complete. 
In this case, 
\Cref{prop_continuity} shows that 
$F(R) \simeq \varprojlim F(R/I^n)$. 
We consider the tower (in $n$), 
$T_n = \mathrm{fib}( F(R/I^n) \to F(R/I))$; the fiber of each successive map
$T_n \to T_{n-1}$ belongs to $D^{\leq i}( \mathbb{F}_p)$, and thus we get that
$\varprojlim_n T_n = \mathrm{fib}( F(R) \to F(R/I)) \in D^{\leq
i}(\mathbb{F}_p)$ as desired.

Finally, suppose $(R, I)$ is a general henselian pair; we prove \eqref{cohbnd}.
Since $F$ commutes with filtered colimits, it suffices to assume that the
pair $(R, I)$ is the
henselization of a finitely generated $\mathbb{Z}$-algebra $R_0$ along an ideal
$I_0 \subset R_0$. 
By the previous paragraph, we have
$\mathrm{fib}( F( \hat{R}_I) \to F(R/I)) \in D^{\leq i}(\mathbb{F}_p)$, since
$\hat{R}_I$ is an $F$-finite, noetherian ring. 
Now $R_0$ is an excellent ring as a finitely generated $\mathbb{Z}$-algebra; since $R_0 \to R$ is ind-\'etale, $R$ is also excellent
\cite{Greco1976}. 
It follows that $R \to \widehat{R}_I$ is geometrically regular \cite[7.8.4(v)]{EGA_IV}
and is therefore a filtered colimit of smooth maps 
by N\'eron--Popescu desingularisation \cite{Popescu1, Popescu2} \cite[Tag
07BW]{stacks-project}); each of these maps necessarily admits a section. 
In particular, 
the map 
$\mathrm{fib}( F(R) \to F(R/I)) \to \mathrm{fib}( F(\widehat{R}_I) \to F(R/I))$
is a filtered colimit of maps, each of which admits a section. 
Since we have just seen that the target of this map belongs to $D^{\leq
i}(\mathbb{F}_p)$, the source does 
as well. 
\end{proof}

To complete the proof of \Cref{rigidity} by treating the pairs in the statement of  \Cref{prop_reduction}, we will exploit the presence of the grading induced by the variables.

\begin{remark}[$\THH$ for graded rings] \label{rem_THHofgraded}
In the remainder of this subsection we will systematically use graded objects
indexed over a commutative monoid $M$ (which will be $\mathbb Z[1/p]_{\ge0}$ or
$\mathbb Z_{\ge0}$): an {\em $M$-graded object} in a ($\infty$-)category
$\mathcal{C}$ is
by definition a functor $R_\star\colon  M\to \mathcal{C}$. When $\mathcal{C}$ is
symmetric monoidal, then so is the resulting category
$\mathrm{Fun}(M,\mathcal{C})$ of $M$-graded objects, under the Day convolution product, and an {\em $M$-graded ring} is then a ($\mathbb E_\infty$-, etc.) monoid object in $M$-graded objects. 

The {\em underlying object} $R$ of a graded object $R_\star$ is by definition
$\bigoplus_{m\in M}R_m$, when it exists. When all direct sums exist, the functor
$R_\star \mapsto R$ is conservative (and faithful when $\mathcal{C}$ is
an ordinary category).
We will be particularly interested in {\em $p$-complete graded rings}, by which
we mean a graded ring $R_\star$ in the category of $p$-complete abelian groups;
the underlying object $R$ is then the $p$-completed direct sum
$\widehat{\bigoplus}_{m\in M}R_m$, which is itself a $p$-complete ring. We
sometimes abusively identify $\widehat{\bigoplus}_{m\in M}R_m$ with $R_\star$
itself; this is a mild abuse of notation given that the functor
$R_\star \mapsto\widehat{\bigoplus}_{m\in M}R_m$ is conservative and faithful.

An $M$-graded commutative ring $R_\star$ may of course be viewed  as an $M$-graded
$\mathbb E_\infty$-ring in spectra, i.e., an $\mathbb E_\infty$-monoid in the
    symmetric monoidal stable $\infty$-category $\mathrm{Fun}(M,\Sp)$. By
    \Cref{app:gradedcyc}, we may then
form the $S^1$-equivariant object $\THH(R_\star)\in \mathrm{Fun}(M,\Sp)^{B S^1}$ and
the associated homotopy fixed points $\TC^-(R_\star)$, homotopy orbits
$\THH(R_\star)_{hS^1}$, and Tate construction $\TP(R_\star)$, all of which are
$M$-graded spectra. Note that the underlying spectrum of $\THH(R_{\star})$
is the $\THH$ of
the underlying ring spectrum of $R_\star$ 
because the underlying spectrum functor preserves tensor products and colimits;
this is not true for $\TC^-, \TP$ because the underlying spectrum functor need
not preserve limits (e.g., $S^1$-homotopy fixed points). 
\end{remark}

\begin{construction}[\cite{BMS2} for graded rings] \label{consBMS2graded}
Assume that the monoid $M$ is uniquely $p$-divisible, such as $\mathbb Z[1/p]_{\ge0}$. Then the main constructions and results of \cite{BMS2} extend to $M$-graded rings.

We will say that a $p$-complete $M$-graded ring $R_\star$ is \emph{quasisyntomic} (resp.~\emph{quasiregular semiperfectoid}) if the underlying ring $R=\widehat{\bigoplus}_{m\in M}R_m$ is quasisyntomic (resp.~quasiregular semiperfectoid). One has a natural graded analog of the quasisyntomic site, and similarly quasiregular semiperfectoids form a basis (for example, by extracting $p$-power roots of homogeneous elements as in \cite[Lem.~4.27]{BMS2}; this is why $p$-divisibility of $M$ is required); one obtains an analog of unfolding in
this context. 

For such $R_\star$, we have seen in \Cref{rem_THHofgraded} that we have
natural $M$-graded spectra,
$\THH(R_{\star}; \mathbb{Z}_p)$,  $\TC^-(R_{\star};
\mathbb{Z}_p)$, $\THH(R_{\star};
\mathbb{Z}_p)_{hS^1}$, and $\TP(R_{\star};
\mathbb{Z}_p)$. Moreover, the latter are naturally filtered objects in $M$-graded $p$-complete spectra, by carrying over the construction of the motivic filtration of \cite{BMS2} to the graded context. 
It follows that we get $M$-graded $p$-complete objects
$ \Prism_{R_{\star}}\left\{i\right\}, \mathcal{N}^{\geq i}
\Prism_{R_{\star}}\left\{i\right\}$, etc. 
In general, 
the underlying ($p$-complete) objects of  $
 \Prism_{R_{\star}}\left\{i\right\}, \mathcal{N}^{\geq i}
\Prism_{R_{\star}}\left\{i\right\}$ 
do \emph{not} agree with those of $\Prism_R\left\{i\right\}, \mathcal{N}^{\geq i}
\Prism_R\left\{i\right\}$, because the underlying object of $\TC^-(R_{\star};
\mathbb{Z}_p)$ is not $\TC^-(R; \mathbb{Z}_p)$. 
However, 
for each $j \geq i$, the underlying $p$-complete object of 
$\mathcal{N}^{\geq i} \Prism_{R_{\star}}\left\{i\right\} / \mathcal{N}^{\geq j}
\Prism_{R_{\star}}
\left\{i\right\}$ is $\mathcal{N}^{\geq i} \Prism_{R}\left\{i\right\} / \mathcal{N}^{\geq j}
\Prism_{R}
\left\{i\right\}$. 
This follows because the underlying object of $\THH(R_{\star}; \mathbb{Z}_p)$
is $\THH(R; \mathbb{Z}_p)$ and the forgetful functor from $M$-graded spectra to
spectra commutes with \emph{finite} homotopy limits. 

Throughout \cite{BMS2}, a basic tool is the cotangent complex and its 
wedge powers; here we implicitly use that if $R_{\star}$ is a $p$-complete $M$-graded ring,
then we have natural $p$-complete graded objects $\bigwedge^i
L_{R_{\star}/\mathbb{Z}_p}$ (defined in the usual manner as a left derived
functor of differential forms). The Hochschild--Kostant--Rosenberg theorem
remains valid in the graded context, and from there the results of
section~\ref{sec:relativeTHH} carry over to this context as well. 
\end{construction} 

We now turn to the interaction of the Frobenius with the grading, in the case which interests us. 
\begin{proposition}[Frobenius multiplies grading by $p$] 
\label{autograding}
Let $R_{\star}$ be a quasisyntomic $\mathbb{Z}[1/p]_{\geq 0}$-
(resp.~$\mathbb{Z}_{\geq 0}$-) graded ring (with underlying quasisyntomic ring $R$). 
Then we claim that
\begin{enumerate}
\item  $\Prism_R\left\{i\right\}/ \mathcal{N}^{\geq n} \Prism_R\{i\}$ naturally
upgrades to have the structure of a $\mathbb{Z}[1/p]_{\geq 0}$- (resp.~$\mathbb{Z}_{\geq 0}$-) graded object in the $p$-complete derived $\infty$-category $\widehat{D(\mathbb{Z}_p)}$
\item
the 
Frobenius map modulo $p$, as in (\ref{nilpFpngeneral})
$$\varphi_i\colon \mathcal{N}^{\geq i}\Prism_R\left\{i\right\}/
\mathcal{N}^{i+r} \Prism_R\left\{i\right\}  \otimes_{\mathbb{Z}}^{{L}} \mathbb{F}_p \to
\Prism_R\left\{i\right\}/\mathcal{N}^{\geq i+r} \Prism_R\left\{i\right\}
\otimes_{\mathbb{Z}}^{{L}} \mathbb{F}_p$$ 
 multiplies degrees by $p$. 
\end{enumerate}
\end{proposition}
\begin{proof}
In the case in which $R_{\star}$ is $\mathbb Z[1/p]_{\ge0}$-graded, part (1) is covered by the general \Cref{consBMS2graded}: arguing locally on the graded version of the quasisyntomic site, we require a natural graded version of \Cref{def_prism_sheaves} for any graded quasiregular semiperfectoid ring. But this follows from $\THH$ of a graded ring being a graded spectrum with $S^1$-action, as explained in \Cref{rem_THHofgraded}.
 
When $R$ is actually $\mathbb Z_{\ge 0}$-graded, then we claim that the same is
true of $\Prism_R\left\{i\right\}/ \mathcal{N}^{\geq n} \Prism_R\{i\}$. By
d\'evissage, it suffices to prove this for each $\mathcal{N}^t \Prism_R$, which
in turn follows from the filtrations of \Cref{filtongrTHH} and
\Cref{filtconnTHHbounds}. Here we use that the $p$-completed cotangent complex
$\widehat{L_{R/\mathbb{Z}_p}}$ and its wedge powers are naturally $\mathbb
Z_{\ge 0}$-graded, and the filtrations of the aforementioned corollaries respect
these gradings. 

For part (2) we apply a left Kan extension argument to assume that $R$ is $p$-torsion-free, and then argue locally as above to reduce to the case that $R$ is a $p$-torsion-free, graded, quasiregular semiperfectoid. Then both sides of $\phi_i$ are discrete, and so the problem reduces to verifying the {\em property} that it multiplies degrees by $p$. 
But this follows from the treatment of graded cyclotomic spectra
in \Cref{app:gradedcyc}. 
\end{proof} 

\begin{proof}[Proof of \Cref{prop_reduction}] 
Let $R=A \oplus N$ and $I=N$ be a henselian pair of the form of
\Cref{prop_reduction}. We view $R$ as being $\mathbb Z_{\ge0}$-graded, with $A$
in degree $0$ and $N$ in degree one. We must prove that $\mathrm{fib}(\mathbb F_p(i)(R)\to\mathbb F_p(i)(A))\in D^{\le i}(\mathbb F_p)$. Choose $r\gg0$ so that expression (\ref{nilpFpngeneral}) is valid.

The degree $0$ part of $ \widehat{\bigwedge^i L_{R/\mathbb{Z}_p}}$ identifies with $\widehat{\bigwedge^i L_{A/\mathbb{Z}_p}}$, so by d\'evissage using \Cref{filtongrTHH} and \Cref{filtconnTHHbounds} we see that the same is true of the $\mathbb{Z}_{\geq 0}$-graded object $\Prism_{R}\left\{i\right\}/\mathcal{N}^{\geq i+r}
\Prism_{R}\left\{i\right\}$. Namely, its degree $0$ part is $\Prism_{A}\left\{i\right\}/\mathcal{N}^{\geq i+r}
\Prism_{A}\left\{i\right\}$. The same is true modulo $p$, and so $\mathrm{fib}(\mathbb F_p(i)(R)\to\mathbb F_p(i)(A))$ identifies with the fiber of \[
\can - \phi_i\colon (\mathcal{N}^{\geq i} \Prism_R\left\{i\right\}/
\mathcal{N}^{\geq i+r} \Prism_R \left\{i\right\})_{>0} \otimes_{\mathbb{Z}}^L
\mathbb{F}_p \to 
( \Prism_R\left\{i\right\}/
\mathcal{N}^{\geq i+r} \Prism_R \left\{i\right\})_{>0} \otimes_{\mathbb{Z}}^L
\mathbb{F}_p,
\]
where the subscript $>0$ denotes the $\mathbb Z_{>0}$-subobject of a $\mathbb Z_{\ge0}$-graded object.

To complete the proof we must verify the conditions of \Cref{grconnlemma2} below. Firstly, the Frobenius multiplies degrees by $p$ by \Cref{autograding}(2). Next, the fiber of can is $(\Prism_R\left\{i\right\}/
\mathcal{N}^{\geq i} \Prism_R \left\{i\right\})_{>0}[-1]$, which lies in $D^{\le i}(\mathbb Z_p)$ by \Cref{cofiberbound1}. Finally, to verify condition (2), observe that each cohomology group of $\widehat{\bigwedge^i
L_{R/\mathbb{Z}_p}}$ is a $\mathbb Z_{\ge0}$-graded, finitely generated $R$-module, so necessarily zero except in finitely many degrees of the grading. The same then holds for $\Prism_{R}\left\{i\right\}/\mathcal{N}^{\geq i+r}
\Prism_{R}\left\{i\right\}$ by another d\'evissage through \Cref{filtongrTHH} and \Cref{filtconnTHHbounds}.
\end{proof}

\begin{lemma} 
\label{grconnlemma2}
Let $M, N$ be $\mathbb{Z}_{>0}$-graded objects of $D(
\mathbb{F}_p)$, and $i\ge0$.
Let $\mathrm{can}\colon M \to N$ be a map of graded objects and let
$\varphi\colon M \to N$ be a map which multiplies degrees by $p$. 
Suppose that
\begin{enumerate}
    \item[{\rm (1)}]  the fiber of $\mathrm{can}$ belongs to $D(\mathbb{F}_p)^{\leq i}$,
\item[{\rm (2)}] 
for any fixed $n$, the cohomologies $H^n(M), H^n(N)$ vanish except in finitely many degrees of the grading.
\end{enumerate}
Then $\mathrm{fib}( \can - \phi\colon M \to N)$ belongs to
    $D(\mathbb{F}_p)^{\leq i}$.
\end{lemma} 
\begin{proof} 
By (2), we can replace the direct sums $\bigoplus M_i$ and $\bigoplus N_i$ with the
    corresponding infinite
products. Therefore, the result follows from \Cref{connectivitynilpotencelemma}. 
\end{proof}

Finally, we use the rigidity theorem to give a description of the top cohomology
of the $\mathbb{F}_p(i)$. 
For $B$ an $\mathbb{F}_p$-algebra, let $\Omega^i_{B}$ be the (underived) module of $i$-forms on $B$ (relative to
$\mathbb{Z}_p$, or $\mathbb{F}_p$), and let \[ C^{-1}\colon \Omega^{i}_{B} \to \Omega^i_{B}/d \Omega^{i-1}_{B}\] be the inverse Cartier operator.

\begin{corollary}[Top cohomology of $\mathbb{F}_p(i)$] 
\label{topcohexpr}
Let $R$ be a $p$-complete ring. 
Then there is a natural isomorphism
 \begin{equation} H^{i+1}( \mathbb{F}_p(i)(R)) \simeq
\mathrm{coker}(1 -C^{-1}\colon \Omega^{i}_{R/p} \to
\Omega^i_{R/p}/d\Omega^{i-1}_{R/p}). 
\label{topcohFpirigid}
\end{equation}
In particular, if $R$ is $w$-strictly local, then $H^{i+1}( \mathbb{F}_p(i)(R)) =
H^{i+1}( \mathbb{Z}_p(i)(R)) =0 $. 
\end{corollary} 
\begin{proof} 
Without loss of generality, we can assume that $R$ is an $\mathbb{F}_p$-algebra 
via \Cref{rigidity}. 
We now use a standard argument (cf., e.g.,~\cite{Hoobler06}) to reduce to the case where $R$ is ind-smooth
over $\mathbb{F}_p$. 
Choose a polynomial $\mathbb F_p$-algebra $P$ surjecting onto $R$ and henselize
$P$ along the kernel of $P \twoheadrightarrow R$, forming a $P$-algebra $P'$
augmented over $R$; another use of \Cref{rigidity} 
gives $H^{i+1}(\mathbb{F}_p(i)(R)) = H^{i+1}( \mathbb{F}_p(i)(P'))$. Moreover, 
\cite[Prop.~4.31]{CMM} gives that 
the right-hand-side of 
\Cref{topcohFpirigid} is also invariant under passage from $P'$ to $R$. 
Replacing $R$ by $P'$, 
we thus reduce the case that $R$ is ind-smooth over $\mathbb F_p$. But then
\begin{equation} 
\label{FpiR}
\mathbb{F}_p(i)(R) \simeq \mathrm{fib}( \Omega^i_R \xrightarrow{1 - C^{-1}}
\Omega^i_R/d \Omega^{i-1}_R ) [-i].
\end{equation} 
This follows because of the expression of 
$\mathbb{F}_p(i)$ as the shifted \'etale cohomology of 
logarithmic forms
$\Omega^i_{\mathrm{log}}$
(cf.~\cite[Cor.~8.21]{BMS2} and reduction modulo $p$; we also review this in the
next section) and the short exact
sequence of \'etale sheaves (cf.~\cite[Sec.~2.4]{Ill79} and \cite[Cor.~4.2]{MorrowHW})
\[ 0 \to \Omega^i_{\mathrm{log}} \to \Omega^i \xrightarrow{1 - C^{-1}}
\Omega^i/d \Omega^{i-1} \to 0 .  \]
Expression \eqref{FpiR} implies the claim. 
Note that this short exact sequence (in the \'etale topology) implies also that
for any ind-smooth $\mathbb{F}_p$-algebra which has no non-split \'etale covers
(e.g., $R$ could be $w$-strictly local), $H^{i+1}( \mathbb{F}_p(i)(R)) = 0$. 
\end{proof}

\Cref{Zpibound},  \Cref{rigidity}, and \Cref{topcohexpr} imply Theorem G from
the introduction. 

\section{The comparison with syntomic cohomology}

\newcommand{\ri}{\mathrm{Ring}_{\mathrm{ideal}}}
\newcommand{\fib}{\mathrm{fib}}
\newcommand{\zps}[1]{\mathbb{Z}_p(#1)^{\mathrm{FM}}}
\newcommand{\qps}[1]{\mathbb{Q}_p(#1)^{\mathrm{FM}}}

\newcommand{\SCR}{\mathrm{SCR}}

In this section, we show that the $\mathbb{Z}_p(i)$ for $i \leq p-2$ and the
$\mathbb{Q}_p(i)$ for all $i$ can be described purely in terms of derived de
Rham (instead of prismatic) cohomology, using a form of syntomic cohomology
\cite{FM87, Ka87}. 
The strategy is to use the description of the $\mathbb{Z}_p(i)$ in equal
characteristic $p$ from \cite{BMS2} together with the Beilinson fiber square
to relate the $\mathbb{Z}_p(i)$ in mixed and equal characteristic. 
In particular, we prove Theorem~\hyperref[thm:f]{F}.

\subsection{Syntomic cohomology}

To begin with, we define another form of syntomic cohomology 
via the quasisyntomic site, by descent from
quasiregular semiperfectoids. 

\begin{definition}[$p$-adic derived de Rham cohomology] 
For a map of  rings $ A \to R$, we let $L \Omega_{R/A} \in D(A)$ denote the
\emph{$p$-adic derived de Rham cohomology} of $R$ relative to $A$, see
\cite{Bhattpadic}. 
By definition, when $R$ is a finitely generated polynomial $A$-algebra, $L
\Omega_{R/A}$ is given by the $p$-completed relative de Rham complex
$\Omega_{R/A}^{\bullet}$, and in general $L \Omega_{R/A} $ is defined via
$p$-complete left Kan extension. 
One can show \cite[Cor.~3.10]{Bhattpadic} that $L \Omega_{R/A}$ more generally  agrees with
the $p$-completed (underived) relative de Rham complex when $R$ is smooth
over $A$. When $A = \mathbb{Z}$, we omit $A$ from the notation. 

The $p$-adic derived de Rham cohomology $L \Omega_{R/A}$ is equipped with the derived Hodge filtration
$L \Omega_{R/A}^{\geq \star}$, obtained by left Kan extending the
naive filtration in the polynomial (or more generally smooth) case. 
\label{padicddR}
\end{definition} 

\begin{example}[Derived de Rham cohomology and divided powers] 
\label{dividedpowers}
We recall the following basic calculation: for the map $\mathbb{Z}[x] \to
\mathbb{Z}$, we have that 
$L \Omega_{\mathbb{Z}/\mathbb{Z}[x]} \simeq \widehat{ \Gamma(x)}$ is the
$p$-complete divided power algebra on the class $x$, and the derived Hodge filtration
is the divided power filtration. 
This is a special case of \cite[Cor.~3.40]{Bhattpadic}. 
See also \cite[Prop.~3.16]{SzZa18} for an account. 
\end{example} 

\begin{definition}[Derived de Rham--Witt cohomology] 
For an $\mathbb{F}_p$-algebra $S$, we let $LW\Omega_S \in D(
\mathbb{Z}_p)$ denote the $p$-adic derived de Rham--Witt cohomology or
derived crystalline cohomology of $S$ (defined via $p$-complete left Kan extension from
finitely generated polynomial $\mathbb{F}_p$-algebras) \cite[Sec.~8]{BMS2}; for ind-smooth $\mathbb{F}_p$-algebras $S$ this agrees with Illusie's usual  $W\Omega_S$. 
It comes equipped with the derived Nygaard filtration $\mathcal{N}^{\geq\star}LW\Omega_S$ obtained via left Kan extension of the usual Nygaard
filtration in the finitely generated polynomial case (see \cite[Sec.~8]{BLM} for an account);
we write $\widehat{LW\Omega}_S$ for the completion of $LW\Omega_S$ with respect to the Nygaard
filtration.
We have by~\cite[Lemma~8.2]{BMS2} an identification of the associated graded terms 
of the Nygaard filtration 
$\mathcal{N}^{\geq i}LW\Omega_S/\mathcal{N}^{\geq i+ 1}LW \Omega_S  \simeq
L \left(\tau^{\le i} \Omega_{S/\mathbb{F}_p}\right)$ with the derived functor of $S
\mapsto \tau^{\le i} \Omega_{S/\mathbb{F}_p}$. 
The Frobenius $\varphi\colon S \to S$ induces an endomorphism of $LW \Omega_S$; on
$\mathcal{N}^{\geq i} L W \Omega_S$ it becomes divisible by $p^i$, and indeed we
have divided Frobenius maps \begin{equation}\label{dividedFrob} \varphi_i: \mathcal{N}^{\geq i} L W \Omega_S \to L W
\Omega_S. \end{equation}
Finally, using the de Rham-to-crystalline comparison, we find that if $R$ is a
$p$-torsion free ring, then there is a natural equivalence
\begin{equation} 
L \Omega_R \simeq L W \Omega_{R/p};
\label{eqn:de_rham_crys}
\end{equation} 
in particular, $L \Omega_R$ naturally carries a Frobenius operator $\varphi$. 
\end{definition}

\begin{remark}[Sheaf properties] 
\label{drwisqsynsheaf}
The functors $S \mapsto L W \Omega_S$ and $S\mapsto\widehat{LW \Omega_S}$ are  
sheaves on $\mathrm{qSyn}_{\mathbb{F}_p}$. 
For $LW \Omega_S$, it suffices to work modulo $p$, and then use the conjugate
filtration on derived de Rham cohomology \cite{Bhattpadic} and the flat descent
for the wedge powers of the cotangent complex \cite[Sec.~3]{BMS2}. For the
Nygaard-completed  
$\widehat{LW \Omega_S}$, this follows since the associated graded terms have
this property, by a similar argument. 
Similarly, the filtration pieces $S\mapsto \mathcal{N}^{\geq i} L W \Omega_S$
    and $S\mapsto \mathcal{N}^{\geq i} \widehat{LW \Omega_S}$ are sheaves on
$\mathrm{qSyn}_{\mathbb{F}_p}$. 
\end{remark}

\begin{construction}[{Derived de Rham--Witt cohomology of quasiregular semiperfect
rings, cf.~\cite[Sec.~8.2]{BMS2}}] 
\label{drwofqrp}
Let $S \in \qlrsp_{\mathbb{F}_p}$ be a quasiregular semiperfect
$\mathbb{F}_p$-algebra. In this case, one forms the ring $\acrys(S)$ (defined by Fontaine \cite{Fo94}),
which is the universal $p$-adically complete divided power thickening of $S$,
with divided powers compatible with those on $(p) \subset \mathbb{Z}_p$;
quasiregularity ensures that it is $p$-torsion free. 
Then, one has a natural identification $$LW \Omega_S = \acrys(S),$$ and the
Nygaard filtration becomes the filtration 
$$\mathcal{N}^{\geq i} \acrys(S) = \left\{x \in \acrys(S)\colon \varphi(x)\in p^i\acrys(S)
\right\}.$$
Here $\varphi\colon \acrys(S) \to \acrys(S)$ denotes the endomorphism induced by the
Frobenius on $S$; it has the further property that $\varphi(x) \equiv x^p \pmod{p}$ for $x \in \acrys(S)$, i.e., $\varphi$ defines the structure of a
$\delta$-ring on the $p$-torsion free ring $\acrys(S)$. 
\end{construction}

\begin{construction}[Derived de Rham and de Rham--Witt cohomology of qrsp rings] 
% Let $R \in \qlrsp_{\mathbb{Z}_p}$. 
% We consider the 
% derived de Rham--Witt cohomology $LW \Omega_{R/p}$ of $R/p$. 
% 
% Since the ring $R/p$ is a quasiregular semiperfect $\mathbb{F}_p$-algebra, 
% it follows from Construction \ref{drwofqrp} that 
% there is an isomorphism
% \( LW \Omega_{R/p} \simeq \acrys(R/p),  \)
% and that this is a $p$-complete, $p$-torsionfree discrete ring with an
% endomorphism arising from the Frobenius of $R/p$. 
% Moreover, via the de Rham-to-crystalline comparison, we have also an isomorphism
% $L \Omega_R \simeq LW\Omega_{R/p} ,$ 
% and this ring therefore acquires a multiplicative, descending Hodge
% filtration. 

% \marginpar{MM: I found that this construction was clearer previously, since it emphasised that we would work with $L \Omega_R$. Why not just follow the referee's suggestion of arguing via part (1) and equation (\ref{eqn:de_rham_crys})?}

Let $R \in \qlrsp_{\mathbb{Z}_p}$. 
We consider the following rings.
\begin{enumerate}
\item 
The derived de Rham--Witt cohomology $LW \Omega_{R/p}$ of $R/p$. 
Since the ring $R/p$ is a quasiregular semiperfect $\mathbb{F}_p$-algebra, 
it follows from Construction \ref{drwofqrp} that 
there is an isomorphism
\[ LW \Omega_{R/p} \simeq \acrys(R/p).  \]
\item The ($p$-adic) derived de Rham cohomology $L \Omega_R$. 
Here $L \Omega_R$ is a discrete, $p$-complete and $p$-torsion free ring, as follows from the previous point and (\ref{eqn:de_rham_crys}). The ring $L \Omega_R$ is also equipped with the multiplicative, descending
Hodge filtration $L \Omega_R^{\geq \star}$. 
\end{enumerate}
Via the de Rham-to-crystalline comparison, we have equivalences
$L \Omega_R \simeq LW\Omega_{R/p} \simeq \acrys(R/p).$ 
In particular, we find via (1) and (2) above that the $p$-complete,
$p$-torsion free ring $L \Omega_R$ is equipped with both a Frobenius
operator and a Hodge filtration. 

\end{construction}

\begin{lemma} 
\label{padicvallemma}
Let $r \geq 0$ be an integer. 
\begin{enumerate}
\item[{\rm (1)}] The $p$-adic valuation of $\frac{(pr)!}{r!}$ is equal to $r$.  
\item[{\rm (2)}] The $p$-adic valuation of $\frac{p^r}{r!}$ is at least $\min(r, p-1)$. 
\end{enumerate}
\end{lemma} 
\begin{proof} 
Both assertions follow from Legendre's formula, $v_p(n!) = \sum_{j > 0} \lfloor n/p^j
\rfloor$ for $v_p$ the $p$-adic valuation. 
\end{proof} 
\begin{proposition}[{Divisibility of Frobenius, cf.~also
\cite[Lem.~A1.4]{Ts99}}] 
Let $R \in \qlrsp_{\mathbb{Z}_p}$. 
Then for $i \leq p-1$, the Frobenius $\varphi\colon L \Omega_R \to L \Omega_R$ carries $L
\Omega_R^{\geq i}$ into $p^i L \Omega_R$, or in other words the de
Rham-to-crystalline comparison carries $L\Omega_R^{\ge i}$ into $\mathcal N^{\ge
i}\acrys(R/p)$.
\end{proposition} 
\begin{proof} 
For any $R \in \qlrsp_{\mathbb{Z}_p}$, we can write $R = W(A)/I$, where $A$ is a perfect
$\mathbb{F}_p$-algebra
and $I \subset W(A)$ is an ideal. 
We have an identification of the $p$-complete cotangent complex,
$\widehat{L_{R/\mathbb{Z}_p}} \simeq \widehat{I/I^2}[1]$.

We first verify the 
assertion when the ideal $I$ as above can be written as $I = (f)$, for $f$ a
nonzerodivisor, so $R = W(A)/(f)$. 
In this case, 
in view of 
\Cref{dividedpowers} and base change, 
we find that 
\( L\Omega_R = L \Omega_{R/W(A)}   \) is the $p$-completion of 
the divided power envelope of the regular ideal $(f)$, i.e., the ring
$W(A)[ f^n/n!]_{n \geq 1}$; furthermore, for each $i$, the Hodge filtered piece $L
\Omega_R^{\geq i}$ identifies with the corresponding divided power filtration,
i.e., the ideal $( f^j/j!)_{j
\geq i}$. 
Now the Frobenius $\varphi$ on $L\Omega_R \simeq LW \Omega_{R/p} \simeq
\acrys(R/p)$ is a Frobenius lift coming from a $\delta$-structure, so 
\begin{equation} \varphi\left( \frac{f^j}{j!} \right) =
\frac{(f^p + p \delta(f))^j}{j!} = \frac{\sum_{0 \leq l \leq j}
\binom{j}{l} f^{pl} p^{j-l} \delta(f)^{j-l}
}{j!}.
\end{equation}
The $l$th term in the sum above is divisible (in the ring $L \Omega_R$) by 
$\frac{ f^{pl}
p^{j-l}}{l! (j-l)!} =  \frac{f^{pl}}{(pl)!}  \frac{(pl)!
p^{j-l}}{l! (j-l)!}$, where we use the divided powers on $(f)$ to see
$\frac{f^{pl}}{(pl)!} \in L \Omega_R$.
Now  the $p$-adic valuation of 
$ \frac{(pl)!
p^{j-l}}{l! (j-l)!}$
is at least $l + \min(j-l, p-1)$ thanks to \Cref{padicvallemma}. 
So if $i \leq p-1$, 
then it follows that $\varphi$ carries $L \Omega^{\geq i}_R$ into 
$p^i L \Omega_R$. 

Now suppose $R$ is a $p$-complete tensor product over $\alpha \in \mathcal{A}$ of 
rings of the form $W(A_\alpha)/(f_\alpha)$, for $A_\alpha$ perfect
$\mathbb{F}_p$-algebras and $f_\alpha \in W(A_\alpha)$ regular elements. 
In this case, we have an isomorphism (after $p$-completion) of filtered rings
$L \Omega_R^{\geq\star}\simeq \bigotimes_{\alpha \in
\mathcal{A}}L
\Omega_{W(A_\alpha)/f_\alpha}^{\geq\star}$ by the K\"unneth formula, which is compatible with the
Frobenius operators. The assertion 
$\varphi( L \Omega_R^{\geq i}) \subset p^i L \Omega_R$ for $i \leq p-1$
for such $R$ thus follows by taking tensor products. 

Finally, let $R \in \qlrsp_{\mathbb{Z}_p}$ be arbitrary and write $R =
W(A)/I$ for $A$ a perfect $\mathbb{F}_p$-algebra. 
To prove the claim $\varphi(L\Omega_R^{\geq i}) \subset p^i L \Omega_R$ for $i
\leq p-1$, we will reduce to the previous cases, following the strategy of
\cite[Theorem 8.14]{BMS2}. 
Let $\left\{x_t\right\}_{t \in T}$ be a system of generators for the ideal $I$
and for each $t$, we write
$x_t = \sum_{i \geq 0} p^i [y_{t, i}]$ for some $y_{t, i} \in A$. 
For each $t \in T$, we have a map 
\begin{equation} \label{Wpolymap} W(\mathbb{F}_p[u_1, u_2, \dots, ]_{\mathrm{perf}})/( [u_1] + p[u_2] +
\dots ) \to W(A)/I =  R  
\end{equation}
sending $[u_i] \mapsto [y_{t, i}]$; note that the source belongs to
$\qlrsp_{\mathbb{Z}_p}$, and its cotangent complex is the shift of a free of
rank $1$ module. 
The map \eqref{Wpolymap} has image
on $p$-completed cotangent complexes given by the class of $x_t$. 

We consider the $p$-completed tensor product $$R' \stackrel{\mathrm{def}}{=}W(A) \hat{\otimes} \bigotimes_{t
\in T} ( W(\mathbb{F}_p[u_1, u_2, \dots, ]_{\mathrm{perf}})/ 
[u_1] + p [u_2] + \dots 
),$$ which maps surjectively to $R$ (via the above maps) and induces a
surjection on $H^0(\widehat{L_{-/\mathbb{Z}_p}}[-1])$.  
Comparing with the reduction mod $p$ and using the Hodge and conjugate
filtrations on
derived de Rham cohomology, we find that 
$L\Omega_{R'}^{\geq i} \to L \Omega_R^{\geq i}$ is a surjection for each
$i$. 
Since the previous discussion shows that $\phi( L \Omega_{R'}^{\geq i}) \subset
p^i L \Omega_{R'}$ for $i \leq p-1$, we can now conclude the claim for $R$ by
naturality, as desired. 
\end{proof} 

Using the divisibility property of Frobenius, 
we can define, for $R \in \qlrsp_{\mathbb{Z}_p}$ and ${i \leq p-1}$, 
a \emph{divided Frobenius} $\varphi/p^i \colon L \Omega_R^{\geq i} \to L \Omega_R$
(of discrete, $p$-torsion free abelian groups). 
Using the divided Frobenius, we now define syntomic cohomology; this definition 
is based on the ideas of \cite{FM87, Ka87} (and can be compared with it using
the comparison between derived de Rham and crystalline cohomology in the lci
case, cf.~\cite[Sec.~3]{Bhattpadic}). 
\begin{definition}[Syntomic cohomology] 
We define sheaves 
$\zps{i}$ for $0 \leq i \leq p-2$, and $\qps{i}$ for $i\ge0$, 
on $\qlrsp_{\mathbb{Z}_p}$ via 
\begin{gather} \zps{i}(R) = \mathrm{fib}( \phi/p^i - 1\colon L \Omega_R^{\geq i} \to
    L \Omega_R),\label{eq:zpfm} \\
\qps{i}(R) = \mathrm{fib}( \phi - p^i \colon L \Omega_R^{\geq i} \to L \Omega_R)_{\QQ_p}. \end{gather}
These are sheaves on 
$\qlrsp_{\mathbb{Z}_p}$ because $R \mapsto L \Omega_R^{\geq i}$
is a sheaf. 
Unfolding, we obtain sheaves $\zps{i}$ for $0 \leq i \leq p-2$ and $\qps{i} $ for all
$i\ge0$ on $\mathrm{qSyn}_{\mathbb{Z}_p}$. 
\end{definition} 

\begin{remark}
    While one could define $\ZZ_p(p-1)^{\mathrm{FM}}(R)$ via the same formula, this does not
    give the correct integral theory in weight $(p-1)$.
	 \end{remark}

\subsection{The $\mathbb{Z}_p(i)$ in  equal characteristic $p$}

In equal characteristic $p$, the $\mathbb{Z}_p(i)$ can be determined via
the theory of the de Rham--Witt complex and its derived versions, cf.
\cite[Sec.~VIII.2]{Ill72}, \cite{Bhattpadic}, and in particular
\cite[Sec.~8]{BMS2}.\footnote{See also \cite{GL00, GH99} for the identification with  
with $p$-adic \'etale motivic cohomology.} 
We next review this. 

\begin{theorem}[{$\mathbb{Z}_p(i)$ in equal characteristic $p$,
cf.~\cite[Sec.~8]{BMS2}}] 
\label{charpdecomplete}
Suppose $S$ is a quasisyntomic $\mathbb{F}_p$-algebra. 
Then, for each $i$,
\begin{enumerate}
\item[{\rm (1)}] 
${\Prism_S}\left\{i\right\}$ is the Nygaard-completed derived de Rham--Witt cohomology
$\widehat{LW\Omega}_S$
of $S$ and
\item[{\rm (2)}]
the Nygaard filtration $\mathcal{N}^{\geq i} {\Prism}_S\left\{i\right\} $
identifies with the de Rham--Witt Nygaard filtration 
$\mathcal{N}^{\geq i} \widehat{LW\Omega}_S$, and the prismatic Frobenius
$\varphi_i$ identifies with the divided Frobenius \eqref{dividedFrob}. 
\end{enumerate}
Consequently,
 \begin{equation} \label{Zpincharp}\mathbb{Z}_p(i)(S) = \mathrm{fib}( \varphi_i - \mathrm{can}\colon
\mathcal{N}^{\geq i} \widehat{LW\Omega}_S \to \widehat{LW\Omega}_S),
\end{equation}
where $\varphi_i: \mathcal{N}^{\geq i} 
\widehat{LW\Omega}_S \to \widehat{LW\Omega}_S$ is the divided Frobenius operator, so that
$p^i\varphi_i$ is the Frobenius. 
\end{theorem}

\begin{remark} 
\label{dontneednygaardcompletion}
The Nygaard completion is 
redundant in the formula \eqref{Zpincharp} for  $\mathbb{Z}_p(i)(S)$. 
This follows easily from the fact that 
$\varphi_i$ acts by zero  on $\mathcal{N}^{\geq i+1} \widehat{LW\Omega}_S/p $. 
In particular, we can write
\[ \mathbb{F}_p(i)(S) = \mathrm{fib}( \varphi_i - \mathrm{can}\colon  
(\mathcal{N}^{\geq i}LW\Omega_S / 
\mathcal{N}^{\geq i+1}LW\Omega_S) \otimes_{\mathbb{Z}_p}^L \mathbb{F}_p  \to (LW \Omega_S/\mathcal{N}^{\geq
i+1}LW\Omega_S) \otimes_{\mathbb{Z}_p}^L \mathbb{F}_p).
\]
\end{remark}
\begin{example}\label{Zpiforqrsperfs} 
If $S $ is
a quasiregular semiperfect $\mathbb{F}_p$-algebra, then  for $i  > 0$,
$\mathbb{Z}_p(i)(S)$ is discrete, $p$-torsion free, and 
there is a natural  isomorphism of abelian groups
\begin{equation} \label{QpofFpalgebra} \mathbb{Z}_p(i)(S) \simeq
\mathrm{ker}( \phi  - p^i\colon \acrys(S) \to \acrys(S)),   \end{equation} 
where $\varphi$ is induced by the Frobenius. For $i = 0$, we should instead take the homotopy fiber of $\varphi - 1$ on
$\acrys(S)$, so it may have terms in cohomological degree $1$. 
\end{example} 

In the ind-smooth case, one has an 
identification with logarithmic de Rham--Witt forms. 

\begin{definition}[{Logarithmic de Rham--Witt forms}] 
For $S$ an ind-smooth $\mathbb{F}_p$-algebra, we  let $W \Omega_{S, \mathrm{log}}^\bullet$  denote the graded subring 
of the de Rham-Witt complex $W \Omega_S^\bullet$ consisting of fixed points for $F$. When $S$ is local, one knows that $W
\Omega_{S, \mathrm{log}}^\bullet$ 
is generated, modulo any power of $p$, in degree $1$ by elements of the form $d[x]/[x]$, for $x \in
S^{\times}$ and $[x] \in W(S)$ the Teichm\"uller representative, cf.~\cite[Th. 5.7.2]{Ill79} which proves this \'etale locally and \cite[Theorem
0.10]{MorrowHW} for a very general Zariski local result. 
Note that for each $i$, $W \Omega_{-, \mathrm{log}}^i$ defines a pro-\'etale
sheaf on $\spec(S)$.
\end{definition}

\begin{theorem}[{Cf.~\cite[Cor.~8.21]{BMS2}}] 
\label{Zpindsmooth}
Let $S$ be an ind-smooth $\mathbb{F}_p$-algebra. 
Then there are natural identifications
\[ \mathbb{Z}_p(i)(S) \simeq R \Gamma_{\mathrm{proet}} ( \spec(S), W \Omega_{-,
\mathrm{log}}^i)[-i]. \]
\end{theorem}

\subsection{The Beilinson fiber square on graded pieces}

Our goal is to relate the $\mathbb{Z}_p(i)$ in mixed and in equal
characteristic. 
We use the Beilinson fiber sequence to prove a basic fiber square which gives a version of Theorem~\hyperref[thm:a]{A} on associated
graded terms for the motivic filtrations.
\begin{construction}[The trace on graded pieces] 
Let $R$ be any commutative ring. Then we have the trace maps 
$$\K(R ; \mathbb{Z}_p) \to \mathrm{TC}(R; \mathbb{Z}_p) \to \HC^-(R;
\mathbb{Z}_p).$$
 When $R \in \qlrsp_{\mathbb{Z}_p}$, 
then $\HC^-(R; \mathbb{Z}_p)$ is concentrated in even degrees and 
$\pi_{2i}$ is given by $\widehat{L \Omega}_R^{\geq i}$, cf.~\cite[Sec.~5]{BMS2}
and \cite{antieauperiodic}. 
Unfolding we conclude that, on graded pieces, we obtain a natural map
\(  \mathbb{Z}_p(i)(R) \to  
\widehat{L \Omega}_R^{\geq i}
\)
for $R\in\mathrm{qSyn}_{\mathbb{Z}_p}$. This naturally factors through $L \Omega_R^{\geq i}$ since $R \mapsto
\mathbb{Z}_p(i)(R)$ is left Kan extended from $p$-complete polynomial algebras
(\Cref{Zpibound}). 
\end{construction}

\begin{theorem}[The Beilinson fiber square on graded terms] 
\label{Beilinsonongr}
Let $R \in \mathrm{qSyn}_{\mathbb{Z}_p}$.  Then, for each $i \geq 0$, there exists a natural 
map $\chi_i\colon \mathbb{Q}_p(i)(R/p) \to ({L\Omega}_R)_{\mathbb{Q}_p}$
and a functorial pullback square
\begin{equation}\label{Bfibgr}\begin{gathered} 
\xymatrix{
\mathbb{Q}_p(i)(R)  \ar[d]  \ar[r] &  \mathbb{Q}_p(i)(R/p) \ar[d]^{\chi_i}  \\
({L\Omega}_R^{\geq i})_{\mathbb{Q}_p} \ar[r] & ({L
\Omega}_R)_{\mathbb{Q}_p}
}  \end{gathered}
\end{equation}
in the  derived $\infty$-category $D(\mathbb{Q}_p)$.
The map $\chi_i$ arises from a natural map $\mathbb{Z}_p(i)(R/p) \to p^{-N} L
\Omega_R$ for some $N \gg 0$ (depending only on $i$), fitting into an analogous commutative diagram.

Furthermore, the associated fiber sequence holds up to isogeny: 
$\mathrm{cofib}(\mathbb{Z}_p(i)(R) \to \mathbb{Z}_p(i)(R/p))$ 
and $L \Omega_R/L \Omega_R^{\geq i}$ are naturally isogenous to each other. 
Finally, 
for $i \leq  p-2$, we have natural equivalences for
$\mathrm{qSyn}_{\mathbb{Z}_p}$,
\begin{equation}  \label{integralBfibgr}\mathrm{fib}(\mathbb{Z}_p(i)(R) \to  \mathbb{Z}_p(i)(R/p)) 
\simeq \mathrm{fib} \left(
L \Omega_R/L \Omega_R^{\geq i} \to 
L \Omega_{R/p}/L \Omega_{R/p}^{\geq i}
\right)[-1]
.\end{equation}\end{theorem} 

\begin{proof} 
Note that the hypothesis that $R \in \mathrm{qSyn}_{\mathbb{Z}_p}$ ensures that $R/p \in
\qsyn$. 
The square \eqref{Bfibgr} will be constructed in 
the $\infty$-category of $D(\mathbb{Q}_p)^{\geq 0}$-valued sheaves on
$\mathrm{qSyn}_{\mathbb{Z}_p}$. 
It suffices to construct the above pullback square for  $R \in
\qlrsp_{\mathbb{Z}_p}$, by
unfolding. 
For $R \in \qlrsp_{\mathbb{Z}_p}$, we have a pullback square by
\Cref{TCformThmA},
\begin{equation} \label{tcfibseqa}  \begin{gathered}
 \xymatrix{
\TC(R; \mathbb{Q}_p) \ar[d]  \ar[r] &  \TC(R/p; \mathbb{Q}_p) \ar[d]  \\
 \HC^-(R; \mathbb{Q}_p) \ar[r] &  \HP(R; \mathbb{Q}_p). }\end{gathered}
 \end{equation}
Note  since $R
\in \qlrsp_{\mathbb{Z}_p}$, 
the terms in the bottom row of the  above fiber square  are 
 concentrated in even degrees
\cite[Lem.~5.14]{BMS2}. 
Consequently, for each $i$, we can apply 
$\tau_{[2i-1, 2i]}$ and still obtain a fiber square. 
By definition of the $\mathbb{Q}_p(i)$ and by 
the corresponding description  of derived de Rham cohomology, as in \cite[Th.~1.17]{BMS2} (or using the filtration of \cite{antieauperiodic}), we obtain
\eqref{Bfibgr} for $R$ (after a shift), albeit with a Hodge completion. 
In particular, instead of $\chi_i$, we obtain a completed version
\[ \hat{\chi_i}\colon \mathbb{Q}_p(i)(R/p) \to (\widehat{L
\Omega}_R)_{\mathbb{Q}_p},  \]
as well as a Hodge-completed version of the fiber square \eqref{Bfibgr}. 

We can refine $\hat{\chi_i}$ to $\chi_i$ (and obtain \eqref{Bfibgr}) as follows.  
First, by construction of these maps via the Beilinson fiber square, a multiple of $\hat{\chi_i}$ lifts to a map 
$\mathbb{Z}_p(i)(R/p) \to \widehat{L \Omega}_R$. Since the source is left Kan
extended (as a functor to the $p$-complete derived $\infty$-category) from
finitely generated $p$-complete polynomial algebras, we can restrict and left Kan extend to obtain that a
multiple of $\hat{\chi_i}$ lifts to $\mathbb{Z}_p(i)(R/p) \to L \Omega_R$. 
Inverting $p$, we obtain that $\hat{\chi_i}$ factors through a map 
$\chi_i\colon \mathbb{Q}_p(i)(R/p) \to (L \Omega_R)_{\mathbb{Q}_p}$. 

From the quasi-isogeny between $\TC(R, (p); \mathbb{Z}_p)$ and 
$\Sigma \HC(R, (p); \mathbb{Z}_p)$ as in \Cref{thm_a_text}, 
 we
obtain the isogeny claim.

Finally, we verify \eqref{integralBfibgr}; again, we can assume that $R \in
\qlrsp_{\mathbb{Z}_p}$ by unfolding. 
Since everything is a pro-\'etale sheaf, we can even assume that $R$ is
$w$-strictly local in the sense of \cite{BS15}, so that $\pi_{-1} \TC(R;
\mathbb{Z}_p) = \mathrm{coker}(F - 1\colon W(R) \to W(R))$ 
(by \cite[Theorem F]{HM})
vanishes. 
Recall we have an equivalence
\( \tau_{\leq 2p-4} \mathrm{TC}(R, (p); \mathbb{Z}_p) 
\simeq \tau_{\leq 2p-4}\Sigma \mathrm{HC}(R, (p); \mathbb{Z}_p)
\) by \Cref{thm_a_text}. 
It follows that for $i \leq p-2$, we have an equivalence
$$ 
\tau_{[2i-1, 2i]} \mathrm{TC}(R, (p); \mathbb{Z}_p) \simeq \tau_{[2i-1, 2i]}
\Sigma \mathrm{HC}(R, (p); \mathbb{Z}_p).$$
Now $\mathrm{TC}(R/p; \mathbb{Z}_p)$, $\mathrm{HC}(R; \mathbb{Z}_p)$, and
$\mathrm{HC}(R/p; \mathbb{Z}_p)$
are concentrated in even degrees since $R \in \qlrsp_{\mathbb{Z}_p}$. For the first claim see
\cite[Proposition 8.20]{BMS2}. 
The second and third follow from the filtrations
constructed in \cite{antieauperiodic} and \cite[Sec.~5]{BMS2}. 

It follows from the above definitions that 
$$\tau_{[2i-1, 2i]}
 \mathrm{TC}(R, (p); \mathbb{Z}_p) \simeq \mathrm{fib}( \mathbb{Z}_p(i)(R) \to
 \mathbb{Z}_p(i)(R/p))[2i],$$
 and from \cite[Sec.~5]{BMS2} 
and \cite{antieauperiodic} 
 that 
$$
\tau_{[2i-1, 2i]}
\Sigma \mathrm{HC}(R, (p); \mathbb{Z}_p) 
\simeq 
\mathrm{fib}(L \Omega_{R}/L \Omega_R^{\geq i} \to L
    \Omega_{R/p}/L \Omega_{R/p}^{\geq i})[2i-1].$$
Using these identifications, 
we deduce \eqref{integralBfibgr}. 
\end{proof}

We next identify the $p$-adic Chern character 
$\chi_i\colon \mathbb{Q}_p(i)(R/p) \to (L \Omega_R)_{\mathbb{Q}_p}$ on graded pieces
(from \Cref{Beilinsonongr})
more
explicitly. 
To this end, we prove the following basic result: 

\begin{proposition}[The image of $\chi_i$] 
\label{imagechern}
Let $R \in \qlrsp_{\mathbb{Z}_p}$. 
Then for each $i > 0$, the map (of $\mathbb{Q}_p$-vector
spaces) $\chi_i\colon \mathbb{Q}_p(i)(R/p) \to (L
\Omega_R)_{\mathbb{Q}_p} = \acrys(R/p)_{\mathbb{Q}_p}$ is injective, and has
image given by the $\phi = p^i$ eigenspace. 
\end{proposition} 
The main issue is the following: both the source and target of
$\chi_i$ are functors of $R/p$, thanks to de Rham--Witt theory. We have seen that the $p$-adic Chern character $\chi_i$ induces a natural map 
$\mathbb{Z}_p(i)(R/p) \to p^{-N} L \Omega_R $ for $R \in \mathrm{qSyn}_{\mathbb{Z}_p}$
for some $N$. 
However, it is not a priori obvious that the map $\chi_i$ arises from a natural
transformation of functors on $\mathbb{F}_p$-algebras (which would force it to
commute with Frobenius operators, for example). Our first goal is to verify
this.

\begin{lemma}[The Frobenius action on $\mathbb{Z}_p(i)(R)$]
\label{rem:frob}
For any $R \in \mathrm{qSyn}_{\mathbb{F}_p}$, the Frobenius on $R$ acts as multiplication by $p^i$ on
$\mathbb{Z}_p(i)(R)$. 
\end{lemma} 
\begin{proof} 
This reduces to the case of a  quasiregular semiperfect
$\mathbb{F}_p$-algebra by descent. But in this case, the identification of \Cref{Zpiforqrsperfs} clearly proves
the claim.
\end{proof} 

\begin{corollary} 
\label{howtogetnaturalmap}
The natural map $\chi_i\colon \mathbb{Z}_p(i)(R/p) \to p^{-N}(L \Omega_R) \to (L
\Omega_R)_{\mathbb{Q}_p}$, for $R
\in \mathrm{qSyn}_{\mathbb{Z}_p}$  arises (by precomposition with reduction mod $p$)
from a unique natural transformation
$\chi_i\colon \mathbb{Z}_p(i)\to p^{-N'} L W \Omega_{(-)}$ on
$\mathrm{qSyn}_{\mathbb{F}_p}$ for some $N' \geq N$. 
\end{corollary} 
\begin{proof} 
This follows from 
\Cref{adjunctionppower} and \Cref{rem:frob} (the latter shows that the
hypotheses of the former are satisfied), and then left Kan extension from
finitely generated $p$-complete polynomial rings. 
Uniqueness follows since these sheaves are torsion free. 
\end{proof}

Next, we consider the sheaf of graded $\mathbb{E}_\infty$-rings
$\bigoplus_{i = 0}^\infty \mathbb{Z}_p(i)$ on
$\mathrm{qSyn}_{\mathbb{F}_p}$. 
For each $N \geq 0$, we can also truncate to 
obtain a sheaf of graded $\mathbb{E}_\infty$-rings $\bigoplus_{i = 0}^N
\mathbb{Z}_p(i)$. 

\begin{proposition}\label{endoisscalar}
Let $f\colon \bigoplus_{i = 0}^N \mathbb{Z}_p(i) \to \bigoplus_{i = 0}^N
\mathbb{Z}_p(i)$ be a natural map of sheaves of graded $\mathbb{E}_\infty$ rings on
$\mathrm{qSyn}_{\mathbb{F}_p}$. Then there exists $\lambda \in \mathbb{Z}_p$ such that
in degree $i$, $f$ is given by multiplication by $\lambda^i$. 
\end{proposition} 

\begin{proof} 
We first observe that the only endomorphisms of $\mathbb{Z}_p(1)$ (as a
functor on $\mathrm{qSyn}_{\mathbb{F}_p}$) are given by
scalars. 
It suffices to verify this on quasiregular semiperfect algebras, and there
$\mathbb{Z}_p(1)$ is corepresentable (cf.~\cite[Props.~7.17, 8.20]{BMS2}) by  $\mathbb{F}_p[x^{1/p^\infty}]/(x-1)$,
on which $\mathbb{Z}_p(1)( \mathbb{F}_p[x^{1/p^\infty}]/(x-1)) \simeq
\mathbb{Z}_p$. 
So the endomorphism $f$ is given by a scalar action at least on
$\mathbb{Z}_p(1)$. 

Note that all these functors are left Kan extended from smooth algebras (to
the $p$-complete category), so 
$f$ is determined by the values on smooth $\mathbb{F}_p$-algebras. 
Furthermore, 
the map $f$
is determined by its values modulo $p^n$ for each $n$. 
However, classes in $H^i( \mathbb{Z}/p^n(i))$ are \'etale locally written as
sums of products of classes in $H^1( \mathbb{Z}/p^n(1))$ (thanks to
\Cref{Zpindsmooth}), so the value of $f$
on $\mathbb{Z}_p(1)$ determines the value of $f$ in general. The result now
follows because on smooth algebras, $\mathbb{Z}/p^n(i)$ is concentrated in
cohomological degree $i$ \'etale locally. 
\end{proof} 

\begin{proof}[Proof of \Cref{imagechern}] 
Recall that the map $\chi_i$ is actually a special case of a map 
$\mathbb{Z}_p(i)(R_0) \to p^{-N'}(L W \Omega_{R_0})$
defined on $R_0 \in \mathrm{qSyn}_{\mathbb{F}_p}$, by \Cref{howtogetnaturalmap}. 
For $R_0$ quasiregular semiperfect, 
we know that the Frobenius acts as $p^i$ on $\mathbb{Z}_p(i)(R_0)$, so 
we obtain a natural, multiplicative map 
$\mathbb{Q}_p(i)(R_0) \to (\acrys(R_0)^{\phi = p^i})_{\mathbb{Q}_p}$. 
We wish to see that these maps are isomorphisms. 

Now we know independently that $\mathbb{Q}_p(i)(R_0)$ is identified (for $i
> 0$) with 
$\acrys(R_0)^{\phi = p^i}_{\mathbb{Q}_p}$ via the theory of topological cyclic homology
(\Cref{charpdecomplete}, following \cite[Sec.~8]{BMS2}). 
Thus, we actually obtain natural, multiplicative (in $i$) maps 
$\mathbb{Q}_p(i)(R_0) \to \mathbb{Q}_p(i)(R_0)$
for $R_0 \in \mathrm{qSyn}_{\mathbb{F}_p}$, and we wish to see that these are
isomorphisms. Up to rescaling by a power of $p$, furthermore, they carry
$\mathbb{Z}_p(i)(R_0)$ into $\mathbb{Z}_p(i)(R_0)$. 
As we saw in 
\Cref{endoisscalar}, these maps are necessarily all given by scalar multiplication by
some $\lambda^i$ in degree $i$, for some $\lambda \in \mathbb{Z}_p$;
we know that $\lambda \neq 0$ (by comparing with $i =1 $, say), so the
result now follows. \end{proof} 

\subsection{Comparison of the $\zps{i}$ and $\mathbb Z_p(i)$}
Our main result is the following comparison, which establishes Theorem~\hyperref[thm:f]{F}. 
\begin{theorem} 
\label{FMcomparedtoBMS2}
For $R\in\mathrm{qSyn}_{\mathbb{Z}_p}$, there are natural, multiplicative identifications
$\zps{i}(R) \simeq \mathbb{Z}_p(i)(R)$ for $i \leq p-2$ and
$\qps{i}(R) \simeq \mathbb{Q}_p(i)(R)$ for all $i\ge0$.
\end{theorem} 

By \cite[Theorem 1.3]{Ge04}, for $i \leq p-2$ and for formally smooth schemes
over DVRs, syntomic cohomology in the above form (see also \cite{Ka87, Ku87}) is $p$-adic \'etale motivic cohomology. 

\begin{proof}[Proof of the rational case of \Cref{FMcomparedtoBMS2}]
Fix $i\ge0$. It is enough to prove the equivalences for all $R \in
\qlrsp_{\mathbb{Z}_p}$. Thanks to the odd vanishing conjecture proved in
\cite[Sec.~14]{Prisms}, we may
moreover assume that $R \in \qlrsp_{\mathbb{Z}_p}$ is such that
$\mathbb{Z}_p(i)(R)$ is concentrated in degree zero. In the homotopy cartesian
    square of \Cref{Beilinsonongr}, the terms $\mathbb{Q}_p(i)(R)$, $(L \Omega^{\geq
    i}_R)_{\mathbb{Q}_p}$, and $(L \Omega_R)_{\mathbb{Q}_p}$ are all
    concentrated in degree $0$, whence the same is true of the remaining term
    $\mathbb{Q}_p(i)(R/p)$ (that is, $\varphi - p^i\colon (L
    \Omega_R)_{\mathbb{Q}_p} \to (L \Omega_R)_{\mathbb{Q}_p}$ is surjective)
    and the fiber square is simply a cartesian and cocartesian square of
    abelian groups
\[ \xymatrix{
\mathbb{Q}_p(i)(R) \ar[d]  \ar[r] & \mathbb{Q}_p(i)(R/p) \ar[d]  \\
(L \Omega^{\geq i}_R)_{\mathbb{Q}_p} \ar[r] &  (L \Omega_R)_{\mathbb{Q}_p}.
}\]
Note that all the arrows are injections: the bottom since it is the inclusion of the Hodge filtration, the right by \Cref{imagechern}, and the others since the diagram is cartesian.

We claim that the map $\phi - p^i\colon (L \Omega_R^{\geq i})_{\mathbb{Q}_p} \to (L
\Omega_R)_{\mathbb{Q}_p}$ is surjective. 
Indeed, 
given $x \in (L \Omega_R)_{\mathbb{Q}_p}$, we can write $x = (\phi - p^i)(x')$
for some $x' \in (L \Omega_R)_{\mathbb{Q}_p}$; as we noted above, $\phi - p^i\colon (L \Omega_R)_{\mathbb{Q}_p} \to (L \Omega_R)_{\mathbb{Q}_p}$ is surjective. Using \Cref{imagechern} to identify the image of the vertical map, the diagram being cocartesian means that
$(L \Omega_R)_{\mathbb{Q}_p}^{\phi = p^i} \oplus (L \Omega_R^{\geq
i})_{\mathbb{Q}_p} \to (L \Omega_R)_{\mathbb{Q}_p} = \acrys(R/p)_{\mathbb{Q}_p}$
is surjective. So we can write $x' = y' + z'$ for $y' \in \ker(
\varphi - p^i)$ and $z' \in (L \Omega_R)_{\mathbb{Q}_p}^{\geq i}$. 
Applying $\varphi - p^i$, we get that $x = (\varphi - p^i)(z')$, proving the
claim as desired. 

Combining these observations, we have established a natural  identification
\[ \mathbb{Q}_p(i)(R) = (L \Omega_R)_{\mathbb{Q}_p}^{\phi = p^i}\cap (L \Omega_R^{\geq
i})_{\mathbb{Q}_p}\simeq \mathrm{fib}( \varphi - p^i\colon (L \Omega_R^{\geq
i})_{\mathbb{Q}_p} \to (L \Omega_R)_{\mathbb{Q}_p}),  \]
as desired.
\end{proof}

\begin{corollary}[A description of $\mathrm{TC}(R; \mathbb{Q}_p)$] 
Let $R$ be any simplicial commutative ring. Then there is a natural equivalence
\[ \TC(R; \mathbb{Q}_p) \simeq \bigoplus_{i \geq 0} \mathrm{fib}( \varphi -
p^i\colon L \Omega_R^{\geq i} \to L \Omega_R )_{\mathbb{Q}_p}.  \]
\end{corollary} 
\begin{proof} 
The map from $R$ to its derived $p$-adic completion induces an equivalence on all the terms appearing in the statement: for derived de Rham cohomology and its Hodge filtration this follows from base change, while it holds for $\THH(-;\mathbb Z_p)$ (and hence $\TC(-;\mathbb Q_p)$) by \cite[Lem.~5.2]{CMM}. We may therefore assume $R$ is $p$-complete, at which point we know from \Cref{LKEmotivicfilt} that $\TC(R;\mathbb Q_p)$ admits a complete descending filtration with associated graded given by $\mathbb Q_p(i)(R)[2i]$, for $i\ge0$. Using Adams operations on $\TC$
as in \cite[Sec.~9.4]{BMS2}, we can
split the filtration functorially. 
Combining with the rational part of \Cref{FMcomparedtoBMS2} (or more precisely
its left Kan extension to $p$-complete simplicial commutative rings), the claim follows.
\end{proof}

Next, we will prove the integral case of \Cref{FMcomparedtoBMS2}. 
The main step is to show that the $\zps{i}(-)$ for $i\leq p-2$  are discrete, as sheaves on $\mathrm{qSyn}_{\mathbb{Z}_p}$;
this is the analog of \Cref{oddvanishingconj}, i.e., of the odd vanishing conjecture. 
To see this, we will use the odd vanishing conjecture itself and some cases of
the results of Li--Liu \cite{LL20}. 

\begin{proposition} 
\label{discretenessFMsheaf}
As $D(\mathbb{Z}_p)^{\geq 0}$-valued sheaves on $\mathrm{qSyn}_{\mathbb{Z}_p}$,
\begin{enumerate}
    \item[{\rm (1)}] $\zps{i}(\cdot)$ is discrete and torsion-free for $0\leq
        i\leq p-2$ and
    \item[{\rm (2)}] $\qps{i}(\cdot)$ is discrete. 
\end{enumerate}
\end{proposition} 
\begin{proof} 
Item (2) has already been proved above (in light of the odd vanishing
conjecture \cite[Th.~14.1]{Prisms}), so here we prove (1).  
It suffices to show that $\zps{i}(\cdot)$ is discrete and torsion-free as a
$D(\mathbb{Z}_p)^{\geq 0}$-valued sheaf on $\mathrm{qSyn}_{\mathcal{O}_C}$, for
$\mathcal{O}_C$ the ring of integers in $C = \mathbb{C}_p$. 
In particular, we will show that for any $R \in 
\mathrm{qSyn}_{\mathcal{O}_C}$, there exists a quasisyntomic cover $R \to R'$ such that 
$R'$ is quasiregular semiperfectoid and such that 
$\phi/p^i - 1 \colon L \Omega_{R'}^{\geq i} \to L \Omega_{R'}$ is surjective
for all $i \leq p-2$. 

To this end, on 
$\qlrsp_{\mathcal{O}_C}$, we consider the (non-Nygaard-complete) sheaf
$R \mapsto \Prismnc_R$ (cf.~\cite[Sec.~7]{Prisms}) together with its Nygaard filtration $\mathcal{N}^{\geq
\star} \Prismnc_R$ and divided Frobenius maps $\frac{\phi}{\widetilde{\xi}^i}
\colon \mathcal{N}^{\geq i} \Prismnc_R \to \Prismnc_R$, for $\widetilde{\xi}$
a generator of the ideal defining the prism structure
on $\Prismnc_{\mathcal{O}_C}$. 
For any $R \in \qlrsp_{\mathcal{O}_C}$, we use \cite[Th.~3.5,
Rem.~3.6]{LL20} (applied to the perfect prism $(\ainf(\mathcal{O}_C),
\widetilde{\xi})$)  
or the crystalline comparison 
\cite[Th.~5.2]{Prisms}
to obtain a natural map 
of $\delta$-rings
\begin{equation}\label{prismtoderham} \Prismnc_R \to  \acrys(R/p) \simeq L \Omega_R.
\end{equation}
This map of $\delta$-rings 
carries $\widetilde{\xi}$ to an element of the form $p \phi(u)$, for $u \in (L
\Omega_R)^{\times}$ and $\phi \colon L \Omega_R \to L \Omega_R$ the crystalline Frobenius. 
By \cite[Th.~4.13]{LL20}, 
the map 
\eqref{prismtoderham} canonically refines to a filtered map
$\mathcal{N}^{\geq \star} \Prismnc_{R} \to  L \Omega_R^{\geq \star}$. 
Moreover, the filtered map 
$\mathcal{N}^{\geq \star} \Prismnc_R \to L \Omega_R^{\geq \star}$ induces an
isomorphism on associated graded terms in degrees $\leq p-1$ (\emph{loc.~cit.}). 

By the odd vanishing conjecture \cite[Sec.~14]{Prisms}, there exists a basis of
objects $R \in \qlrsp_{\mathcal{O}_C}$ for which the map
\( \phi/\widetilde{\xi}^i -1 \colon \mathcal{N}^{\geq i} \Prismnc_R \to
\Prismnc_R \) is surjective. Choose any such $R$.  Consider now the evident commutative diagram
\begin{equation} \label{frobderhamprism} \xymatrix{
\mathcal{N}^{\geq i} \Prismnc_R \ar[d]  \ar[r]^{\phi/\widetilde{\xi}^i-1} &
\Prismnc_R \ar[d]  \\
L \Omega_R^{\geq i} \ar[r]^{\phi/\widetilde{\xi}^i -1} & L \Omega_R
}.\end{equation}
For $i \leq p-2$, it follows from 
the surjectivity of the top horizontal arrow in 
\eqref{frobderhamprism} (and that $\widetilde{\xi}$ maps to $p \phi(u)$ in $L
\Omega_R$, for $u \in (L \Omega_R)^{\times}$) 
that the composite 
\begin{equation} \label{secondsurjmap} L \Omega_R^{\geq i} \xrightarrow{\phi/p^i -1}  L \Omega_R \to L \Omega_R/L
\Omega_R^{\geq i+1} \simeq \Prismnc_R/\mathcal{N}^{\geq i+1} \Prismnc_R
\end{equation}
is surjective. 
Now the map 
$$L \Omega_R^{\geq i + 1} \subset L \Omega_R^{\geq i} \xrightarrow{\phi/p^i -1}
L \Omega_R \to L \Omega_R / p $$
is simply the inclusion composed with reduction mod $p$, since $\phi/p^i$ is
divisible by $p$ on  
$L \Omega_R^{\geq i + 1}$. 
This observation combined with the surjectivity of \eqref{secondsurjmap} and
$p$-completeness now gives the surjectivity of $L \Omega_R^{\geq i}
\xrightarrow{\phi/p^i -1} L \Omega_R$ as desired. 
\end{proof}

\begin{proof}[Proof of \Cref{FMcomparedtoBMS2} for $i \leq p-2$]
Fix $i \leq p-2$ and $R \in \qlrsp_{\mathbb{Z}_p}$ such that $\mathbb{Z}_p(i)(R)$ and $\zps{i}(R)$
are concentrated
in degree zero for all $i \leq p-2$; we can do this by the odd vanishing
conjecture and \Cref{discretenessFMsheaf}. 
In this case, the map of discrete abelian groups $\mathbb{Z}_p(i) (R) \to \mathbb{Z}_p(i)(R/p)$
is injective and has torsion free cokernel, thanks to the equivalence
\eqref{integralBfibgr}. So from the Beilinson fiber square on graded pieces (\Cref{Beilinsonongr}) and the description of the image of $\chi_i$ (\Cref{imagechern}), we find that $\mathbb{Z}_p(i)(R) \subset \mathbb{Z}_p(i)(R/p) = (L
\Omega_R)^{\phi = p^i}$ is
the submodule 
consisting of those elements such that the image in $L \Omega_R$ belongs to
$L \Omega_R^{\geq i}$. 
In particular, it is precisely the kernel of $\varphi/p^i -1 \colon L \Omega_R^{\geq
i} 
\to L \Omega_R$. Since this map is surjective, we get the natural 
equivalence $\mathbb{Z}_p(i)(R) \simeq \mathbb{Z}_p(i)(R)^{\mathrm{FM}}$ as
desired. 
\end{proof}

\section{Examples}

\subsection{$\mathrm{K}$-theory of $p$-adic fields}

Let $F$ be a complete discretely valued field of characteristic $0$ with ring
of integers $\Oscr_F \subset F$ and 
perfect residue field $k$ of characteristic $p$. In this subsection, we will use
the Beilinson fiber sequence to recover various calculations of the $p$-adic
$\K$-theory of $F$. All these results are previously known at least in the case of $F$ local; see
\cite[Theorem 61]{weibel-integers} for a detailed survey.

\begin{theorem}\label{thm:Wagoner_revisited}
The homotopy groups of $\K(F;\QQ_p)$ are given (as $\mathbb Q_p$-vector spaces) as follows: 
\begin{enumerate}
\item $\K_{2s}(F;\QQ_p)=0$ for $s>0$;
\item there is a natural isomorphism $\K_{2s-1}(F;\QQ_p)\simeq F$ for each $s>1$;
\item there is a natural short exact sequence $0\to F\to \K_1(F;\QQ_p)\to\QQ_p\to 0$.
\end{enumerate} 
\end{theorem}
\begin{proof}
Since $k$ is perfect, we have that $\K_i(k; \mathbb{Z}_p) = \mathbb{Z}_p$ for
$i = 0$ and $0$ otherwise, cf.~\cite[Th.~5.4]{Hi} and \cite[Cor.~5.5]{Kra}. 
Taking the d\'evissage cofiber sequence $\K(k)\rightarrow\K(\Oscr_F)\rightarrow\K(F)$ with $\ZZ_p$-coefficients shows that $\K_i(\Oscr_F;\ZZ_p)\iso\K_i(F;\ZZ_p)$ for $i\neq 1$ and that there is an exact sequence $$0\rightarrow\K_1(\Oscr_F;\ZZ_p)\rightarrow\K_1(F;\ZZ_p)\rightarrow\ZZ_p\rightarrow
    0,$$
where the map $\K_1(F;\ZZ_p)\iso F^\times\otimes_\ZZ \ZZ_p\rightarrow\ZZ_p$ is induced by 
    the $p$-adic valuation.
    
Next, since $\K(\Oscr_F/p;\QQ_p)\we\K(k;\QQ_p)\we\QQ_p$, the Beilinson fiber
square (Theorem~\hyperref[thm:a]{A}) for $\Oscr_F$ yields a fiber sequence
$$\Sigma\HC(\Oscr_F;\QQ_p)\rightarrow\K(\Oscr_F;\QQ_p)\rightarrow\QQ_p.$$ But
note that the cyclic homology term may be equivalently written as $\Sigma\HC(F/F_0)$,
where $F_0:=W(k)[\tfrac1p]$; indeed, the vanishing of the $p$-adic completion of
the cotangent complex $L_{W(k)/\ZZ}$ implies that $\HC(\Oscr_F;\ZZ_p)\simeq
\HC(\Oscr_F/\Oscr_{F_0};\ZZ_p)$, but $\HC(\Oscr_F/\Oscr_{F_0})$ is already
derived $p$-adically complete since its homology groups are all finitely
generated $\Oscr_F$-modules. The Beilinson fiber sequence therefore implies
that \[\tau_{\ge2}\Sigma\HC(F/F_0)\simeq\tau_{\ge2}\K(F;\QQ_p)\] and
\[\HC_0(F/F_0)\simeq \K_1(\Oscr_F;\QQ_p).\] The proof is completed by noting
that, since $F$ is an \'etale $F_0$-algebra, its cyclic homology is given by
$\HC_i(F/F_0)=F$ for $i\ge0$ even and is $0$ otherwise.
\end{proof}

\begin{example}[Local fields]\label{rmk:Wagoner}
If $F$ is a finite extension of $\QQ_p$, of degree $d$, the theorem shows that 
\begin{equation} \label{dimcalcF} \dim_{\QQ_p}\K_{2s-1}(F;\QQ_p)=\begin{cases}d+1&\text{if $s=1$,}\\
d&\text{otherwise.}\end{cases}\end{equation} This dimension calculation is a
classical result, arising from Wagoner's \cite{wagoner} calculation of the
ranks of the continuous $\K$-groups and Panin's \cite{panin} proof of an early
case of the $p$-adic continuity of $\K$-theory.

In addition, the dimension calculation \eqref{dimcalcF} is in accordance with the Beilinson--Lichtenbaum conjecture for
$F$. Recall that the Beilinson--Lichtenbaum conjecture, prior to its
general proof by Rost--Voevodsky for all fields, was proved by Hesselholt--Madsen
\cite{HM03} in this case when $p > 2$ using $\mathrm{TC}$-theoretic methods. 
Since $F$ has cohomological dimension $2$, 
the Beilinson--Lichtenbaum conjecture 
predicts
$\K_{2s-1}(F; \mathbb{Q}_p) \simeq H^1_{\et}(F, \mathbb{Q}_p(s))$ and
$\K_{2s-2}(F; \mathbb{Q}_p) \simeq H^2_{\et}(F, \mathbb{Q}_p(s))$ for $s > 0$. 
Then the dimensions in \eqref{dimcalcF} agree with the dimensions of the
$\mathbb{Q}_p$-cohomology of $F$, as computed via Tate's local duality and Euler characteristic
formula 
(see \cite[VII.3]{Neukirch} for an account). 
\end{example}

\begin{example}[Integral calculation, unramified case]\label{rmk:unramified}
Assume in this example that $p> 3$ so that the results hold in a non-empty
range. We will show that, in the range $1\le i\le 2p-5$, the $p$-adic $\K$-groups of $W(k)$ are given by 
\begin{equation} \K_i(W(k);\ZZ_p)\simeq\begin{cases}W(k) & \text{if }i=2s-1,
\\ 0 & \text{if }i\textrm{ is even.}\end{cases} \end{equation} 
Note that for $k$ finite, the calculation of the entire homotopy type of $\K( W(k); \mathbb{Z}_p)$
is carried out by B\"okstedt--Madsen \cite{BM95} at odd primes and Rognes
\cite{Rog99} at $p= 2$ (for $k = \mathbb{F}_2$); again, see  
\cite[Theorem~61]{weibel-integers} for a survey for all of these results. 
	 
	 The integral form of the Beilinson fiber sequence (Corollary
    \hyperref[quasiisogenythmB]{B}) takes the form of a natural fiber sequence 
\[\tau_{\le 2p-5}\Sigma\HC(W(k),(p);\ZZ_p)\rightarrow\tau_{\le 2p-5}\K(W(k);\ZZ_p)\to \ZZ_p.\] As in the proof of Theorem \ref{thm:Wagoner_revisited}, the cyclic homology term may be replaced by the cyclic homology  $\tau_{\le 2p-5}\Sigma\HC(W(k),(p)/W(k))$ over $W(k)$. 

A standard calculation of derived de Rham cohomology with divided powers (as in
\cite[Prop.~3.16]{SzZa18}) gives $
L \Omega_{k}/L \Omega^{\geq s}_k
\simeq W(k)/p^s$ for $s\le p-1$; in general, $L \Omega_k/L \Omega^{\geq s}_k$ is
discrete for all $s$. 
Using the filtrations of \cite{antieauperiodic}, we conclude
for $i\le 2p-1$,
\begin{equation} \pi_i\HC(k/W(k))\simeq\begin{cases}
        W(k)/p^{s+1}&\text{if $i=2s\geq 0$,}\\
    0&\text{otherwise.}\end{cases} \end{equation}
See also \cite[Prop.~7.2]{Brunfiltered} for this calculation.
Also, $\HC_i(W(k); \mathbb{Z}_p) \iso \HC_i(W(k)/W(k))\iso W(k)$ for $i\ge0$ even and
is
zero otherwise. Thus, in the range $0\le i\le 2p-2$, we deduce that
    \[\pi_i\HC((W(k),(p)/W(k))\simeq\begin{cases}p^{s+1}W(k) & \text{if }i=2s, \\ 0 &
    \text{if }i\textrm{ is odd}\end{cases}\] This completes the proof.
\end{example}

\begin{remark}[Integral calculation, ramified case]
    Assume now that $F$ is ramified so that $\Oscr_F/p\iso k[x]/(x^e)$, where $e$
    is the absolute ramification degree of $F$.
    Corollary~\hyperref[quasiisogenythmB]{B} implies,
    by rewriting the $p$-adic cyclic homology as non-completed cyclic
    homology with respect to $W(k)$, that $\tau_{\le
    2p-5}\Sigma\HC(\Oscr_F,(p)/W(k))\simeq \tau_{\le
    2p-5}\K(\Oscr_F,(p);\ZZ_p)$.

We now appeal to the fact that the algebraic $\K$-theory of truncated
polynomial rings over fields is
known~\cite{hesselholt-madsen-truncated,speirs-revisited}. The positive even
$p$-adic $\K$-groups of $k[x]/(x^e)$ vanish so that we get
%a surjection
%$$\pi_{2p-5}\Sigma\HC(\Oscr_F,(p)/W(k))\rightarrow\K_{2p-5}(\Oscr_F;\ZZ_p)$$
%and hence 
$5$-term exact sequences
\begin{gather*}
0\rightarrow
\pi_{2s-1}\Sigma\HC(\Oscr_F,(p)/W(k))\rightarrow\K_{2s-1}(\Oscr_F;\ZZ_p)\\
\rightarrow
\WW_{se}(k)/V_e\WW_s(k)\rightarrow\pi_{2s-2}\Sigma\HC(\Oscr_F,(p)/W(k))\rightarrow\K_{2s-2}(\Oscr_F;\ZZ_p)\rightarrow 0
\end{gather*}
for $2\leq s\leq p-2$. In low degrees, this gives a computation of the integral
$p$-adic $\K$-groups of $\Oscr_F$ which is independent of
 \cite{HM03}; on the other hand, using this calculation and
\cite{HM03}, we can view the computation as giving information about the low-degree
\'etale cohomology of $F$.
\end{remark}

We can also carry out previous types of calculations for the syntomic complexes
$\mathbb{Z}_p(i)$ rather than $\K$-theory. 

\begin{theorem}[Syntomic cohomology of DVRs]
    Let $\Oscr_F$ be a complete discrete valuation ring of mixed characteristic
    $(0,p)$ with perfect residue field $k$.
\begin{enumerate}
    \item[{\rm (1)}] We have natural identifications
\begin{equation} \QQ_p(i)(\Oscr_F)\simeq \begin{cases}R
\Gamma_{\mathrm{proet}}( \spec(k), \QQ_p )& \text{if } i=0,\\
F[-1] & \text{if }i>0.\end{cases} \end{equation}
\item[{\rm (2)}]
    In the unramified case $\Oscr_F = W(k)$,
\begin{equation} 
\ZZ_p(i)(W(k))
\simeq \begin{cases} 
R\Gamma_{\mathrm{proet}}(\spec(k), \ZZ_p) & \text{if } i = 0,\\ 
W(k)[-1] & \text{if } 0 < i \leq p-2.
 \end{cases}
\end{equation} 
\end{enumerate}
\end{theorem} 
\begin{proof} 
For (1), 
we have (where the first equivalence follows from \Cref{rem:frob} and taking
powers of Frobenius)
$$\QQ_p(i)(\Oscr_F/p) \simeq 
\mathbb{Q}_p(i)(k) \simeq 
\begin{cases}
       R \Gamma_{\mathrm{proet}}( \spec(k), \mathbb{Q}_p) 
		  &\text{if $i=0$,}\\
        0&\text{otherwise.}
    \end{cases}$$
    Therefore, the fiber sequence from \Cref{Beilinsonongr} yields equivalences
	 \[\QQ_p(i)(\Oscr_F)\simeq \begin{cases}
	R \Gamma_{\mathrm{proet}}(\spec(k), \mathbb{Q}_p) 
	 & \text{if } i=0,\\ \left(L
	 \Omega_{\Oscr_F}/L\Omega^{\ge i}_{\Oscr_F}\right)_{\QQ_p}[-1] & \text{if
	 }i>0.\end{cases}\] As in the proof of Theorem \ref{thm:Wagoner_revisited},
	 the latter truncated $p$-adic derived de Rham cohomologies may be computed
	 as the analogous un-completed derived de Rham cohomologies for $F_0\to F$;
     since $L_{F/F_0}\simeq 0$, we conclude that \(\QQ_p(i)(\Oscr_F)\simeq
     F[-1]\) for $i>0$.

The integral claim follows from 
\Cref{FMcomparedtoBMS2}. 
Indeed, 
we find that $\mathbb{Z}_p(0)(W(k)) = \mathrm{fib}( \varphi-1\colon W(k) \to W(k))
\simeq R \Gamma_{\mathrm{proet}}( \spec(k), \mathbb{Z}_p)$. 
For $i > 0$, 
we get $\mathbb{Z}_p(i)(W(k)) = \mathrm{fib}( \varphi/p^i - 1\colon 0 \to W(k))$, so
the claim follows. 
\end{proof}

\subsection{Perfectoid rings}
In this section, we apply the Beilinson fiber square to a perfectoid ring. 
The main result (which was indicated to us by Scholze) is that 
it recovers the fundamental exact sequence in $p$-adic Hodge theory. 

Let $R$ be a perfectoid ring. We review the period rings associated to $R$ and their interpretation via
derived de Rham theory, cf.~\cite{Bei12, Bhattpadic}.

\begin{construction}[Period rings] 
\label{periodrings}
Let $R$ be a perfectoid ring. 
\begin{enumerate}
\item  
As before, we have Fontaine's ring $\ainf(R)$ equipped with the canonical map
$\theta\colon \ainf(R) \to R$ with kernel $(\xi)$. Here $\ainf(R)$ is also the prismatic cohomology
$\Prism_R$; the Nygaard filtration is the $\xi$-adic filtration. 
\item
We have $\acrys(R) = \acrys(R/p)$, the $p$-adic completion of the divided power envelope of $(\xi) \subset \ainf(R)$; we
have $\acrys(R) \simeq L \Omega_R$ is the derived de Rham cohomology of $R$. 
The Hodge filtration is given by the divided power filtration. 
We let $\bcrys(R) =  \acrys(R)[1/p] = (L \Omega_R)_{\mathbb{Q}_p}$; the ring
$\bcrys(R)$ also inherits a Frobenius operator $\varphi$. 
\item We have $\bdr^+(R) = \varprojlim (\ainf(R)/\xi^n [1/p]) $. 
The ring $\bdr^+(R)$ can also be obtained as the Hodge completion of $(L
\Omega_R)_{\mathbb{Q}_p}$.  
The Hodge filtration yields a filtration 
(the $\xi$-adic filtration) on $\bdr^+(R)$. 
\end{enumerate}
\end{construction} 
Our goal is now to recover the following result in $p$-adic Hodge theory,
cf.~\cite[Theorem 5.3.7]{Fo94} and \cite[Theorem 6.4.1]{FF}.

\begin{theorem}[The fundamental exact sequence] 
\label{fundexactseq}
For any $R \in \pfd$ and $i>0$, there is a natural pullback square
in $D(\mathbb{Q}_p)$,
\begin{equation}\label{fundfibsq}\begin{gathered}
\xymatrix{\R\Gamma_{\mathrm{proet}}(\spec(R[1/p]), \mathbb{Q}_p(i)) 
\ar[d]  \ar[r] & 
 \bcrys(R )^{\varphi = p^i} \ar[d]  \\
 \mathrm{Fil}^{\geq i}\bdr^+(R) \ar[r] &  \bdr^+(R) 
 } \end{gathered}
.\end{equation}
\end{theorem} 

\begin{example} 
    When $R=\Oscr_C$, where $C$ is a complete algebraically closed
    nonarchimedean field, the fundamental exact sequence is often written as the
    exact sequence
\[ 0 \to \mathbb{Q}_p(i) \to \bcrys(\mathcal{O}_C)^{\varphi = p^i} \to
\bdr^+(\mathcal{O}_C)/\mathrm{Fil}^{\geq i} \bdr^+(\mathcal{O}_C) \to 0\]
of abelian groups.
\end{example} 

\begin{proof}[Proof of \Cref{fundexactseq}] 
We apply \Cref{Beilinsonongr} to the perfectoid ring $R$ and obtain a fiber
square
\begin{equation} \label{fibsq46}\begin{gathered} \xymatrix{\mathbb{Q}_p(i)(R) \ar[d] \ar[r] & \mathbb{Q}_p(i)(R/p) \ar[d] \\
(L \Omega_R^{\geq i})_{\mathbb{Q}_p} \ar[r] & 
(L \Omega_R)_{\mathbb{Q}_p}.
}\end{gathered} \end{equation}
By \cite[Theorem 9.4]{Prisms}, the first term is identified with
$\R\Gamma_{\mathrm{proet}}( \spec(R[1/p]),
\mathbb{Q}_p(i))$. 
The ring $R/p$ is quasiregular semiperfect, so we have
$\mathbb{Q}_p(i)(R/p) \simeq \bcrys(R)^{\varphi = p^i}$
\cite[Sec. 8]{BMS2}. 

Note that we can replace the square \eqref{fibsq46} by Hodge completing the
bottom row and
it will still remain cartesian, since the homotopy fibers do not change. 
This yields a new homotopy cartesian square
where one identifies the rings as in \Cref{periodrings}, and then the result
follows. 
\end{proof}

\begin{remark}[Identifying the maps] 
Unfortunately, in general we do not know a good way of identifying the map
$\K(\mathcal{O}_C/p;\QQ_p) \to
\mathrm{HP}(\mathcal{O}_C; \mathbb{Q}_p)$ with the usual map in the fundamental
exact sequence. However, we can argue that it has to match with the
usual map, at least for $C = \mathbb{C}_p$, by appealing to some general results. 
For simplicity, in this example, we drop the argument of the perfectoid ring,
i.e., we write $\bdr^+$ for $\bdr^+(\mathcal{O}_{\mathbb{C}_p})$, etc. 

Our first goal is to identify the map  
obtained from $\pi_2$ in \eqref{fundfibsq},
\begin{equation} (\bcrys)^{\varphi = p} \to
\bdr^+; \label{map12} \end{equation}
by construction, it is $\mathrm{Gal}(\mathbb{Q}_p)$-equivariant.
Now we have a (Galois-equivariant) short exact sequence $$0 \to
\mathbb{Q}_p(1) \to (\bcrys)^{\varphi = p} \to \mathbb{C}_p \to 0$$ as above. 
Furthermore, the map 
\eqref{map12}
when restricted to the submodule 
$\mathbb{Q}_p(1) \subset ( \bcrys)^{\varphi = p}$  is
essentially determined: it is given by the
$\mathrm{dlog}$ map to derived de Rham cohomology
 as it comes from the usual Chern character $\K( \mathcal{O}_C;
 \mathbb{Q}_p) \to \HC^-(\mathcal{O}_C; \mathbb{Q}_p)$, for $\mathcal{O}_C$. 
 As is proved in~\cite[Prop. 2.17]{fontaine-bt}, the image of a generator of $\mathbb{Q}_p(1)$ gives a uniformizer of
$\bdr^+$ (which is  a DVR).

Recall that $\bdr^+$  has a complete, exhaustive
filtration (via powers of the augmentation ideal) with associated graded given
by $\mathbb{C}_p, \mathbb{C}_p(1), \mathbb{C}_p(2), \dots $. 
Moreover, there are no $\mathrm{Gal}(\mathbb{Q}_p)$-equivariant maps
$\mathbb{C}_p \to \bdr^+$ (see~\cite[Rem.~1.5.8]{Fo94}).
Thus there is at most one (and hence exactly one, by construction)
Galois-equivariant map
$ (\bcrys)^{\varphi = p} \to \bdr^+$ which extends
the $\mathrm{dlog}$ map. 
This shows that the map \eqref{map12} is actually completely determined by its
behavior on $\mathbb{Q}_p(1)$. 

Now by a deep result of Fargues--Fontaine, the graded ring $\bigoplus_{i \geq 0}
\bcrys(\mathcal{O}_C)^{\varphi = p^i}$ is
generated in degree one \cite[Th.~6.2.1]{FF}, so the maps for higher $i$ are determined by their
behavior for $i  = 1$ by multiplicativity. 
In particular, these
observations show that the maps in the fundamental exact sequence, although here
they are produced by topological means, are entirely determined by their value
on $\mathbb{Q}_p(1)$, as long as they are Galois-equivariant. 
\end{remark}

\subsection{Application to $p$-adic nearby cycles}

Let $C$ be an algebraically closed,  complete nonarchimedean field of mixed
characteristic $(0, p)$. 
In \cite[Sec. 10]{BMS2}, an explicit description of the $\mathbb{Z}_p(i)$
sheaves is given for smooth formal schemes over $\mathcal{O}_C$. 
Using this, we can recover some cases of comparison results of Colmez--Nizio\l~
\cite{CN17} and Tsuji \cite{Tsuji1999}, cf.~also Kato \cite{Ka87}.

\begin{definition}[$p$-adic nearby cycles] 
Let $\mathfrak{X}$ be a formal scheme over $\mathcal{O}_C$. 
We consider the pro-\'etale site $\mathfrak{X}_{\mathrm{proet}}$ of $\mathfrak{X}$ (equivalently, of its special
fiber), cf.~\cite{BS15}. 

For each $i$, we consider\footnote{The functor $\psi$ here should refer to the generic
fiber functor, but we do not define it here to avoid technicalities.} the sheaf of \emph{$p$-adic nearby cycles} 
$R \psi_* ( \mathbb{Z}_p(i)) $, which is a
$D(\mathbb{Z}_p)^{\geq 0}$-valued sheaf
 on $\mathfrak{X}_{\mathrm{proet}}$. 
Explicitly, 
given an affine pro-\'etale open $\spf A \to \mathfrak{X}$, 
we have that $R \Gamma(\spf A, R \psi_*(\mathbb{Z}_p(i))) \simeq 
R \Gamma_{\mathrm{proet}}( \spec A[1/p], \mathbb{Z}_p(i))$ is the pro-\'etale cohomology
of $A[1/p]$ with values in the (usual) sheaf $\mathbb{Z}_p(i)$.\footnote{Here we can consider either the
scheme $\spec(A[1/p])$ or the rigid analytic generic fiber by the affinoid
comparison theorem \cite[Cor.~3.2.2]{Hu96}.} 
\end{definition}

\begin{theorem}[Bhatt--Morrow--Scholze \cite{BMS2}] 
\label{Zpinsmoothcase}
Let $R$ be a formally smooth $\mathcal{O}_C$-algebra and let $\mathfrak{X} =
\spf(R)$. Then, as sheaves on $\mathfrak{X}_{\mathrm{proet}}$, we have a natural equivalence
$\mathbb{Z}_p(i) \simeq \tau^{\leq i} R \psi_*( \mathbb{Z}_p(i))$. 
In particular, it follows that
\[ \mathbb{Z}_p(i)(R) \simeq R \Gamma( \mathfrak{X}_{\mathrm{proet}}, 
\tau^{\leq i} R \psi_*( \mathbb{Z}_p(i))). 
\]
\end{theorem}

To apply this, 
let $K$ be a \emph{discretely} valued field with perfect residue field $k$, ring
of integers $\mathcal{O}_K \subset K$, and uniformizer $\pi \in
\mathcal{O}_K$; suppose  
$K \subset C$ (e.g., we could take $C = \widehat{\overline{K}}$). 
Let $\mathfrak{X}_0$ be a smooth proper formal scheme over $\mathcal{O}_K$ with generic
fiber $X_0$, a smooth proper rigid space over $K$. 

\begin{cons}[de Rham cohomology of formal schemes and rigid spaces]
We let $L \Omega_{\mathfrak{X}_0/\mathcal{O}_K}$ denote the ($p$-adic) derived de Rham cohomology of
$\mathfrak{X}_0$ over $\mathcal{O}_K$ equipped with its Hodge filtration. 
In fact,  
$L \Omega_{\mathfrak{X}_0/\mathcal{O}_K}$ is also the $p$-complete \emph{usual} de Rham complex since
$\mathfrak{X}_0$ is formally smooth over $\mathcal{O}_K$, cf.~\cite{Bhattpadic}, and the Hodge filtration is a finite filtration. 
    We let (by a slight abuse of notation)
    $\Omega_{X_0/K}=(L\Omega_{\mathfrak{X}_0/\Oscr_K})_{\QQ_p}$ denote the rationalization, which we can interpret as
the de Rham cohomology of the rigid generic fiber $X_0$. 
Note that $L \Omega_{\mathfrak{X}_0/\mathcal{O}_K}$ is a perfect 
$\mathcal{O}_K$-module, and $\Omega_{X_0/K}$ is a perfect $K$-module. 
\end{cons}

We also consider the ring $\bdr^+ = \bdr^+(\mathcal{O}_C)$ with its $\xi$-adic
filtration. 
Together, it follows that 
$\Omega_{X_0/K} \otimes_{K} \bdr^+$ admits a filtration in the
derived $\infty$-category $D(K)$. 
Then one has the following result, a special case of results of \cite{CN17, Tsuji1999} in
the case of good reduction; note that \cite{CN17, Tsuji1999} treat the more general
semistable case, which we do not consider here.  
In the following, all references to $\acrys, \bdr^+$, etc.~will implicitly 
be with respect to the perfectoid ring $\mathcal{O}_C$. 

\begin{theorem}[Cf.~Colmez--Nizio\l~\cite{CN17}, Tsuji~\cite{Tsuji1999}] 
Let $\mathfrak{X}_0/\mathcal{O}_K$ be a smooth proper formal scheme. 
Let $\mathfrak{X}$ denote the base change of $\mathfrak{X}_0$ to $\mathcal{O}_C$ and
$\overline{\mathfrak{X}_0}$ its reduction modulo $\pi$. 
For each $i \geq 0$, we have a natural pullback square in $D(\mathbb{Q}_p)$,
\[   
\xymatrix{
R \Gamma( \mathfrak{X}_{\mathrm{proet}}, 
\tau^{\leq i} R \psi_*( \mathbb{Q}_p(i))) 
\ar[d]  \ar[r] & 
(\acrys {\otimes_{W(k)}} R \Gamma_{\mathrm{crys}}(
\overline{\mathfrak{X}}_0/W(k)))^{\varphi = p^i} [1/p] \ar[d]  \\
\mathrm{Fil}^{\geq i}
(\Omega_{X_0/K}{\otimes}_{K} 
\bdr^+)  \ar[r] &   (\Omega_{X_0/K} \otimes_{K} 
\bdr^+).
}\]
\end{theorem} 
\begin{proof} 
We claim that this follows from 
\Cref{Beilinsonongr}, applied to $\mathfrak{X}$. 
Note that 
\eqref{Bfibgr} gives a fiber square
\begin{equation}
\label{secondfibsq}
    \begin{gathered}
\xymatrix{
\mathbb{Q}_p(i)( \mathfrak{X}) \ar[d]  \ar[r] & 
\mathbb{Q}_p(i)( \mathfrak{X}/\pi) \ar[d] \\
( L \Omega_{\mathfrak{X}/\mathcal{O}_K}^{\geq i})_{\mathbb{Q}_p} \ar[r] &  (
L \Omega_{\mathfrak{X}/\mathcal{O}_K})_{\mathbb{Q}_p}
}
    \end{gathered}
\end{equation}
in fact, $\mathbb{Q}_p(i)( \mathfrak{X}/p) \to \mathbb{Q}_p(i)( \mathfrak{X}/\pi)$ is
an isomorphism thanks to \Cref{rem:frob}.

The top left term in \eqref{secondfibsq} is identified via \Cref{Zpinsmoothcase}. 
For the top right, we observe that there is a natural equivalence
\[ 
\mathfrak{X} \otimes_{\mathcal{O}_C} \mathcal{O}_C/\pi \simeq 
\overline{\mathfrak{X}_0} \otimes_k \mathcal{O}_C/\pi.
\]
For $\mathbb{F}_p$-algebras, the construction  ${L W} \Omega_{(-)}$
satisfies a K\"unneth formula, so we get
$$ 
{L W} \Omega_{\mathfrak{X}/\pi} \simeq L W \Omega_{\overline{\mathfrak{X}_0}} {\otimes}_{W(k)}
\acrys.
$$
Note that we do not need to $p$-complete again,
since $\mathfrak{X}_0$ is smooth and proper. 
Taking Frobenius fixed points, we identify the top right term now rationally,
thanks to \eqref{QpofFpalgebra}. 

We can replace the bottom row of \eqref{secondfibsq} with its completion with
respect to the (rationalized) Hodge filtration. 
Recall  that $p$-adic derived de Rham cohomology together
with its Hodge filtration (so as a filtered $\mathbb{E}_\infty$-algebra)
satisfies a K\"unneth formula.
Therefore, we have
$$L \Omega_{\mathfrak{X}/\mathcal{O}_K} \simeq L
\Omega_{\mathfrak{X}_0/\mathcal{O}_K} {\otimes_{\mathcal{O}_K}} L
\Omega_{\mathcal{O}_C/\mathcal{O}_K},$$ 
since
the filtration on 
$L
\Omega_{\mathfrak{X}_0/\mathcal{O}_K}$ is finite, and it is by perfect
$\mathcal{O}_K$-modules. 
It follows that the Hodge completion of the rationalization of $L
\Omega_{\mathfrak{X}/\mathcal{O}_K}$ is equivalent,
in the filtered derived $\infty$-category of $K$,
to
$$ \Omega_{X_0/K} \otimes_{K}  \bdr^+,$$ 
where we use 
\Cref{periodrings} for the identification with $\bdr^+$. 
\end{proof}

\appendix

\newcommand{\ft}{\mathrm{afp}}
\appendix

\section{Twisted Tate diagonals} 
\label{app:gradedcyc}
In this section we investigate under which conditions Hochschild homology in a general symmetric monoidal $\infty$-category admits a (twisted) cyclotomic structure.
The main result is Corollary \ref{cor_twisted} and the fact that it applies to graded and filtered $\THH$ as recorded in Examples \ref{Exgraded} and \ref{Exfiltered}. \\

As usual, we fix a prime $p$.
Let $\Cscr$ be a presentably symmetric monoidal $\infty$-category and $L\colon \Cscr \to \Cscr$ be a symmetric monoidal, left adjoint functor. 

\begin{definition}\label{TateDiagonal}
An \emph{$L$-twisted diagonal} on $\Cscr$ is a symmetric monoidal natural transformation 
\[
\Delta\colon L(C) \to (C \otimes \dots \otimes C)^{hC_p} 
\]
of lax symmetric monoidal functors $\Cscr \to \Cscr$. Assume $\Cscr$ is
    additionally stable;\footnote{In fact,  semiadditive (so that the Tate
    construction is defined) suffices.} then an $L$-twisted \emph{Tate diagonal} is a
    symmetric monoidal natural transformation
\[
\Delta\colon L(C) \to T_p(C) := (C \otimes \dots \otimes C)^{tC_p}  \ .
\] 
\end{definition}

\begin{example}
\begin{enumerate}
\item
The $\infty$-category of spaces admits a (unique) $\mathrm{id}$-twisted diagonal
 and the $\infty$-category of spectra admits a (unique) $\mathrm{id}$-twisted
 Tate diagonal, see \cite[Section III.1]{NS18}. 
\item More generally, if $\Cscr$ admits the cartesian symmetric monoidal structure then it admits a canonical $\mathrm{id}$-twisted diagonal induced by the actual diagonal. 
 \end{enumerate}
\end{example}

\begin{example}
\label{twistedtatemodule}
 Suppose 	$R$ is an $\mathbb{E}_\infty$-ring and $\Cscr = \mathrm{Mod}_R$,
 considered as a symmetric monoidal $\infty$-category with the $R$-linear
 tensor product. Then every
 left adjoint, symmetric monoidal functor $L$ is given by an induction along an
 $\mathbb{E}_\infty$-map $l\colon R \to R$ and we will prove below that the datum of an $L$-twisted Tate diagonal on $\Cscr$ is 
equivalent to an $\mathbb{E}_\infty$-homotopy between the composition
 \[
 R \xrightarrow{l} R \xto{\triv} R^{tC_p}
 \]
 and the Tate-valued Frobenius $\varphi\colon R \to R^{tC_p}$ of $R$,
 \cite[IV.1]{NS18}. We shall refer to an $\mathbb{E}_\infty$-ring $R$ with such a datum as a cyclotomic base. An example is $R= \mathbb{S}[z]$, see Example \ref{polynomial_tate} below. 
\end{example}
\begin{proof}
 To see that the datum of a twisted Tate diagonal is equivalent to such an
 equivalence, we first note that a symmetric monoidal natural transformation $L
 \to T_p$ of functors $\Mod_R \to \Mod_R$ is determined by its restriction to
 the perfect modules $\Mod_R^\omega \subset \Mod_R$ since $L$ preserves filtered colimits. Since $T_p$ is lax symmetric monoidal we get a factorization
 \[
\Mod_R \xto{T_p} \Mod_{T_p(R)} \xto{\mathrm{res}_{\triv}} \Mod_R 
 \]
as lax symmetric monoidal functors. Upon restriction to $\Mod_R^\omega$ the
first functor is given by base-change along the  Tate-valued Frobenius $\varphi\colon R \to R^{tC_p}$ such that we get a factorization $T_p\mid_{\Mod_R^\omega} = \mathrm{res}_{\triv} \circ \mathrm{ind}_{\varphi}$. Now a symmetric monoidal transformation 
\[
 L\mid_{\Mod_R^\omega} = \mathrm{ind}_l \to \mathrm{res}_{\triv} \circ \mathrm{ind}_{\varphi}
\]
is by adjunction equivalent to a natural transformation
\[
\mathrm{ind}_{\triv \circ l} = \mathrm{ind}_{\triv} \mathrm{ind}_{l} \to 
\mathrm{ind}_{\varphi} \ .
\]
But every object in $\Mod_R^\omega$ is dualizable and both functors $\mathrm{ind}_{\triv \circ l}$ and $\mathrm{ind}_{\varphi}$ are symmetric monoidal. Thus every such symmetric monoidal transformation is necessarily an equivalence and thus induced by an equivalence of maps of $\mathbb{E}_\infty$-rings $R \to R^{tC_p}$. This shows the claim. 
\end{proof}

\begin{remark}
Note that a general  symmetric monoidal, stable $\infty$-category $\Cscr$ does
    not admit $L$-twisted Tate-diagonals for arbitrary $L$. For example if we
	 consider $\Cscr = \mathcal{D}(\mathbb{Z}) \simeq \mathrm{Mod}_{H\mathbb{Z}}$
	 then by \Cref{twistedtatemodule} above a twisted Tate diagonal would be the same as a factorization of the Tate-valued Frobenius
\[
H\mathbb{Z} \to (H\mathbb{Z})^{tC_p} 
\]
through the $\mathrm{triv}$-map $H\mathbb{Z} \to H\mathbb{Z}^{tC_p}$. These two maps however differ by Steenrod operations as shown in \cite[IV.1]{NS18}. But any twist would be the identity. 
\end{remark}

In the following, we use the notation $\HH A$
to denote the Hochschild homology object of an algebra object $A \in
\mathrm{Alg}(\Cscr)$
internal to $\Cscr$; 
e.g., if $\Cscr = \sp$, this recovers $\THH$. 

\begin{proposition}\label{prop_twisted_cyclotomic}
Assume that $\Cscr$ is equipped with an $L$-twisted diagonal resp.~Tate
    diagonal. Then we get for each algebra object $A \in \mathrm{Alg}(\Cscr)$
    an induced $S^1$-equivariant map
\[
L (\HH A) \to (\HH A)^{hC_p} \qquad \text{resp.} \qquad L( \HH A) \to (\HH A)^{tC_p}
\]
    which is functorial and symmetric monoidal in $A$.
\end{proposition}
\begin{proof}
We closely follow the construction of the cyclotomic structure on $\THH$ given
in \cite[Sec.~III.2]{NS18}. We will mostly indicate the necessary changes and thus recommend that the reader take a look at the construction there first. We treat the case of the Tate diagonal which is the only case that we will need in this paper, but the case of the diagonal works exactly the same. 

We first recall that $\HH A$ is the geometric realization of the cyclic object in $\Cscr$ informally written as
\[\xymatrix{
 \cdots \ar[r]<1.5pt>\ar[r]<-1.5pt>\ar[r]<4.5pt>\ar[r]<-4.5pt> & A\otimes A\otimes A\ar@(ul,ur)^{C_3} \ar[r]<3pt>\ar[r]\ar[r]<-3pt>  & A\otimes A\ar@(ul,ur)^{C_2} \ar[r]<1.5pt>\ar[r]<-1.5pt> & A\ .
}\]
Thus $L( \HH A)$ is the geometric realization of the cyclic object 
\[\xymatrix{
 \cdots \ar[r]<1.5pt>\ar[r]<-1.5pt>\ar[r]<4.5pt>\ar[r]<-4.5pt> & L(A\otimes A\otimes A)\ar@(ul,ur)^{C_3} \ar[r]<3pt>\ar[r]\ar[r]<-3pt>  &L( A\otimes A)\ar@(ul,ur)^{C_2} \ar[r]<1.5pt>\ar[r]<-1.5pt> & L(A)\ .
}\]
For a given $L$-twisted Tate diagonal 
\[
\Delta\colon L(C) \to (C \otimes ... \otimes C)^{tC_p}  = T_p(C)
\]
we want to construct a natural map of cyclic objects
\[\xymatrix{
\cdots \ar[r]<1.2ex> \ar[r]<-1.2ex>\ar[r]<0.4ex> \ar[r]<-0.4ex> & L(A^{\otimes 3})\ar[d]^{\Delta} \ar@(ul,ur)^{C_3}\ar[r]<0.8ex>\ar[r]<0ex>\ar[r]<-0.8ex> \ar[r]& L(A^{\otimes 2}) \ar[d]^{\Delta}\ar@(ul,ur)^{C_2}
\ar[r]<0.4ex> \ar[r]<-0.4ex>& L A \ar[d]^{\Delta} \\
\cdots \ar[r]<1.2ex> \ar[r]<-1.2ex>\ar[r]<0.4ex> \ar[r]<-0.4ex> & \big(A^{\otimes 3p}\big)^{tC_p} \ar@(dl,dr)_{C_3}\ar[r]<0.8ex>\ar[r]<0ex>\ar[r]<-0.8ex> \ar[r]& \big(A^{\otimes 2p}\big)^{tC_p}\ar@(dl,dr)_{C_2}
\ar[r]<0.4ex> \ar[r]<-0.4ex>& \big(A^{\otimes p}\big)^{tC_p}
},\]
and obtain the desired map $L(\HH A) \to (\HH A)^{tC_p}$ as the geometric realization of this map of cyclic objects followed by the canonical interchange map from the realization of the Tate constructions to the Tate construction of the realization. 

In order to construct such a natural transformation of cyclic objects, we
    proceed as in \cite{NS18}: we eventually need to show that we can extend the symmetric monoidal natural transformation $\Delta\colon L \to T_p$ of functors $\Cscr \to \Cscr$ to a $BC_p$-equivariant symmetric monoidal natural transformation of functors  
 from the functor
\[
\tilde{L}\colon N(\mathrm{Free}_{C_p})\times_{N(\mathrm{Fin})} \Cscr^\otimes_\act\xto{\mathrm{pr}} \Cscr^\otimes_\act\xto{\otimes}\Cscr \xto{L} \Cscr
\]
given by
\[
(S,(X_{\overline{s}\in \overline{S}=S/C_p}))\mapsto L\Big(\bigotimes_{\overline{s}\in \overline{S}} X_{\overline{s}}\Big)
\]
to the functor
\[
\tilde{T}_p\colon N(\mathrm{Free}_{C_p})\times_{N(\mathrm{Fin})} \Sp^\otimes_\act\to (\Cscr^\otimes_\act)^{BC_p}\xto{\otimes} \Cscr^{BC_p}\xto{-^{tC_p}} \Cscr
\]
given by
\[
(S,(X_{\overline{s}\in \overline{S}=S/C_p}))\mapsto (\bigotimes_{s\in S} X_{\overline{s}})^{tC_p}\ .
\]
Here $\mathrm{Free}_{C_p}$ is the category of finite free $C_p$-sets equipped
    with the cocartesian symmetric monoidal structure. The group object
    $BC_p$-acts on this category in the obvious way and acts trivially on
    $\Cscr$. For a precise construction of these functors we refer to
    \cite[Section III.3]{NS18}, specifically Proposition III.3.6 and the
    construction around that.

The inclusion
\[
\Fun_{\otimes}\left(N(\mathrm{Free}_{C_p})\times_{N(\mathrm{Fin})} \Cscr^\otimes_\act, \Cscr\right) \subseteq \Fun_{\mathrm{lax}}\left(N(\mathrm{Free}_{C_p})\times_{N(\mathrm{Fin})} \Cscr^\otimes_\act, \Cscr\right)
\]
admits a right adjoint by 
Lemma III.3.3 resp.~Remark III.3.5 in \cite{NS18}. Using the same construction and argument as in the proof of 
\cite[Lemma III.3]{NS18} we see that the $\infty$-category $\Fun_{\otimes}\left(N(\mathrm{Free}_{C_p})\times_{N(\mathrm{Fin})} \Cscr^\otimes_\act, \Cscr\right)$ is equivalent to the $\infty$-category $\Fun(N\mathrm{Tor}_{C_p}, \Fun_{\mathrm{lax}}(\Cscr, \Cscr))$ where $\mathrm{Tor}_{C_p}$ denotes the category of $C_p$-torsors. 
Under this equivalence the right adjoint to the inclusion is given by restricting a functor 
in $ \Fun_{\mathrm{lax}}\left(N(\mathrm{Free}_{C_p})\times_{N(\mathrm{Fin})} \Cscr^\otimes_\act, \Cscr\right)$ to $N\mathrm{Tor}_{C_p} \times \Cscr \subseteq N(\mathrm{Free}_{C_p})\times_{N(\mathrm{Fin})} \Cscr^\otimes_\act$ and forming the adjunct. 

Now the functor $\tilde L$ is symmetric monoidal rather than lax symmetric
monoidal. Thus to construct a map from $\tilde L$ to $\tilde{T}_p$ is by
adjunction equivalent to constructing a transformation in
$\Fun(N\mathrm{Tor}_{C_p}, \Fun_{\mathrm{lax}}(\Cscr, \Cscr))$ between the
respective restrictions. Moreover $BC_p$-acts on all those categories, i.e., to construct a $BC_p$-equivariant transformation between $\tilde L$ and $\tilde{T}_p$ is equivalent to construct a transformation in 
\[
\Fun^{BC_p}(N\mathrm{Tor}_{C_p}, \Fun_{\mathrm{lax}}(\Cscr, \Cscr)).
\]
Now the category $\mathrm{Tor}_{C_p}$ is in fact equivalent to $BC_p$. Since the $BC_p$-action on $\Fun_{\mathrm{lax}}(\Cscr, \Cscr)$ is trivial it follows that the above $\infty$-category of $BC_p$-equivariant functors is equivalent to  $\Fun_{\mathrm{lax}}(\Cscr, \Cscr)$. 

Taking everything together we see that there is a unique symmetric monoidal transformation $\tilde L \to \tilde{T}_p$ extending the transformation $\Delta\colon L \to T_p$. 
Together with the constructions above this finishes the proof. 
\end{proof}

We shall refer to the map $L( \HH A) \to (\HH A)^{tC_p}$ as a twisted cyclotomic structure on $\HH A$. Thus the last result shows that for $\infty$-categories with a twisted Tate diagonal we find that Hochschild homology admits a twisted cyclotomic structure.
\begin{lemma}\label{lem_stabilization}
For a given $L$-twisted diagonal on $\Cscr$, the stabilization $\Sp(\Cscr)$ admits a canonical induced $\Sp(L)$-twisted Tate diagonal.
\end{lemma}
\begin{proof}
We would like to construct a symmetric monoidal natural transformation
\[
\Sp(L)(C) \to ( C \otimes \dots  \otimes C)^{tC_p} = T_p(C) \ .
\]
Such a transformation is by adjunction  the same as a symmetric monoidal transformation
\[
\mathrm{id} \to R' T_p
\]
where $R'\colon \Sp(\Cscr) \to \Sp(\Cscr)$ is the right adjoint to $\Sp(L)$.
We now use that the functor $\Omega^\infty$ induces an equivalence
\[
\Fun_{\mathrm{lax}}^{\mathrm{Ex}}(\Sp(\Cscr), \Sp(\Cscr)) \to \Fun_{\mathrm{lax}}^{\mathrm{Ex}}(\Sp(\Cscr), \Cscr) 
\]
by \cite{Nik}. It follows that it suffices to  construct a symmetric monoidal transformation
\[
\Omega^\infty \to \Omega^\infty R' T_p
\]
of functors $\Sp(\Cscr) \to \Cscr$. We denote by $R\colon \Cscr \to \Cscr$ the right adjoint to the functor $L\colon \Cscr \to \Cscr$. Then we have an equivalence $\Omega^\infty R' \simeq R \Omega^\infty$ of lax symmetric monoidal functors which follows from the fact that the left adjoint diagram
\[
\xymatrix{
\Cscr \ar[r]^L\ar[d]^{\Sigma^\infty} & \Cscr\ar[d]^{\Sigma^\infty} \\
\Sp(\Cscr) \ar[r]^{\Sp(L)} & \Sp(\Cscr) \ .
}
\]
commutes (up to symmetric monoidal equivalence). As a result we need to construct a symmetric monoidal natural transformation
\begin{equation}\label{trafo_lax}
\Omega^\infty \to R \Omega^\infty T_p \ .
\end{equation}
Now we use that we have  canonical symmetric monoidal transformations
\[
\gamma\colon (\Omega^\infty C \otimes \dots \otimes \Omega^\infty C)^{hC_p}
\to 
\Omega^\infty \left( (C \otimes \dots \otimes C)^{hC_p} \right)
\to 
\Omega^\infty \left( (C \otimes \dots \otimes C)^{tC_p} \right)
\]
where the first one is induced by the lax symmetric monoidal structure of $\Omega^\infty$ together with the fact that it commutes with limits and the second by the canonical map from homotopy fixed points to the Tate construction.

Now we use the unstable diagonal on $\Cscr$ to get as the adjoint a symmetric monoidal natural transformation
\[
\Omega^\infty C \to R (\Omega^\infty C \otimes \dots \otimes \Omega^\infty C)^{hC_p}
\]
and compose it with the map $R(\gamma)$ above to get a symmetric monoidal
    natural transformation as in \eqref{trafo_lax}.
\end{proof}

For every symmetric monoidal $\infty$-category $I$ we consider the symmetric monoidal functor $l_p\colon I \to I$ given by sending $i$ to $i^{\otimes p}$. We let 
\[
L_p\colon  \Fun(I, \mathcal{S}) \to \Fun(I,\mathcal{S})
\]
be left Kan extension along $l_p$. We equip the category $\Fun(I, \mathcal{S})$ with the
Day convolution symmetric monoidal structure. Then the left Kan extension $L_p$ becomes symmetric monoidal.
\begin{lemma}\label{lem_diagonal}
Assume that the $\infty$-category $I$ has the following property:
for every pair of objects $i,j \in I$ we have that the canonical forgetful map
\begin{equation}\label{condition_sym}
\mathrm{Map}_I(i^{\otimes p}, j)^{hC_p} \to \mathrm{Map}_I(i^{\otimes p}, j)
\end{equation}
is an equivalence of spaces.\footnote{Note that an equivalent way of stating this condition is to say that the homotopy orbits $(i^{\otimes p})_{hC_p}$ exist in $I$ and the map $i^{\otimes p} \to (i^{\otimes p})_{hC_p}$ is an equivalence. }
Then the inverse of the map \eqref{condition_sym} induces a canonical $L_p$-twisted diagonal on $\Fun(I, \mathcal{S})$. 
\end{lemma}
\begin{proof}
We consider the symmetric monoidal (co)Yoneda embedding 
\[
I^\op \to \Fun(I, \mathcal{S}) \ .
\]
Then symmetric monoidal transformations  
\[
L_p(C) \to (C \otimes \dots \otimes C)^{hC_p}
\]
as functors $\Fun(I, \mathcal{S}) \to  \Fun(I, \mathcal{S})$
are the same as symmetric monoidal transformations between the restrictions of the functors along the Yoneda embedding. The restricted functors $I^\op \to \Fun(I, \mathcal{S})$ are given by the lax symmetric monoidal assignments
\[
    i \mapsto (j \mapsto \Map_I(i^{\otimes p}, j)) \qquad \text{and} \qquad
    i \mapsto (j \mapsto \Map_I(i^{\otimes p}, j)^{hC_p}) \ .
\]
The canonical map $\Map_I(i^{\otimes p}, j)^{hC_p} \to \Map_I(i^{\otimes p},
    j)$ is a symmetric monoidal natural transformation. By assumption it is an equivalence so that the inverse induces the required transformation.
\end{proof}

\begin{remark}
For a general symmetric monoidal $\infty$-category $I$ the category $\Fun(I,
\mathcal{S})$ does not admit an $L_p$-twisted diagonal. As an example consider
any  cocartesian symmetric monoidal $\infty$-category $I$. Then the Day
convolution structure on $\Fun(I, \Sp)$ is cartesian.\footnote{This follows from
the fact that generally Day convolution for a cocartesian source is given by the
pointwise tensor product, which in our case happens to agree with the cartesian
product.} Thus an $L_p$-twisted diagonal would amount to a natural symmetric monoidal transformation
\[
F^{\times p} \to (F^{\times p})^{hC_p}
\]
which does not exist. 

But note that this category admits an $\mathrm{id}$-twisted diagonal. This raises the question if for every symmetric monoidal $\infty$-category $I$ there is a twist on $\Fun(I, \mathcal{S})$ and a twisted diagonal. The answer to this question is also `no' in general but we will not go into the intricacies of concrete counterexamples here.
 \end{remark}

\begin{corollary}\label{cor_twisted}
If $I$ is a symmetric monoidal $\infty$-category satisfying the condition of Lemma \ref{lem_diagonal} then we have for every algebra $A$ in $\Fun(I, \Sp)$ a twisted cyclotomic structure on $\HH A$, i.e.  an $S^1$-equivariant map
\[
L_p(\HH A) \to (\HH A)^{tC_p} \ .
\]
This map is natural and symmetric monoidal in $A$.
\end{corollary}
\begin{proof}
Combine Proposition \ref{prop_twisted_cyclotomic} with Lemma \ref{lem_stabilization} and Lemma \ref{lem_diagonal}.
\end{proof}

\begin{example}\label{Exgraded}
We consider the category $I = \mathbb{Z}_{\geq 0}^{\mathrm{ds}}$. Then 
$\Fun(I, \mathrm{Sp})$ is the $\infty$-category of graded spectra. The category
$I$ obviously satisfies the condition of Lemma \ref{lem_diagonal}. Thus we get
that for a graded ring $R_\bullet$, graded $\THH$ admits an $L_p$-twisted cyclotomic structure or equivalently a sequence of $S^1$-equivariant maps 
\[
\THH(R)_i \to \THH(R)_{pi}^{tC_p} \ .
\]
The same logic applies to spectra graded over any discrete monoid in place of $\mathbb{Z}_{\geq 0}^{\mathrm{ds}}$.
\end{example}

\begin{example}\label{Exfiltered}
Consider the $\infty$-category $I = \mathbb{Z}_{\geq 0}^{\op}$ associated to the poset of positive integers. Then this also satisfies the condition of Lemma \ref{lem_diagonal}. The category of functors $\Fun(I, \Sp)$ is given by filtered spectra and thus filtered $\THH$ of a filtered ring spectrum $R$ admits a filtered cyclotomic structure, i.e. $S^1$-equivariant maps
\[
\mathrm{Fil}^{\geq i} \THH(R) \to  (\mathrm{Fil}^{\geq pi}
    \THH(R))^{tC_p}.
\]
\end{example}

\begin{example}\label{polynomial_tate}
Consider the category $I = B \mathbb{Z}_{\geq 0}$. This category also obviously satisfies the condition of Lemma \ref{lem_diagonal}. Thus the category
\[
\Fun(I, \Sp) \simeq \mathrm{Mod}_{\mathbb{S}[z]}
\]
admits a twisted Tate diagonal and thus relative $\THH$ admits a (twisted)
cyclotomic structure, as is used in \cite[Sec.~11]{BMS2}. The twist $L_p$ corresponds to the map $l\colon \mathbb{S}[z] \to \mathbb{S}[z]$ sending $z$ to $z^p$. 
\end{example}

We want to end this section by remarking some functorialities of the twisted cyclotomic structures.
\begin{definition}
A \emph{symmetric monoidal category with (Tate) diagonals} consists of a triple $(\Cscr, L, \Delta)$ as in Definition \ref{TateDiagonal}. A map of symmetric monoidal categories with (Tate) diagonals 
\[
(\Cscr, L, \Delta) \to (\Cscr', L', \Delta')
\]
is given by a left adjoint symmetric monoidal functor $F\colon \Cscr \to \Cscr'$ together with a symmetric monoidal equivalence $L' \circ F \simeq F \circ L$ and a natural symmetric monoidal equivalence between the two maps
\[
L'(FX) \to (FX\otimes \dots \otimes FX)^{tC_p} 
\]
induced from $\Delta$ and $\Delta'$ (both sides considered as lax symmetric monoidal functors $\Cscr \to \Cscr'$).
\end{definition}

Form the construction of the twisted cyclotomic structure in Proposition \ref{prop_twisted_cyclotomic} we see immediately that for such a map of symmetric monoidal $\infty$-categories with Tate diagonals we get an equivalence of twisted cyclotomic objects
\[
F( \HH A ) \simeq \HH(FA)
\]
for every algebra $A$ in $\Cscr$. Here the first object $F(\HH A)$ is twisted cyclotomic by the composition
\[
LF( \HH A) \xto{\simeq} F L ( \HH A) \xto{F \varphi} F( \HH A^{tC_p}) \to F( \HH A)^{tC_p} \ .
\]
We also have a relative analogue of Lemma \ref{lem_stabilization}: every map of
symmetric monoidal $\infty$-categories with diagonals induces upon
stabilization a map of symmetric monoidal $\infty$-categories with Tate
diagonals. This is straightforward to prove. Finally there is also an analogue
of Lemma \ref{lem_diagonal} which we  will state and prove now.

\begin{lemma}
Assume that $f\colon I \to I'$ is a symmetric monoidal functor such that $I$ and $I'$ satisfy the condition of Lemma \ref{lem_diagonal}. Then left Kan extension along $f$ induces a map of symmetric monoidal $\infty$-categories with diagonals
\[
(\Fun(I, \mathcal{S}), L_p, \Delta) \to (\Fun(I', \mathcal{S}), L'_p, \Delta') \ .
\]
where $L_p, L_p', \Delta$ and $\Delta'$ are as in Lemma \ref{lem_diagonal}.
\end{lemma}
\begin{proof}
We have a commutative square
\[
\xymatrix{
I \ar[r]^f\ar[d]^{l_p}  & I'\ar[d]^{l_p'} \\
I \ar[r]^f & I'
}
\]
for the functors $l_p(i) = i^{\otimes p}$ and $l_p'(j) = j^{\otimes p}$. Thus we get an induced square of the left Kan extensions
\[
\xymatrix{
\Fun(I, \mathcal{S}) \ar[r]^F\ar[d]^{L_p}  & \Fun(I', \mathcal{S})\ar[d]^{L_p'} \\
\Fun(I, \mathcal{S}) \ar[r]^F &\Fun(I', \mathcal{S}) \ .
}
\]
This provides the first part of the datum of a map of symmetric monoidal $\infty$-categories with Tate diagonals. We now also have to provide an equivalence of two different natural transformations between two functors
\[
\Fun(I, \mathcal{S}) \to \Fun(I', \mathcal{S}) \ .
\]
Such a transformation is determined by its restriction to $I^\op \subseteq
\Fun(I, \mathcal{S})$ and there the functors are given by
\[
i \mapsto \left(j \mapsto \mathrm{Map}_{I'}(f(i)^{\otimes p}, j)\right)
\]
and 
\[
i \mapsto \left(j \mapsto \mathrm{Map}_{I'}(f(i)^{\otimes p}, j)^{hC_p}\right) \ .
\]
Unravelling the constructions we see that both of the two transformations are given by the inverse of the canonical forgetful map 
\[
 \mathrm{Map}_{I'}(f(i)^{\otimes p}, j)^{hC_p} \to \mathrm{Map}_{I'}(f(i)^{\otimes p}, j) 
\]
and thus are canonically equivalent.
\end{proof}

From these statements together we can deduce the following corollary:

\begin{corollary}\label{cor_leftKan}
Assume that $f\colon I \to I'$ is a symmetric monoidal functor such that $I$ and $I'$ satisfy the condition of Lemma \ref{lem_diagonal}. Then for every algebra $A \in \Fun(I, \mathrm{Sp})$ we have an equivalence of $L'_p$ twisted cyclotomic objects 
\[
F( \HH A ) \simeq \HH(FA)
\]
where $F$ is left Kan extension along $f$.
\end{corollary}

\begin{example}
For a graded ring spectrum $R_\bullet$ we have that the direct sum 
\[
\bigoplus_{i} \THH(R_\bullet)_i 
\]
is equivalent to $\THH(\bigoplus_i R)$ as cyclotomic spectra. Similarly, for a filtered ring spectrum $R$ we have that the filtered cyclotomic structure refines the cyclotomic structure on $\THH(R)$. 
\end{example}
\begin{example}
We finally note that one can also look at the functor 
\[
\mathrm{ev}_0\colon \Fun(\mathbb{Z}_{\geq 0}^{\mathrm{ds}}, \Sp) \to \Sp
\]
given by restriction to the $0$-th component. We claim that this also refines
    to a map of symmetric monoidal $\infty$-categories with Tate diagonals.
    This can be seen by verifying the corresponding unstable statement which is
    straightforward  using an argument similar to the one in the proof of
    Corollary \ref{cor_leftKan}. This then shows that the cyclotomic structure
    on the $0$-th graded component $\THH(R_\bullet)_0$ agrees with the one on
    $\THH(R_0)$ for every graded ring spectrum $R_\bullet$.
\end{example}
\section{Categorical lemmas} 
 \label{app:Kan}
\begin{construction}[Left Kan extensions] 
\label{LKEcons}
Let $R$ be a ring, and 
let $\mathrm{Poly}_R$  be the category of finitely generated polynomial $R$-algebras. Given a presentable $\infty$-category
$\mathcal{C}$ and an accessible functor $f\colon \mathrm{Poly}_R \to
\mathcal{C}$, we can left Kan
extend to obtain a functor $Lf\colon \mathrm{SCR}_R \to \mathcal{C}$ which commutes
with geometric realizations, for
$\mathrm{SCR}_R$ the $\infty$-category of simplicial commutative $R$-algebras. 
Compare \cite[Sec. 5.5.8]{HTT} and \cite[Sec. 4.2]{DAGVIII}. 
\end{construction} 

Let $(\mathcal{L}, \mathcal{R}) \colon \mathcal{C} \rightleftarrows \mathcal{D}$ be an adjunction 
of $\infty$-categories. Then 
for any $\infty$-category $\mathcal{E}$, we obtain an adjunction
\begin{equation} \label{pullbackadj} 
(\mathcal{R}^*, \mathcal{L}^*) = (f \mapsto f \circ \mathcal{R}, f' \mapsto f \circ \mathcal{L})\colon 
\fun( \mathcal{C}, \mathcal{E}) \rightleftarrows \fun(\mathcal{D},
\mathcal{E})  . \end{equation}

\begin{remark} 
\label{adjequiv}
Let $f_1, f_2\colon \mathcal{D} \to \mathcal{E}$ be functors. 
Suppose that for any $x \in \mathcal{D}$, the natural map $f_1( \mathcal{L}
\mathcal{R} x ) \to f_1(x)$ is an equivalence. 
Then 
we find
\begin{equation}  \hom_{\fun(\mathcal{D}, \mathcal{E})}(f_1, f_2) 
\simeq \hom_{\fun(\mathcal{C}, \mathcal{E})}(f_1 \circ \mathcal{L}, f_2 \circ
\mathcal{L}). 
\end{equation}
This follows from the adjunction \eqref{pullbackadj}. 
\end{remark}

Now we specialize to the case where $\mathcal{C} = \SCR$ is the
$\infty$-category of simplicial commutative rings and $\mathcal{D} =
\SCR_{\mathbb{F}_p}$ is the $\infty$-category of simplicial commutative
$\mathbb{F}_p$-algebras. 
We have an adjunction $(\mathcal{L}, \mathcal{R})\colon \SCR \rightleftarrows \SCR_{\mathbb{F}_p}$, where the
left adjoint is $R \mapsto R \otimes^L_{\mathbb{Z}} \mathbb{F}_p$ and the right adjoint is
simply the forgetful functor. 

For any $R \in \SCR_{\mathbb{F}_p}$, we have a canonical endomorphism $\varphi\colon
R \to R$, the Frobenius. 

\begin{lemma} 
\label{froblemmaSCR}
Let $R \in \SCR_{\mathbb{F}_p}$. 
There is a natural map $f\colon R \to R \otimes^{L}_{\mathbb{Z}} \mathbb{F}_p$ in 
$\SCR_{\mathbb{F}_p}$
such
that the composites 
$R \stackrel{f}{\to} R \otimes^{L}_{\mathbb{Z}} \mathbb{F}_p \to R$ and
$R \otimes^{L}_{\mathbb{Z}} \mathbb{F}_p \to R \to R \otimes^{L}_{\mathbb{Z}}
\mathbb{F}_p$ are the respective Frobenius endomorphisms. 
\end{lemma} 
\begin{proof} 
It suffices to assume that $R$ is discrete (even a finitely generated
polynomial ring) via left Kan extension. 
In this case, $R \otimes^L_{\mathbb{Z}} \mathbb{F}_p$ is concentrated in
homological degrees zero  and one (with $\pi_0 = R$ itself), and one knows that the Frobenius endomorphism 
annihilates $\pi_1$, cf.~\cite[Prop.~11.6]{BSproj}. Thus, the Frobenius map
$R\otimes_\ZZ^L\FF_p\rightarrow R\otimes_\ZZ^L\FF_p$ factors canonically through the truncation map
$R\otimes_\ZZ^L\FF_p\rightarrow\pi_0(R\otimes_\ZZ^L\FF_p)\iso R$.
This gives the map $f$ as desired.
\end{proof} 

\begin{corollary} 
\label{adjunctionppower}
Let $F_1, F_2\colon \SCR_{\mathbb{F}_p} \to D(\mathbb{Z})$ be two functors. 
Suppose $F_1$ has the property that the natural map $F_1(R) \to F_1(R)$ given
by Frobenius is multiplication by $p^i$. 
Then for any natural transformation 
$u\colon F_1( - \otimes^L_{\mathbb{Z}} \mathbb{F}_p) \to F_2(-
\otimes^L_{\mathbb{Z}} \mathbb{F}_p)$ of functors $\SCR_{\mathbb{Z}} \to D(\mathbb{Z})$, we have that $p^i u$ arises from a
natural transformation 
$F_1 \to F_2$.
In fact, we have 
$$\hom_{\fun(\SCR_\ZZ, D(\mathbb{Z}))}( F_1( - \otimes_{\mathbb{Z}_p}^L
\mathbb{F}_p), F_2( - \otimes_{\mathbb{Z}_p}^L \mathbb{F}_p)) [1/p]
\simeq 
\hom_{\fun(\SCR_{\mathbb{F}_p}, D(\mathbb{Z}))}( F_1, F_2) [1/p].$$
\end{corollary} 
\begin{proof} 
This follows from \eqref{pullbackadj} in the case 
of the adjunction $(\mathcal{L}, \mathcal{R}) \colon \SCR_\ZZ \rightleftarrows \SCR_{\mathbb{F}_p}$. 
By construction, we are given a map 
$\mathcal{L}^* F_1 \to \mathcal{L}^* F_2$ of functors $\SCR_{\mathbb{F}_p}
\to D(\mathbb{Z})$, or equivalently by adjointness a map 
$\mathcal{R}^*\mathcal{L}^* F_1 \to F_2$ of functors $\SCR_\ZZ \to D(\mathbb{Z})$. 
Now we have  a natural map $F_1 \to \mathcal{R}^* \mathcal{L}^* F_1$ 
given by the natural map 
$f \colon R \to R \otimes^L _{\mathbb{Z}_p} \mathbb{F}_p$ of \Cref{froblemmaSCR};
it has the property that the composites in either order with the adjunction
map $\mathcal{R}^* \mathcal{L}^* F_1 \to F_1$ are given by multiplication by
$p^i$. 
The composition
$F_1\to \mathcal{R}^* \mathcal{L}^* F_1 \to F_2$ defines the desired map $F_1 \to F_2$. 
This argument also proves the displayed equation. 
\end{proof} 
\newcommand{\fsm}{\mathrm{FSmooth}}
\newcommand{\sm}{\mathrm{Smooth}}

Let $K$ be a complete discretely valued field with ring of integers
$\mathcal{O}_K \subset K$ and residue 
field $k$; let $\pi \in \mathcal{O}_K$ be a uniformizer. 
Let $\fsm_{\mathcal{O}_K}$ denote the category of topologically finitely generated, formally smooth
$\mathcal{O}_K$-algebras and let $\sm_k$ denote the category of smooth
$k$-algebras. 

We now give a similar result for functors defined on a restricted class of
simplicial commutative $k$-algebras.
For the next result, we will argue similarly, but with a smaller set of $\infty$-categories. 
For these finiteness conditions, see
\cite[Sec.~7.2]{HA} (in the slightly more complicated $\mathbb{E}_\infty$-case). 
\begin{definition} 
\label{almostfp}
\begin{enumerate}
    \item[{\rm (1)}]
Let $\SCR^{\ft}_k$ denote the $\infty$-category of simplicial commutative
$k$-algebras 
$R$
which are almost finitely presented: equivalently, 
$\pi_0(R)$ is finitely generated as a $k$-algebra and each $\pi_i(R)$ is a
finitely generated $\pi_0(R)$-module. 
Equivalently, $R$ belongs to $\SCR^{\ft}_k$ if and only if $R$ can be written as
the geometric realization of a simplicial diagram of finitely generated
polynomial $k$-algebras. 
\item[{\rm (2)}]
Similarly, we define $\widehat{\SCR}^{\ft}_{\mathcal{O}_K}$ to be the $\infty$-category of 
$\pi$-complete 
simplicial commutative
$\mathcal{O}_K$-algebras $R$ such that 
$\pi_0(R)$ is topologically finitely generated 
over $\mathcal{O}_K$ (i.e., a quotient of a $\pi$-completed polynomial ring)
and each  $\pi_i(R)$ is finitely generated over $R$. 
Equivalently, $R$ belongs to $\widehat{\SCR}^{\ft}_{\mathcal{O}_K}$ if and only if $R$ can be written as
the geometric realization of a simplicial diagram of $\pi$-completed finitely generated
polynomial $\mathcal{O}_K$-algebras. 
Yet another characterization is that $R$ should be almost finitely presented
over the $\pi$-completion of a finitely generated polynomial algebra over
$\mathcal{O}_K$ with a structure map that is surjective on $\pi_0$. 
\end{enumerate}
\end{definition} 
\begin{corollary} 
\label{TCKanextlemma}
Let $\mathcal{E}$ be an $\infty$-category admitting sifted colimits. 
Let $F_1, F_2\colon \SCR_k^{\mathrm{afp}}  \to \mathcal{E}$ be functors. If
\begin{enumerate}
    \item[{\rm (1)}] $F_1 $  commutes with geometric realizations and
    \item[{\rm (2)}]
	 $F_1(R) \simeq F_1( \pi_0 R)$ for $R \in \SCR_k^{\mathrm{afp}}$,
\end{enumerate}
then 
\[ \hom_{\fun(\sm_{k}, \mathcal{E})}(F_1, F_2) 
\simeq \hom_{\fun( \fsm_{\mathcal{O}_K}, \mathcal{E})}( F_1( -
\otimes_{\mathcal{O}_K} k),
F_2(- \otimes_{\mathcal{O}_K} k)).
\]
\end{corollary} 

\begin{proof} 
Since $F_1$ is left Kan extended from smooth (even finite type polynomial) $k$-algebras
as it commutes with geometric realizations, 
we have 
\[  \hom_{\fun(\sm_{k}, \mathcal{E})}(F_1, F_2) 
 \simeq \hom_{\fun( \SCR_k^{\ft}, \mathcal{E})}(F_1, F_2) . 
  \]
Similarly, 
\[  \hom_{\fun(\fsm_{\mathcal{O}_K}, \mathcal{E})}(F_1( -
\otimes_{\mathcal{O}_K} k), F_2( - \otimes_{\mathcal{O}_K} k)) 
 \simeq \hom_{\fun( \widehat{\SCR}_{\mathcal{O}_K}^{\ft},
 \mathcal{E})}(F_1(- \otimes_{\mathcal{O}_K} k), F_2(- \otimes_{\mathcal{O}_K} k)) , 
  \]
because $F_1(- \otimes_{\mathcal{O}_K} k)\colon
\widehat{\SCR}_{\mathcal{O}_K}^{\ft} \to \mathcal{E}$  is left Kan extended from
$\fsm_{\mathcal{O}_K}$. Now we have an adjunction 
$\widehat{\SCR}_{\mathcal{O}_K}^{\ft} \rightleftarrows \SCR_k^{\ft}$ given by
base-change and restriction of scalars. Thus, the result follows as in
\Cref{adjequiv}. 
\end{proof}

\bibliographystyle{amsalpha}
\bibliography{B}

\end{document}